\documentclass[10pt]{amsart}

\usepackage{float}
\usepackage{amscd,amsmath,amssymb,amsfonts,amsthm,ascmac}
\usepackage{enumerate,stmaryrd,latexsym, comment, mathtools, mathdots} 

\usepackage{varwidth}

\usepackage{amsthm, amssymb}
\usepackage{bm}
\usepackage[leqno]{amsmath}
\usepackage{caption}
\usepackage{subcaption}
\usepackage{hyperref}
\usepackage{relsize}
\usepackage{booktabs}
\usepackage{a4wide}

\usepackage[mathscr]{euscript}

\usepackage[all]{xy}

\usepackage{ragged2e}

\makeatletter
\renewcommand{\tocsection}[3]{%
  \indentlabel{\@ifnotempty{#2}{\bfseries\ignorespaces#1 #2\quad}}\bfseries#3}
\renewcommand{\tocsubsection}[3]{%
  \indentlabel{\@ifnotempty{#2}{\ignorespaces#1 #2\quad}}#3}

\newcommand\@dotsep{4.5}
\def\@tocline#1#2#3#4#5#6#7{\relax
  \ifnum #1>\c@tocdepth 
  \else
    \par \addpenalty\@secpenalty\addvspace{#2}%
    \begingroup \hyphenpenalty\@M
    \@ifempty{#4}{%
      \@tempdima\csname r@tocindent\number#1\endcsname\relax
    }{%
      \@tempdima#4\relax
    }%
    \parindent\z@ \leftskip#3\relax \advance\leftskip\@tempdima\relax
    \rightskip\@pnumwidth plus1em \parfillskip-\@pnumwidth
    #5\leavevmode\hskip-\@tempdima{#6}\nobreak
    \leaders\hbox{$\m@th\mkern \@dotsep mu\hbox{.}\mkern \@dotsep mu$}\hfill
    \nobreak
    \hbox to\@pnumwidth{\@tocpagenum{\ifnum#1=1\bfseries\fi#7}}\par
    \nobreak
    \endgroup
  \fi}
\AtBeginDocument{%
\expandafter\renewcommand\csname r@tocindent0\endcsname{0pt}
}
\def\l@subsection{\@tocline{2}{0pt}{2.5pc}{5pc}{}}
\makeatother

\usepackage{lipsum}

\usepackage{tikz-cd}
\usepackage{tikz}
\usetikzlibrary{3d, matrix,decorations.pathreplacing,calc,decorations.pathmorphing, fit, patterns, trees, decorations.markings}
\usetikzlibrary {positioning}

\numberwithin{equation}{section}
\allowdisplaybreaks[1]

\usepackage{colonequals} 

\allowdisplaybreaks

\DeclareMathOperator{\Lie}{Lie}

\DeclareMathOperator{\Ad}{Ad}
\DeclareMathOperator{\Cone}{Cone}

\DeclareMathOperator{\Cox}{Cox}
\DeclareMathOperator{\Hom}{Hom}
\DeclareMathOperator{\Int}{Int}
\DeclareMathOperator{\Conv}{Conv}

\newcommand{\defi}[1]{{\textit{#1}}}
\newcommand{\Xvw}[2]{{X^{#1}_{#2}}}

\newcommand{\C}{{\mathbb{C}}}

\newcommand{\R}{{\mathbb{R}}}

\newcommand{\Z}{{\mathbb{Z}}}
\newcommand{\Cstar}{{\C^{\ast}}}

\renewcommand{\ll}{\mathsf{L}}
\newcommand{\rr}{\mathsf{R}}
\newcommand{\head}{\mathsf{h}}
\newcommand{\tail}{\mathsf{t}}
\newcommand{\uh}[1]{{#1}_{\head}}
\newcommand{\ut}[1]{{#1}_{\tail}}

\DeclareMathOperator{\flag}{Fl}

\def\Sn{S_{n}}

\def\aret{\mathscr{R}^a}
\def\gret{\mathscr{R}^g}
\def\mret{\mathscr{R}^m}

\def\A{\mathscr{M}}
\def\asc{\mathrm{asc}}
\newcommand{\Poin}{{\mathscr{P}\textup{oin}}}

\newcommand{\T}{\mathscr{T}}
\newcommand{\lT}{\T^{\ll}}
\newcommand{\rT}{\T^{\rr}}
\newcommand{\B}{\mathcal{B}}

\newcommand{\Q}{{\mathsf{Q}}}
\newcommand{\polygon}[1]{\mathsf{P}_{#1}}
\newcommand{\Qvw}[2]{{\Q^{#1}_{#2}}}
\newcommand{\vertex}[1]{\mathsf{#1}}

\DeclareMathOperator{\Perm}{Perm}

\def\Xvw{X^v_w}

\def\GL{\rm GL}

\newcommand{\PP}{{\mathsf{P}}}

\def\Xw{X_w}
\def\Yw{Y_w}

\def\Cv{C(v)}

\def\G{\mathcal{G}}

\theoremstyle{plain}
\newtheorem{theorem}{Theorem}[section]
\newtheorem{lemma}[theorem]{Lemma}
\newtheorem{proposition}[theorem]{Proposition}
\newtheorem{corollary}[theorem]{Corollary}
\newtheorem{conjecture}[theorem]{Conjecture}

\theoremstyle{definition}
\newtheorem{example}[theorem]{Example}
\newtheorem{definition}[theorem]{Definition}

\newtheorem{Problem}[theorem]{Problem}

\theoremstyle{remark}
\newtheorem{remark}[theorem]{Remark}

\newcommand{\lee}[1]{\textcolor{black}{#1}}
\newcommand{\park}[1]{\textcolor{black}{#1}}
\newcommand{\masuda}[1]{\textcolor{black}{#1}}

\hypersetup{colorlinks}
\hypersetup{
	unicode=false,          
	pdftoolbar=true,        
	pdfmenubar=true,        
	pdffitwindow=false,     
	pdfstartview={FitH},    
	pdftitle={My title},    
	pdfauthor={Author},     
	pdfsubject={Subject},   
	pdfcreator={Creator},   
	pdfproducer={Producer}, 
	pdfkeywords={keyword1} {key2} {key3}, 
	pdfnewwindow=true,      
	colorlinks=true,       
	linkcolor=blue,          
	citecolor=red,        
	filecolor=violet,      
	urlcolor=violet       
}

\tikzset{root/.style = {circle, double, draw, inner sep = 1pt}}
\tikzset{vertex/.style = {circle, fill, inner sep = 1.5pt}}

\makeatletter
\def\namedlabel#1#2{\begingroup
   \def\@currentlabel{#2}%
   \label{#1}\endgroup
}
\makeatother

\begin{document}

\author[Eunjeong Lee]{Eunjeong Lee}
\address[E. Lee]{Department of Mathematics,
	Chungbuk National University,
	Cheongju 28644, Republic of Korea}
\email{eunjeong.lee@chungbuk.ac.kr}

\author[Mikiya Masuda]{Mikiya Masuda}
\address[M. Masuda]{Osaka City University Advanced Mathematics Institute 
(OCAMI) \& Department of Mathematics, Graduate School of Science, Osaka City 
University, Sumiyoshi-ku, Sugimoto, 558-8585, Osaka, Japan}
\email{mikiyamsd@gmail.com}

\author[Seonjeong Park]{Seonjeong Park}
\address[S. Park]{Department of Mathematics Education, Jeonju University, Jeonju 55069, Republic of Korea}
\email{seonjeongpark@jj.ac.kr}

\thanks{Lee was supported by the National Research Foundation of Korea(NRF) grant funded by the Korea government(MSIT) (RS-2022-00165641 and RS-2023-00239947), and supported by POSCO Science Fellowship of POSCO TJ Park Foundation.
Masuda was supported in part by 
the HSE University Basic Research Program. Park was supported by the Basic Science Research Program through the National Research Foundation of Korea (NRF) funded by the Government of Korea (NRF-2020R1A2C1A01011045). This work was partly supported by Osaka City University Advanced Mathematical Institute (MEXT Joint Usage/Research Center on Mathematics and Theoretical Physics JPMXP0619217849).}

\title{Torus orbit closures in the flag variety}

\date{\today}



\begin{abstract} 
The study of torus orbit closures in the (complete) flag variety was initiated by Klyachko and Gelfand--Serganova in the mid-1980s. In this chapter, we present some of the work by Klyachko and Gelfand--Serganova and our recent work on the topology, geometry, and combinatorics of torus orbit closures in the flag variety. 
\end{abstract}

\maketitle

\setcounter{tocdepth}{2} \tableofcontents

\section{Introduction}

The study of torus orbit closures in the flag variety was initiated by Klyachko~\cite{Klyachko85Orbits} and Gelfand--Serganova ~\cite{GelfandSerganova87}. Since then, there are some works on the subject but those are mainly about generic orbits or the normality or smoothness of orbit closures (see \cite{Flaschka-Haine}, \cite{Dabrowski}, \cite{Morand}, \cite{ca-ku2000}, \cite{CarrellKuttler03}). 
There have also been combinatorial studies on polytopes related to torus orbit closure in the flag variety (see \cite{ABD_matroidpolytopes}, \cite{KodamaWilliams}, \cite{TsukermanWilliams}, \cite{BK_Lattice_path_matroids}).
Recently, the authors investigated the topology and geometry of (not necessarily generic) torus orbit closures in connection with combinatorics. In this chapter, we present some of the work by Klyachko and Gelfand--Serganova and our recent works (see \cite{LM2020}, \cite{LMP_Catalan}, \cite{LMP_retractions}, \cite{LMP_complexity_one}, \cite{LMPS_Poincare}) on the topology, geometry, and combinatorics of torus orbit closures in the flag variety. We mainly focus on type $A$ case although some of the arguments work for other Lie types. We give comments for other Lie types when necessary. 

Let $\flag(n)$ be the flag variety defined by 
\[
\flag(n) \colonequals \{ (\{0\} \subset V_1 \subset V_2 \subset \cdots \subset V_n = \C^n) \mid \dim_{\C} V_i = i \: \text{ for all }i = 1,\dots,n\},
\]
where each $V_i$ is a linear subspace of $\C^n$.
Let $\Sn$ be the symmetric group on $[n]=\{1,\dots,n\}$. An element of $\Sn$ determines a permutation flag. The natural action of the general linear group~${\rm GL}_n(\C)$ on $\C^n$ induces an action of ${\rm GL}_n(\C)$ on~$\flag(n)$. Let $T$ be the maximal torus of ${\rm GL}_n(\C)$ consisting of diagonal matrices, which is isomorphic to $(\C^*)^n$ where $\C^*=\C\backslash\{0\}$. Then the set of permutation flags is exactly the $T$-fixed point set $\flag(n)^T$ in $\flag(n)$ and we identify $\flag(n)^T$ with $\Sn$. A key tool to study the $T$-action on $\flag(n)$ is a moment map 
\[
\mu\colon \flag(n)\to \R^n
\]
which satisfies the following properties:
\begin{enumerate}
\item[$\bullet$] $\mu(w)=(w^{-1}(1),\dots,w^{-1}(n))$ for $w\in\flag(n)^T=\Sn$, 
\item[$\bullet$] $\mu(\flag(n))$ is the convex hull of $n!$ points $\mu(\flag(n)^T)$ in $\R^n$, that is,  the permutohedron $\Pi_n$. 
\end{enumerate}
We will explore $T$-varieties (i.e., an algebraic variety with an  algebraic $T$-action) in $\flag(n)$ with respect to the action of the maxima torus $T$ of $\lee{\GL_n(\C)}$. 
With this understanding, we will explain the content of each section.

\medskip
\noindent
{\bf The content of Section~\ref{sec:general}}. 
The closure of a $T$-orbit in $\flag(n)$, denoted by $Y$, is a $T$-variety and has an open dense $T$-orbit, so $Y$ is a toric variety.\footnote{$Y$ is normal in type $A$ but not necessarily normal for other Lie types (see~\cite{ca-ku2000}).} 
Since scalar matrices in $T$ act trivially on~$\flag(n)$, the (complex) dimension of $Y$ is at most $n-1$. By the convexity theorem due to Atiyah~\cite[Theorem~2]{Atiyah82} (or Guillemin--Sternberg~\cite{GuilleminSternberg82}), the image $\mu(Y)$ is the convex hull of the points $\mu(Y^T)$ in $\R^n$. The fan of $Y$ is the normal fan of the polytope $\mu(Y)$, so $\mu(Y)$ determines $Y$ as a variety, up to isomorphism. Then two fundamental questions arise: 
\begin{enumerate}
\item[(Q1)\namedlabel{questionQ1}{(Q1)}] Characterize the polytopes which arise as $\mu(Y)$ for some $T$-orbit closure $Y$.  
\item[(Q2)\namedlabel{questionQ2}{(Q2)}] Describe the fan of $Y$ explicitly. 
\end{enumerate}

Gelfand--Serganova \cite{GelfandSerganova87} show that any edge of $\mu(Y)$ must be parallel to an edge of the permutohedron $\Pi_n$. Since the edges of $\Pi_n$ are parallel to roots $\Phi=\{\mathbf{e}_i-\mathbf{e}_j\mid i\not=j\in [n]\}$, where $\mathbf{e}_1,\dots,\mathbf{e}_n$ denote the standard basis of $\R^n$, a polytope with edges parallel to roots is called a \emph{$\Phi$-polytope}. Therefore $\mu(Y)$ is a $\Phi$-polytope. Unfortunately, the converse is not true and the complete solution to~\ref{questionQ1} above is unknown. 
We note that the notion of a $\Phi$-polytope is known as generalized permutohedron~\cite{PostnikovReinerWilliams08}. They were originally defined by Edmonds~\cite{Edmonds70}, and they are also called $M$-convex sets (see~\cite{Murota03}), and polymatroids (see~\cite{Schrijver03}).

$\Phi$-polytopes are closely related to Coxeter matroids. In our context, a Coxeter matroid is a subset~$\A$ of $\Sn$ which satisfies the Maximality (or Minimality) Condition (see Definition~\ref{defi:Coxeter_matroid}). The fixed point set $Y^T$ regarded as a subset of $\Sn$ is a Coxeter matroid. One can define a retraction 
\[
\mret_\A\colon \Sn\to \A\subset \Sn
\] 
for a Coxeter matroid $\A$ and describe the fan of $Y$ explicitly using the retraction to $Y^T$ (see Corollary~\ref{coro:fan_of_Y}), which answers~\ref{questionQ2} above. 
Here, the suffix $m$ in $\mret_\A$ shows that this retraction is defined using the matroid structure on $\A$, so that one can distinguish it from the two more retractions introduced later (see Subsection~\ref{subsec:Retractions and metric on finite Coxeter groups}).

The retraction $\mret_\A$ has two other interpretations, one is geometric and the other is algebraic (see Subsection~\ref{subsec:Retractions and metric on finite Coxeter groups}). It also has the following interesting property. The $1$-skeleton of the permutohedron $\Pi_n$ is a graph with $\Sn$ as vertices. If $\A$ is a Coxeter matroid, then for each $u\in \Sn$, a vertex in $\A$ closest to $u$ (with respect to the graph metric) is unique and the closest vertex in $\A$ is given by $\mret_\A(u)$ (see Proposition~\ref{prop_mret_minimal_distance}).

\medskip
\noindent
{\bf The content of Section~\ref{sec_generic_orbit_in_flag}.} 
We discuss the topology of a \emph{generic} $T$-orbit closure $Y$ in $\flag(n)$, where $Y$ is called generic if 
\[
Y^T=\flag(n)^T=\Sn,
\] 
in other words, $\mu(Y)=\Pi_n$. Since $\Pi_n$ is simple, $Y$ is smooth and of complex dimension $n-1$. In fact, $Y$ is isomorphic to the permutohedral variety $\Perm_n$ which is a toric variety whose fan consists of Weyl chambers. So we may think of the generic $T$-orbit closure $Y$ as $\Perm_n$. The cohomology of $\Perm_n$ can be explicitly described and its Poincar\'e polynomial agrees with $A_n(t^2)$ where $A_n(t)$ denotes the $n$th Eulerian polynomial. 
The fan of $\Perm_n$ has the symmetry of $\Sn$ which induces an action of $\Sn$ on the cohomology of $\Perm_n$. Klyachko \cite{Klyachko85Orbits} shows that the image of the restriction map 
\[
\iota^*\colon H^*(\flag(n);\mathbb{Q})\to H^*(\Perm_n;\mathbb{Q}),
\]
where $\iota\colon \Perm_n\to \flag(n)$ is the inclusion map, is the ring of $\Sn$-invariants $H^*(\Perm_n;\mathbb{Q})^{\Sn}$. This fact is generalized to the setting of regular semisimple Hessenberg varieties. Here, regular semisimple Hessenberg varieties are subvarieties of $\flag(n)$ satisfying certain conditions (see Subsection~\ref{sec_Hess}) and $\Perm_n$ is the only $(n-1)$-dimensional regular semisimple Hessenberg variety that is toric with respect to the maximal torus of $\GL_n(\C)$ \lee{(see, for example,~\cite[Theorem~11]{DePS_Hess_1992})}.  

\medskip
\noindent
{\bf The content of Section~\ref{sec:Schubert}.} 
The Schubert variety $X_w$ $(w\in\Sn)$ is a subvariety of $\flag(n)$ invariant under the $T$-action on $\flag(n)$ and the moment map image of $X_w$ is the polytope
\[
\Q_w\colonequals \Conv\{(u^{-1}(1),\dots,u^{-1}(n))\in\R^n\mid u\le w\},
\]
where $\le$ denotes the Bruhat order\footnote{The \defi{Bruhat order} \lee{(or the \defi{Bruhat--Chevalley order})} on $\Sn$ is defined as follows. Let $s_i$ ($i=1,\dots,n-1$) denote the adjacent transposition interchanging $i$ and $i+1$. For $w \in \Sn$ with a reduced expression $w = s_{i_1} s_{i_2} \cdots s_{i_q}$ and for $u \in \Sn$, we denote $u \leq w$ if there is a reduced expression $u = s_{i_{j_1}} s_{i_{j_2}} \cdots s_{i_{j_k}}$ for $ 1 \leq j_1 < \cdots < j_k \leq q$.} on $\Sn$. 
We consider a \emph{generic} $T$-orbit closure, denoted by $Y_w$, in the Schubert variety $X_w$, where $Y_w$ is called generic in $X_w$ if 
\[
Y_w^T=X_w^T=\{u\in\Sn\mid u\le w\}.
\]

Since $X_{w_0}=\flag(n)$ for the longest element $w_0\in \Sn$, $Y_{w_0}$ is isomorphic to the permutohedral variety $\Perm_n$ as mentioned above. Although $Y_{w_0}$ is smooth, $Y_w$ is not necessarily smooth. For each $u\in \Sn$ with $u\le w$, we introduce a graph $\Gamma_w(u)$ such that $Y_w$ is smooth at $u\in Y_w^T$ if and only if $\Gamma_w(u)$ is a forest. It seems that the graph $\Gamma_w(w)$ is most complicated among graphs~$\Gamma_w(u)$ with $u\le w$. To be more precise, we posed the following as a conjecture in \cite{LM2020} and it was affirmatively answered in~\cite{Gaetz22_one-skeleton} while preparing this chapter. 

\medskip
\noindent
{\bf Theorem} (\cite{LM2020, Gaetz22_one-skeleton}). 
The graph $\Gamma_w(u)$ is a forest for any $u \leq w $ if $\Gamma_w(w)$ is a forest. \textup{(}This is equivalent to saying that $Y_w$ is smooth if it is smooth at $w$, in other words, $\Q_{w}$ is simple if $\Q_w$ is simple at the vertex $\mu(w)$.\textup{)} 
\medskip

Indeed, $Y_w$ is always smooth at the identity $e$ and the theorem above says that $Y_w$ is entirely smooth if it is smooth at $w$.
In contrast to this, the Schubert variety $X_w$ is smooth at $w$ and entirely smooth if it is smooth at the identity $e$ (see~\cite[p.208]{BL20Singular}). 
Interestingly, the graph $\Gamma_w(w)$ appears in a different context, i.e.,  it is shown in \cite{Woo-Yong} that $\Gamma_w(w)$ is a forest if and only if $X_w$ is locally factorial. Moreover, permutations $w$ for which $\Gamma_w(w)$ is forest are characterized in terms of pattern avoidance (see~\cite{bm-bu07}). 

We introduce a polynomial $A_w(t)$ for $w\in \Sn$ purely combinatorially by looking at ascents. The polynomial $A_w(t)$ is the Eulerian polynomial $A_n(t)$ when $w=w_0$. As is well-known, $A_{w_0}(t^2)$ agrees with the Poincar\'e polynomial of the permutohedral variety $\Perm_n$, which is isomorphic to~$Y_{w_0}$. This fact is generalized in such a way that $A_w(t^2)$ agrees with the Poincar\'e polynomial of~$Y_w$ for any~$w\in\Sn$ (see Theorem~\ref{Thm_Poincare}). 

We set
\[
c(w)\colonequals\dim_\C X_w-\dim_\C Y_w.
\]
When $c(w)=0$, we have $X_w=Y_w$ which means that $X_w$ is a toric variety with the $T$-action. There are several equivalent conditions for $X_w$ to be a toric variety (see Theorem~\ref{thm:toric}), e.g., a reduced decomposition of $w$ is a product of distinct adjacent transpositions. We describe the fan of a toric Schubert variety $X_w$ explicitly in terms of $w$ (see Theorem~\ref{thm_char_matrix_of_toric_Schubert}). This implies that $X_w$ is a weak Fano Bott manifold, see Appendix~\ref{sec_Fano_Bott} for details of Bott manifolds. We also present a classification result of toric Schubert varieties $X_w$ for Coxeter elements $w$ (see Theorem~\ref{thm:dynkin}). 

The case where $c(w)=1$ is also studied. In this case, $X_w$ is not necessarily smooth. We present several equivalent conditions for $c(w)=1$ depending on whether $X_w$ is smooth or singular (see Theorems~\ref{thm:smooth-one} and~\ref{thm:singular-one}). 

\medskip
\noindent
{\bf The content of Section~\ref{sec:Richardson}.}  
For a pair $(v,w)$ of permutations with $v\le w$, the intersection of a Schubert variety $X_w$ and an opposite Schubert variety $w_0X_{w_0v}$ is non-empty and irreducible. The intersection is denoted by $X^v_w$ and called a \emph{Richardson variety}. It is invariant under the $T$-action on $\flag(n)$. The moment map image of $X^v_w$ is the \emph{Bruhat interval polytope} 
\[
\Q^v_w\colonequals\Conv\{(u^{-1}(1),\dots,u^{-1}(n))\in\R^n\mid v\le u\le w\},
\]
introduced by Kodama--Williams \cite{KodamaWilliams}. 
Note that $\Q^e_w=\Q_w$. 
The vertex $\mu(v)$ of $\Q^v_w$ is not necessarily simple in $\Q^v_w$ while it is simple when $v=e$. A natural generalization of the conjecture mentioned above is that $\Q^v_w$ is simple if the vertices $\mu(v)$ and $\mu(w)$ are both simple in $\Q^v_w$. 

We study the case when $X^v_w$ is a toric variety with the $T$-action. A toric Richardson variety~$X^v_w$ is not necessarily smooth while every toric Schubert variety $X_w$ is smooth. It turns out that $X^v_w$ is a smooth toric variety if and only if $\Q^v_w$ is combinatorially equivalent to a cube (see Theorem~\ref{theo:3-6}). Although pairs $(v,w)$ for which $\Q^v_w$ is combinatorially equivalent to a cube are not completely understood, there are many such pairs. Especially, it is shown in \cite{HHMP19}
that this is the case when 
\begin{equation*} \label{eq:wv_pair}
w=s_{n-1}s_{n-1}\cdots s_1v\quad (\text{or } w=s_1s_2\cdots s_{n-1} v)\quad\text{and}\quad \ell(w)-\ell(v)=n-1.
\end{equation*}
Toric Richardson varieties $X^v_w$ for the pairs $(v,w)$ above also arise from polygon triangulations, so we call such toric Richardson varieties of \emph{Catalan type}. They can be classified up to isomorphism and the Wedderburn--Etherington numbers which count \emph{unordered} binary trees appear in enumerating the isomorphism classes (see Corollary~\ref{cor_enumeration_Xvw}). 

\medskip 

Finally, in Section~\ref{sec_problems} we pose several problems related to the discussion developed in this chapter. In Appendix~\ref{appendix}, we briefly review the theory of toric varieties and explain Bott manifolds in detail.


\section{Torus orbit closures (general)}\label{sec:general}
In this section, we review the torus actions on flag varieties and the closure of a torus orbit. Moreover, we consider properties of the moment map images of torus orbit closures. Indeed, the fixed point set of a torus orbit closure becomes a Coxeter matroid. 
Finally, we discuss how to describe the fan of a torus orbit closure.


\subsection{Torus actions on flag varieties}
Let $G$ be the general linear group ${\rm GL}_n(\mathbb{C})$ over $\C$, $B \subset G$ the set of upper triangular matrices, and $T\subset G$ the set of diagonal matrices. Let $B^- \subset G$ be the set of lower triangular matrices. 
We denote by $\Sn$ the symmetric group on $[n] \colonequals \{1,\dots,n\}$.
Then $T = B \cap B^- \cong (\Cstar)^n$ and $B^-=w_0Bw_0$, where $w_0$ is the longest element $[n,n-1,\dots,1] \in \Sn$.
The homogeneous space~$G/B$ can be identified with the \emph{flag variety}~$\flag(n)$ defined by 
\[
\flag(n) \colonequals \{ (\{0\} \subset V_1 \subset V_2 \subset \cdots \subset V_n = \C^n) \mid \dim_{\C} V_i = i \: \text{ for all }i = 1,\dots,n\},
\]
where $\C^n$ is considered as the complex vector space consisting of {\it column} vectors.
For $g \in G$, let $\mathbf{c}_1,\dots,\mathbf{c}_n$ be the column vectors of the matrix $g$, that is, $g =[\mathbf{c}_1 \ \cdots \ \mathbf{c}_n]$. Then $gB \in G/B$ corresponds to the flag whose $i$th vector space is spanned by the first $i$ columns $\mathbf{c}_1,\dots,\mathbf{c}_i$ of $g$.

For an element~$w \in \Sn$, we use the same letter $w$ for the permutation matrix $\begin{bmatrix}\mathbf{e}_{w(1)} & \cdots & \mathbf{e}_{w(n)}\end{bmatrix}$ in ${\rm GL}_n(\C)$ to simplify notation, where $\mathbf{e}_1,\dots,\mathbf{e}_n$ are the standard basis vectors in $\mathbb{C}^n$.
The left multiplication by~$T$ on $G$ induces the $T$-action on $G/B$.
\begin{lemma}[{\cite[\S 10.5]{Fulton97Young}}]\label{lemma_T_fixed_points}
The set of $T$-fixed points in $G/B$ bijectively corresponds to the symmetric group $\Sn$ such that each $u \in \Sn$ corresponds to $uB \in G/B$.
\end{lemma}

For each $u\in \Sn$, there is a $T$-invariant local chart $U_u$ given by 
\begin{equation} \label{eq:Uu}
U_u\colonequals\left\{(x_{ij})B\in G/B \,\, \middle|\,\, x_{ij}=\begin{cases}1 \quad &\text{ if }j=u^{-1}(i),\\
0\quad &\text{ if }j>u^{-1}(i)\end{cases}\right\},
\end{equation}
which is isomorphic to $\C^{n(n-1)/2}$ since $x_{ij}$ with $j<u^{-1}(i)$ are arbitrary complex numbers. The $T$-action on $G/B$ restricted to $U_u$ is given by 
\begin{equation} \label{eq:tij}
(t_1,\dots,t_n)\cdot (x_{ij})B=(t_it_{u(j)}^{-1}x_{ij})B. 
\end{equation}
Here, for $(t_1,\dots,t_n) \in (\Cstar)^n$, we associate a diagonal matrix $\textup{diag}(t_1,\dots,t_n) \in T$. Since we have the $1$'s on the $(i,u^{-1}(i))$-entries of $(x_{ij}) \in U_u$ (which is equivalent to saying that we have the $1$'s on the $(u(j),j)$-entries), by considering the multiplication of an element $\textup{diag}(t_{u(1)}^{-1},\dots,t_{u(n)}^{-1})$ of $B$ from the right, we obtain the above computation. We demonstrate this action for $u = 312$ in Example~\ref{exam:312}.
Therefore, $uB$ is a unique $T$-fixed point in $U_u$, which corresponds to the origin in $\C^{n(n-1)/2}$ and a $1$-dimensional $T$-orbit in $U_u$ corresponds to a $1$-dimensional $T$-eigenspace in $\C^{n(n-1)/2}$. 

\begin{example} \label{exam:312}
When $u=312$ in one-line notation, 
\[
U_{312}=\left\{\begin{pmatrix} x_{11}&1&0\\
x_{21}&x_{22}&1\\
1&0&0\end{pmatrix}B \,\, \middle|\,\, x_{ij}\in \C\right\}\cong \C^3,
\]
and the $T$-action on $U_{312}$ is given by 
\[
\begin{split}
(t_1,t_2,t_3)\cdot\begin{pmatrix} x_{11}&1&0\\
x_{21}&x_{22}&1\\
1&0&0\end{pmatrix}B &= 
\begin{pmatrix}
t_1x_{11} & t_1 & 0 \\
t_2 x_{21} & t_2 x_{22} & t_2 \\
t_3 & 0 & 0 
\end{pmatrix}B \\
&= 
\begin{pmatrix}
t_1x_{11} & t_1 & 0 \\
t_2 x_{21} & t_2 x_{22} & t_2 \\
t_3 & 0 & 0 
\end{pmatrix}
\begin{pmatrix}
t_3^{-1} & 0 & 0 \\
0 & t_1^{-1} & 0  \\
0 & 0 & t_2^{-1} 
\end{pmatrix}B \\
&= 
\begin{pmatrix} t_1t_3^{-1}x_{11}&1&0\\
t_2t_3^{-1}x_{21}&t_2t_1^{-1}x_{22}&1\\
1&0&0\end{pmatrix}B. 
\end{split}
\]
\end{example}


\subsection{Moment map}
We describe a moment map $\mu\colon G/B\to \R^n$ explicitly using the Pl\"{u}cker coordinates with respect to the action of the (maximal) compact torus $T_{\R}$ in $T$. Note that $T \cong (S^1)^n$.
Define the set
\[
I_{d,n} \colonequals \{\underline{\mathbf{i}} = (i_1,\dots,i_d) \in \Z^d \mid 1 \leq i_1 < \cdots < i_d \leq n \}.
\]
For an element $x = (x_{ij}) \in G = {\rm GL}_n(\C)$, the $\underline{\mathbf{i}}$th Pl\"{u}cker coordinate $p_{\underline{\mathbf{i}}}(x)$ of~$x$ is given by the $d \times d$ minor of~$x$, with row indices $i_1,\dots,i_d$ and the column indices $1,\dots,d$ for $\underline{\mathbf{i}} = (i_1,\dots,i_d) \in I_{d,n}$. The Pl\"{u}cker embedding is defined to be
\begin{equation}\label{eq_Plucker_embedding}
\psi \colon G/B \to \prod_{d=1}^{n-1} \mathbb{C}P^{\binom{n}{d}-1},
\quad xB \mapsto \prod_{d=1}^{n-1} [p_{\underline{\mathbf{i}}}(x)]_{\underline{\mathbf{i}} \in I_{d,n}}.
\end{equation}
The map $\psi$ is well-defined because $p_{\underline{\mathbf i}}(xb)=(\prod_{j=1}^d b_{jj}) p_{\underline{\mathbf i}}(x)$ for $b=(b_{ij})\in B$ and $\underline{\mathbf i} \in I_{d,n}$.
The map $\psi$ is $T$-equivariant with respect to the action of~$T$ on $\prod_{d=1}^{n-1} \mathbb{C} P^{\binom{n}{d}-1}$ given by
\begin{equation*}\label{eq_action_on_CPn}
(t_1,\dots,t_n) \cdot [p_{\underline{\mathbf{i}}}]_{\underline{\mathbf{i}} \in I_{d,n}}
\colonequals [t_{i_1}\cdots t_{i_d} \cdot p_{\underline{\mathbf i}}]_{\underline{\mathbf{i}} \in I_{d,n}}
\end{equation*}
for $(t_1,\dots,t_n) \in T$ and $\underline{\mathbf{i}} = (i_1,\dots,i_d) \in I_{d,n}$. 
Here, with the abuse of notation, we are denoting by $p_{\underline{\mathbf i}}$ both a function $p_{\underline{\mathbf i}} \colon G \to \C$ and the $\underline{\mathbf i}$th coordinate of $\C P^{\binom{n}{d} -1}$.

Let $\omega_{FS}$ be the Fubini--Study form on a complex projective space $\C P^{m-1}$. With respect to the standard action of the compact torus $(S^1)^{m}$ on $\C P^{m-1}$ defined by 
\[
(t_1,\dots,t_m)\cdot [z_1,\dots,z_m]=[t_1z_1,\dots,t_mz_m],
\]
where each $t_i$ is a complex number with unit length, 
the moment map of $(\C P^{m-1},\omega_{FS}, (S^1)^{m})$ is given by
\begin{equation}\label{eq_moment_map_CPn}
[z_1,\dots,z_m] \mapsto 
\frac{-1}{2 \sum_{i=1}^m |z_i|^2} \left( |z_1|^2, \dots,|z_m|^2 \right) \quad {\text{up to translation}}.
\end{equation}
See, for example,~\cite[Example~IV.1.2]{AudinTorus}.\footnote{Here, we use a different sign convention from that in~\cite{AudinTorus}. That is,  our moment map $\mu \colon (M,\omega,(S^1)^n) \to \text{Lie}((S^1)^n)^{\ast}$ satisfies the following: For each $X \in \text{Lie}((S^1)^n)$, $d \mu^{X} = \iota_{X^{\#}}\omega$, where $\mu^X(p) = \langle \mu(p),X \rangle$ and $X^{\#}$ is the vector field on $M$ generated by the one-parameter subgroup $\{\exp tX \mid t \in \R\} \subset (S^1)^n$.} 
Because the action of $(S^1)^n$ on the factor $\C P^{\binom{n}{d}-1}$ in \eqref{eq_Plucker_embedding} is given through the homeomorphism $\phi_d \colon T_{\R}= (S^1)^n \to (S^1)^{\binom{n}{d}}$ sending $(t_1,\dots,t_n)$ to $(t_{i_1}\cdots t_{i_d})_{\underline{\mathbf{i}} \in I_{d,n}}$, the moment map of $(\C P^{\binom{n}{d}-1}, 2\omega_{FS}, T_{\R})$ is given by
\[
[p_{\underline{\mathbf{i}}}]_{\underline{\mathbf{i}}\in I_{d,n}} \mapsto \frac{-1}{\sum_{\underline{\mathbf{i}}\in I_{d,n}} | p_{\underline{\mathbf{i}}}|^2}\left(\sum_{1\in \underline{\mathbf{i}}\in I_{d,n}}|p_{\underline{\mathbf{i}}}|^2,\ldots,\sum_{n\in \underline{\mathbf{i}}\in I_{d,n}}|p_{\underline{\mathbf{i}}}|^2\right)\quad {\text{up to translation}}. 
\]

\lee{We note that for symplectic manifolds $M_1$ and $M_2$ having Hamiltonian actions of the same Lie group $G$ with moment maps $\mu_1 \colon M_1 \to \text{Lie}(G)^{\ast}$ and $\mu_2 \colon M_2 \to \text{Lie}(G)^{\ast}$, respectively, the diagonal action of $G$ on $M_1 \times M_2$ is Hamiltonian. Moreover, the moment map $\mu \colon M_1 \times M_2 \to \text{Lie}(G)^{\ast}$ is given by the sum of moment maps $\mu_1$ and $\mu_2$. Indeed, $\mu(x_1,x_2) = \mu_1(x_1) + \mu_2(x_2)$ (see, for instance,~\cite[Exercise III.3]{AudinTorus}).
}
By considering a symplectic form $\omega$ on $\prod_{d=1}^{n-1}\C P^{\binom{n}{d}-1}$ given by two times of the Fubini--Study form on each complex projective space \lee{and the diagonal action on the product}, 
the moment map $\tilde\mu\colon {\prod_{d=1}^{n-1}\C P^{\binom{n}{d}-1}} \to \R^n$ is given by
\begin{equation}\label{eq:moment-map}
\prod_{d=1}^{n-1}[p_{\underline{\mathbf{i}}}]_{\underline{\mathbf{i}}\in I_{d,n}} \mapsto -\sum_{d=1}^{n-1}\left\{\frac{1}{\sum_{\underline{\mathbf{i}}\in I_{d,n}} | p_{\underline{\mathbf{i}}}|^2}\left(\sum_{1\in \underline{\mathbf{i}}\in I_{d,n}}|p_{\underline{\mathbf{i}}}|^2,\ldots,\sum_{n\in \underline{\mathbf{i}}\in I_{d,n}}|p_{\underline{\mathbf{i}}}|^2\right)\right\} + \mathbf{c},
\end{equation} where $\mathbf{c}$ is a constant vector in $\R^n$, and the action of $(S^1)^n$ on each factor $\C P^{\binom{n}{d}-1}$ is given through the homomorphism $\phi_d$ above.
We take $\mathbf{c}=(n,\ldots,n)$ in~\eqref{eq:moment-map} and define 
\begin{equation}\label{eq_moment_map}
\mu \colonequals \tilde{\mu} \circ \psi,
\end{equation}
which is a moment map of $(G/B, \psi^{\ast}\omega, T_{\R})$.

Before computing the moment image of the fixed point $uB \in G/B$, we introduce one notation. 
For a subset $S\subset [n]$, we denote by $S\!\!\uparrow$ the ordered tuple obtained from $S$ by sorting its elements in ascending order. For instance, if $S = \{5,2,1,6\} \subset [8]$, then $S \!\!\uparrow = (1,2,5,6)$.
\begin{lemma}
\label{lem:moment_map}
The moment map $\mu$ sends the fixed point $uB \in G/B$ to $(u^{-1}(1),\dots,u^{-1}(n)) \in \mathbb{R}^n$.
\end{lemma}
\begin{proof}
For a permutation $u \in \Sn$, the Pl\"{u}cker coordinates $(p_{\underline{\mathbf{i}}})_{\underline{\mathbf{i}} \in I_{d,n}}$ of $uB$ are given as follows:
\begin{equation}\label{eq_pi_w}
p_{\underline{\mathbf{i}}} = \begin{cases} \pm 1 & \text{ if }\underline{\mathbf{i}} = \{u(1),\dots,u(d)\}\!\!\uparrow,\\
0 & \text{ otherwise,}
\end{cases}
\end{equation}
for each $\underline{\mathbf{i}} \in I_{d,n}.$ 
Therefore, one can see that for a fixed $d \in [n-1]$, $\sum_{\underline{\mathbf{i}}\in I_{d,n}} |p_{\underline{\mathbf{i}}}|^2=1$ and the entries of the vector
\[
\left( \sum_{1 \in \underline{\mathbf{i}}\in I_{d,n}} |p_{\underline{\mathbf{i}}}|^2, \dots, \sum_{n \in \underline{\mathbf{i}}\in I_{d,n}} |p_{\underline{\mathbf{i}}}|^2 \right)
\]
are $1$ for coordinates in $\{u(1),\dots,u(d)\}$ and $0$ otherwise.
Hence the summation
\[
-\sum_{d=1}^{n-1} 
\left\{ \frac{1}{\sum_{\underline{\mathbf{i}}\in I_{d,n}} | p_{\underline{\mathbf{i}}}|^2}\left( \sum_{1 \in \underline{\mathbf{i}}\in I_{d,n}} |p_{\underline{\mathbf{i}}}|^2, \dots, \sum_{n \in \underline{\mathbf{i}}\in I_{d,n}} |p_{\underline{\mathbf{i}}}|^2 \right)
\right\}
\]
is an integer vector such that the $u(k)$-entry is $-(n-k)$. Therefore, $\mu(uB)$ is an integer vector whose $u(k)$-entry is $k$ since $\mathbf{c}=(n,\dots,n)$ in~\eqref{eq:moment-map}. This implies that $\mu(uB) = (u^{-1}(1),\dots,u^{-1}(n))$ since $u^{-1}(u(k)) = k$ for all $k$.
\end{proof}

\begin{example}
\label{example_Plucker_GL3}
When $G = {\rm GL}_3(\mathbb{C})$, for $x = (x_{ij}) \in {\rm GL}_3(\mathbb{C})$, the Pl\"{u}cker embedding
$\psi \colon {G/B} \to \C {P}^{\binom{3}{1}-1} \times \C {P}^{\binom{3}{2}-1}$ maps an element $xB \in G/B$
to
\begin{align*}
&([p_1(x),p_2(x),p_3(x)], [p_{1,2}(x), p_{1,3}(x), p_{2,3}(x)]) \\
&\qquad = ([x_{11}, x_{21}, x_{31}], [x_{11} x_{22} - x_{21}x_{12}, x_{11}x_{32} - x_{31} x_{12}, x_{21}x_{32} - x_{31}x_{22}]).
\end{align*}
Since the action of $T$ on ${\rm GL}_3(\C)$ is given by
\begin{align*}
(t_1,t_2,t_3) \cdot \begin{pmatrix}
x_{11} & x_{12} & x_{13} \\
x_{21} & x_{22} & x_{23} \\
x_{31} & x_{32} & x_{33}
\end{pmatrix}
= \begin{pmatrix}
t_1 x_{11} & t_1 x_{12} & t_1 x_{13} \\
t_2 x_{21} & t_2 x_{22} & t_2 x_{23} \\
t_3 x_{31} & t_3 x_{32} & t_3 x_{33}
\end{pmatrix}
\end{align*}
for $t = (t_1,t_2,t_3) \in T$, we have 
\[
\begin{split}
\psi(t \cdot xB) &= ([t_1x_{11}, t_2x_{21}, t_3x_{31}],
[t_1t_2(x_{11}x_{22}-x_{21}x_{12}),
t_1t_3(x_{11}x_{32} - x_{31}x_{12}),
t_2t_3(x_{21}x_{32} - x_{31}x_{22})]) \\
&= 
t \cdot \psi(xB)
\end{split}
\]
so the map $\psi$ is $T$-equivariant. The moment map 
\[
\tilde\mu\colon\C P^{\binom{3}{1}-1}\times \C P^{\binom{3}{2}-1}\to \R^3
\] 
is given by
\begin{equation*}
\begin{split}
&([p_1,p_2,p_3],[p_{1,2},p_{1,3},p_{2,3}])\\&\qquad\mapsto -\frac{1}{|p_1|^2+|p_2|^2+|p_3|^2}\left({|p_1|^2},{|p_2|^2},{|p_3|^2}\right)\\&\qquad\quad\quad -\frac{1}{|p_{1,2}|^2+|p_{1,3}|^2+|p_{2,3}|^2}\left({|p_{1,2}|^2+|p_{1,3}|^2},{|p_{1,2}|^2+|p_{2,3}|^2},{|p_{1,3}|^2+|p_{2,3}|^2}\right)\\
&\qquad\quad\quad\quad+(3,3,3).
\end{split}
\end{equation*} 
For $312 \in S_3$, we have 
\[\mu(312B)=\tilde\mu\circ\psi(312B)=\tilde\mu(([0,0,1],[0,-1,0]))=-(0,0,1)-(1,0,1)+(3,3,3)=(2,3,1).
\]
\end{example}

\begin{remark}
The Pl\"{u}cker embedding can be generalized to partial flag varieties and other Lie types (see, for example,~\cite{GelfandSerganova87} and \cite{BGW03_Coxeter}). 
\end{remark}


\subsection{Torus orbit closures and their moment polytopes}

For a point $x \in G/B$, the closure $\overline{T \cdot x}$ of the $T$-orbit $T\cdot x$ is a toric variety in~$G/B$.
The moment map image $\mu(\overline{T\cdot x})$ is a convex polytope in $\R^n$ with vertices $\mu((\overline{T\cdot x})^T)$ by the convexity theorem (Atiyah~\cite[Theorem~2]{Atiyah82}, Guillemin--Sternberg~\cite{GuilleminSternberg82} for symplectic case).

On the other hand, $(\overline{T\cdot x})^T$ can be found as follows. 
Choose a representative $\tilde{x}\in {\rm GL}_n(\C)$ of $x$. The non-vanishing of $p_{\underline{\mathbf i}}(\tilde{x})$ is independent of the choice of the representative $\tilde{x}$ of $x$ and we define 
\[
I_d(x) \colonequals \{ \underline{\mathbf i} = (i_1,\dots,i_d) \in I_{d,n} \mid p_{\underline{\mathbf i}}(\tilde{x}) \neq 0\}, \quad 1 \leq d \leq n.
\]
\begin{proposition}[{\cite[Proposition~1 in \S5.2]{GelfandSerganova87}}]\label{proposition_GS_fixed_points}
For an element $x \in {\rm GL}_n(\C)/B$, we have 
\[
(\overline{T\cdot x})^T = \{wB \mid w \in S_n,  \{w(1),\dots,w(d)\}\!\!\uparrow~ \in I_d(x)\: \text{ for all } 1\leq d \leq n\}.
\] 
\end{proposition}
\begin{proof} 
We note that for $\underline{\mathbf i} \in I_{d,n}$, if $p_{\underline{\mathbf i}}(\tilde{x}) = 0$, then $p_{\underline{\mathbf i}}(\tilde{y}) = 0$ for any $y \in \overline{T \cdot x}$. Here, we denote by $\tilde{y} \in \GL_n(\C)$ a representative of $y$. 
Therefore we have
\[
\psi((\overline{T\cdot x})^T) \subset 
\left\{
[p_{\underline{\mathbf i}}]_{\underline{\mathbf i} \in I_{d,n}} 
\in  \prod_{d=1}^{n-1} \mathbb{C}P^{\binom{n}{d}-1}
~~\middle|~~ p_{\underline{\mathbf i}} \neq 0 \: \text{ if }\underline{\mathbf i} \in I_d(x)
\right\}.
\]
This provides 
\[
(\overline{T\cdot x})^T \subset \{wB \mid w \in S_n,  \{w(1),\dots,w(d)\}\!\!\uparrow~ \in I_d(x)\quad \text{for all } 1\leq d \leq n\}
\]
since for $\underline{\mathbf i} \in I_{d,n}$, we have $p_{\underline{\mathbf i}}(wB) \neq 0$ if and only if $\{w(1),\dots,w(d)\}\!\!\uparrow \:= \underline{\mathbf i}$ (see~\eqref{eq_pi_w}).

Now we consider the opposite inclusion.
Recall that the group of homomorphisms from $T$ to~$\Cstar$ is naturally isomorphic to $\Z^n$. For each $u = (u_1,\dots,u_n) \in \Z^n$, we have the one-parameter subgroup~$\lambda_u(t)=(t^{u_1},\dots,t^{u_n})$ of $T$, where $t\in\C^\ast$. 
Suppose that $w \in \Sn$ satisfies $\{w(1),\dots,w(d)\}\!\!\uparrow~ \in I_d(x)$ for all $1\leq d \leq n$. 
Take $u = (u_1,u_2,\dots,u_n) \in \Z^n$ such that 
$u_{w(1)} < u_{w(2)} < \cdots < u_{w(n)}$. Then we have 
\[
\lim_{t \to 0} \lambda_u(t) \cdot x = wB. 
\]
This proves the desired inclusion so we are done.
\end{proof}

\begin{example}\label{example_fixed_points_3}
Let $x=\begin{pmatrix}\alpha&1&0\\ \beta&0&1\\1&0&0 
\end{pmatrix}B \in {\rm GL}_3(\C)/B$ with $\alpha,\beta\in \C^*$. Then we have
\[
I_1(x) = \{(1), (2), (3) \}, \quad I_2(x) = \{(1,2),(1,3) \}, \quad I_3(x) = \{(1,2,3)\}, \text{ and }
\]
\[
\begin{split}
(\overline{T \cdot x})^T
&= \{wB\mid w \in S_3, \{w(1),\dots,w(d)\}\!\!\uparrow~ \in I_d(x) \: \text{ for all }1 \leq d \leq 3\} \\ 
&= \{123B,132B, 213B, 312B\}.
\end{split}
\]
Since
\begin{equation} \label{eq:4points}
\mu(123B)=(1,2,3),\ \mu(132B)=(1,3,2),\ \mu(213B)=(2,1,3),\ \mu(312B)=(2,3,1)
\end{equation}
by Lemma~\ref{lem:moment_map}, $\mu(\overline{T\cdot x})$ is the convex hull of the four points in \eqref{eq:4points}, see Figure~\ref{fig:vertex-labelling}. Now we label each vertex $\mu(uB)$ with $\vertex{u}$ (using a different font) for simplicity.
\end{example}

\begin{figure}[hbt] 
\begin{subfigure}[b]{.4\textwidth}
\centering
\begin{tikzpicture}[x=1cm, y=1cm, z=-0.6cm,scale=.5]
\fill[yellow] (2,3,1)--(1,3,2)--(3,1,2)--(3,2,1)--cycle;
\draw [->] (0,0,0) -- (4,0,0) node [right] {$y$};
\draw [->] (0,0,0) -- (0,4,0) node [left] {$z$};
\draw [->] (0,0,0) -- (0,0,4) node [left] {$x$};
\draw[thick] (2,3,1)--(1,3,2)--(3,1,2)--(3,2,1)--cycle;

\draw [dotted]
(0,3,0) -- (3,3,0) -- (3,0,0)
(3,0,0) -- (3,0,3) -- (0,0,3)
(0,3,0) -- (0,3,3) -- (0,0,3)
(3,3,0) -- (3,0,3)
(3,3,0) -- (0,3,3)
(0,3,3) -- (3,0,3)
(1,3,2) -- (1,2,3)
(2,1,3) -- (3,1,2)
;
\draw (3,0,0) node[below]{$3$};
\draw (0,3,0) node[left]{$3$};
\draw (0,0,3) node[right]{$3$};
\draw (1,2,3) node[left]{\tiny{$(3,1,2)$}};
\draw (2,1,3) node[below]{\tiny{$(3,2,1)$}};
\draw (2,3,1) node[above]{\tiny{$(1,2,3)$}};
\draw (1,3,2) node[left]{\tiny{$(2,1,3)$}};
\draw (3,2,1) node[right]{\tiny{$(1,3,2)$}};
\draw (3,1,2) node[below]{\tiny{$(2,3,1)$}};

\end{tikzpicture}
\subcaption{Vertices are labeled by coordinates}

\end{subfigure}
\begin{subfigure}[b]{.4\textwidth}
\centering
\begin{tikzpicture}[x=1cm, y=1cm, z=-0.6cm,scale=.5]
\fill[yellow] (2,3,1)--(1,3,2)--(3,1,2)--(3,2,1)--cycle;

\draw [->] (0,0,0) -- (4,0,0) node [right] {$y$};
\draw [->] (0,0,0) -- (0,4,0) node [left] {$z$};
\draw [->] (0,0,0) -- (0,0,4) node [left] {$x$};
\draw[thick] (2,3,1)--(1,3,2)--(3,1,2)--(3,2,1)--cycle;

\draw [dotted]
(0,3,0) -- (3,3,0) -- (3,0,0)
(3,0,0) -- (3,0,3) -- (0,0,3)
(0,3,0) -- (0,3,3) -- (0,0,3)
(3,3,0) -- (3,0,3)
(3,3,0) -- (0,3,3)
(0,3,3) -- (3,0,3)
(1,3,2) -- (1,2,3)
(2,1,3) -- (3,1,2)
;
\draw (3,0,0) node[below]{$3$};
\draw (0,3,0) node[left]{$3$};
\draw (0,0,3) node[right]{$3$};
\draw (1,2,3) node[left]{\scriptsize{$\vertex{231}$}};
\draw (2,1,3) node[below]{\scriptsize{$\vertex{321}$}};
\draw (2,3,1) node[above]{\scriptsize{$\vertex{123}$}};
\draw (1,3,2) node[left]{\scriptsize{$\vertex{213}$}};
\draw (3,2,1) node[right]{\scriptsize{$\vertex{132}$}};
\draw (3,1,2) node[below]{\scriptsize{$\vertex{312}$}};
\end{tikzpicture}
\subcaption{Vertices are labeled by permutations}
\end{subfigure}
\caption{Moment polytope $\mu(\overline{T\cdot x})$}
\label{fig:vertex-labelling}
\end{figure}

In general, the moment polytope $\mu(\overline{T\cdot x})$ has the following property. 

\begin{proposition} \label{prop:edge_vector}
If two vertices $\mu(uB)$ and $\mu(vB)$ of $\mu(\overline{T\cdot x})$ are joined by an edge of $\mu(\overline{T\cdot x})$ for $u,v\in \Sn$, then $tu=v$ for some transposition $t$ of $\Sn$, in other words, an edge of $\mu(\overline{T\cdot x})$ is parallel to $\mathbf{e}_i-\mathbf{e}_j$ for some $1\le i<j\le n$. 
\end{proposition}

\begin{proof}
By the convexity theorem, an edge of $\mu(\overline{T\cdot x})$ is $\mu(\overline{O})$ for some $1$-dimensional $T$-orbit $O$ in $\overline{T\cdot x}$. Since $\overline{T\cdot x}$ is invariant under the $T$-action on $G/B$, $O$ is also a $1$-dimensional $T$-orbit in~$G/B$. Note that $O$ is isomorphic to $\C^*$, $\overline{O}$ is $\C P^1$, and $\overline{O}\backslash O$ consists of two $T$-fixed points, say $uB$ and~$vB$, which map to the two end points of the edge $\mu(\overline{O})$ by the moment map $\mu$. 

Let $U_u$ be the $T$-invariant chart of $G/B$ centered at $uB$ in \eqref{eq:Uu}. Then, as noted before, \eqref{eq:tij} implies that the $T$-orbit $O$ is of the form 
\[
(u+cE_{ij})B\quad \text{for some $i,j$ with $j<u^{-1}(i)$}, 
\]
where $c\in \C^*$ and $E_{ij}$ is the $n\times n$ matrix with $1$ at the $(i,j)$ entry and $0$ otherwise. Suppose that $i<u(j)$ (the essentially same argument works when $i>u(j)$). Then we apply the following elementary transformations to $(u+cE_{ij})B$:
\begin{enumerate}
\item multiply the $j$th column by $1/c$,
\item subtract the $j$th column from the $u^{-1}(i)$th column, 
\item multiply the $u^{-1}(i)$th column by $-c$, 
\end{enumerate}
and then approach $c$ to $\infty$: 
\[
\begin{pmatrix}
c&1\\
1&0\end{pmatrix}B\overset{(1)}{=}\begin{pmatrix} 1&1\\
1/c&0\end{pmatrix}B\overset{(2)}{=}\begin{pmatrix}1&0\\
1/c&-1/c\end{pmatrix}B\overset{(3)}{=}\begin{pmatrix}1&0\\
1/c&1\end{pmatrix}B\xrightarrow{c\to\infty} \begin{pmatrix}1&0\\
0&1\end{pmatrix}B,
\]
where the $2\times 2$ matrices above are intersections of the $j$th and $u^{-1}(i)$th columns $(j<u^{-1}(i))$ and the $i$th and $u(j)$th rows $(i<u(j))$. The other entries remain unchanged by the elementary transformations above. 
This shows that when $c$ approaches $\infty$, the point $(u+cE_{ij})B$ approaches the $T$-fixed point $t_{i,u(j)}uB$, where $t_{i,u(j)}$ is the transposition interchanging $i$ and $u(j)$. Since $vB=t_{i,u(j)}uB$, this implies the proposition. 
\end{proof}

Since the root system of type $A_{n-1}$ is $\Phi=\{\pm (\mathbf{e}_i-\mathbf{e}_j)\mid 1\le i<j\le n\}$, Proposition~\ref{prop:edge_vector} says that an edge vector of the moment polytope $\mu(\overline{T\cdot x})$ is parallel to a root in $\Phi$. This fact holds in any Lie type. 

\begin{theorem}[{\cite[Theorem on page xii]{BGW03_Coxeter}}]\label{thm_torus_orbit_closure_Phi_polytope}
Let $G$ be a semisimple algebraic group over $\C$, $B$ a Borel subgroup in $G$ and $T \subset B$ a maximal torus. Let $\Phi$ be the associated root system and $\mu\colon G/B\to \mathfrak{t}_{\R}^*$ a moment map, where $\mathfrak{t}_{\R}^*$ denotes the vector space dual to the Lie algebra $\mathfrak{t}_{\R}$ of the maximal compact torus of $T$. Then, for any point $x\in G/B$, an edge of the moment polytope $\mu(\overline{T\cdot x})$ is parallel to a root in $\Phi$ \textup{(}such a polytope is called a $\Phi$-polytope, or a generalized permutohedron, see Definition~\ref{def_Phi_polytope}\textup{)}. 
\end{theorem}

\begin{remark}
Partial flag varieties also have torus actions and one may consider moment maps.
We refer the reader to~\cite{GGMS87} for the combinatorial properties of moment maps of torus orbit closures in Grassmannian. 
\end{remark}


\subsection{Coxeter matroids and Gelfand--Serganova Theorem}
\label{section:Coxeter matroids}
To study the combinatorial properties of torus orbit closures, Gelfand and Serganova introduced the notion of Coxeter matroids. 
In this subsection, we recall the definition of Coxeter matroids and the characterization of Coxeter matroids in terms of polytopes by Gelfand--Serganova. It turns out that the permutations indexing the $T$-fixed points of a $T$-orbit closure in the flag variety $G/B$ is a Coxeter matroid. 

Let $W$ be a finite Coxeter group, so generators of $W$ are prescribed. The symmetric group $\Sn$ on $[n]$ with adjacent transpositions as generators is a typical example of a Coxeter group. 
For $u\in W$, let $\le^u$ denote the \emph{$u$-shifted order}, that is,  $v\le^u w$ means $u^{-1}v\le u^{-1}w$ in the Bruhat order on $W$. Note that $\le^u$ is a partial order on $W$ with $u$ as the smallest element. 
\begin{definition} \label{defi:Coxeter_matroid}
A subset $\A$ of a finite Coxeter group $W$ is called a {\it Coxeter matroid} if it satisfies the \emph{Maximality Property}: 
for any $u\in W$, there is a unique element $v\in \A$ such that $w\le^u v$ for all $w\in \A$.
\end{definition}
We note that for a finite Coxeter group $W$ in type $A$, Coxeter matroids are also called \emph{flag matroids}. See~\cite{White_matroids}. 
For more details on flag matroids and flag matroid polytopes, we refer the reader to~\cite{ABD_matroidpolytopes, Cameron_Flag_matroids, BK_Lattice_path_matroids}.

\begin{remark} \label{rem:minimality_property}
\begin{enumerate}	
\item	
Since the multiplication by the longest element of $W$ reverses the Bruhat order on~$W$, 
the Maximality Property is equivalent to the \emph{Minimality Property}: for any $u\in W$, there exists a unique element $v\in\A$ such that $v\le^u w$ for all $w\in \A$. 
\item In fact, a Coxeter matroid is defined more generally in~\cite{BGW03_Coxeter}.  Let $W_P$ be a parabolic subgroup of a finite Coxeter group $W$.  Then the Bruhat order on $W$ induces a partial order on $W/W_P$ and a subset $\A$ of $W/W_P$ is called a Coxeter matroid if it satisfies the Maximality Property above.  A Coxeter matroid in Definition~\ref{defi:Coxeter_matroid} is the case where $W_P$ is the identity subgroup and an ordinary matroid can be regarded as the case where $W=\Sn$ and $W_P$ is a maximal parabolic subgroup (see \cite[Theorem~1.3.1]{BGW03_Coxeter}).  
\end{enumerate}
\end{remark}

\begin{example}[{\cite[Example~2.2]{LMP_retractions}}]\label{example_not_Coxeter}
Let $\A=\{213, 132\}$ be a subset of $S_3$. Since $213 \not\leq 132$ and $132 \not\leq 213$, there is no element $v \in \A$ such that $w \leq^{123} v$ for all $w \in \A$. Hence $\A$ is not a Coxeter matroid. However, one can check that $\{231,321\}$ is a Coxeter matroid of $S_3$. 
\end{example}

We recall the characterization of Coxeter matroids in terms of polytopes by Gelfand--Serganova. 
As is well-known, a finite Coxeter group $W$ can be regarded as a reflection group on a vector space~$V$, where the generators of $W$ act on $V$ as reflections (see \cite[Section 5.3]{Humphreys90}). 
Let $\Phi$ be the set of roots of $W$. 
\begin{definition}\label{def_Phi_polytope}
A convex polytope $\Delta$ in $V$ is called a \emph{$\Phi$-polytope}, or a \emph{generalized permutohedron}, if every edge of~$\Delta$ is parallel to a root in $\Phi$.
\end{definition}
We note that the notion of a $\Phi$-polytope is known as generalized permutohedron~\cite{PostnikovReinerWilliams08}. They were originally defined by Edmonds~\cite{Edmonds70}, and they are also called $M$-convex sets (see~\cite{Murota03}), and polymatroids (see~\cite{Schrijver03}).

Choose a point $\nu\in V$ which is not fixed by any reflection in $W$. 
For a subset $\A$ of~$W$, we define $\Delta_\A$ to be the convex hull of the $\A$-orbit $\{w\cdot \nu\mid w\in \A\}$ of the point $\nu$ in $V$. When $\A=W$, $\Delta_W$ is called the \textit{$W$-permutohedron}\label{W-permuto-def} (see~\cite[Section 2.4]{FominReading07_root},  \cite{HLT11_permutahedra}, and \cite{ArdilaCastilloEurPostnikov20}). Two vertices $v\cdot \nu$ and $w\cdot \nu$ of~$\Delta_W$ are joined by an edge of $\Delta_W$ if and only if $v^{-1}w$ is a generator of $W$ (see~\cite[Lemma~2.13]{FominReading07_root}). Thus 
any edge of $\Delta_W$ is parallel to a root and vice versa. Therefore, the $W$-permutohedron~$\Delta_W$ is a $\Phi$-polytope and we may say that a convex polytope $\Delta$ in $V$ is a $\Phi$-polytope if every edge of $\Delta$ is parallel to an edge of $\Delta_W$. 

However, $\Delta_\A$ is not necessarily a $\Phi$-polytope unless $\A=W$. The following is a part of the Gelfand--Serganova Theorem~\cite{GelfandSerganova87} (also see~\cite[Theorem~6.3.1]{BGW03_Coxeter}).

\begin{theorem}[{Gelfand--Serganova Theorem}]\label{thm_GS}
A subset $\A$ of a finite Coxeter group $W$ is a Coxeter matroid if and only if $\Delta_\A$ is a $\Phi$-polytope. Therefore, the permutations indexing the $T$-fixed points of a $T$-orbit closure in the flag variety $G/B$ is a Coxeter matroid by Theorem~\ref{thm_torus_orbit_closure_Phi_polytope}. 
\end{theorem}
We provide the proof for the case $W = \Sn$. We refer the reader to \cite{BGW03_Coxeter} for more details. 
Before giving it, we prepare two lemmas.
Suppose that 
\[
V = \{ (a_1,\dots,a_n) \in \R^n \mid a_1+\dots+a_n = 0\} 
\]
and $\Sn$ acts on $V$ as permuting coordinates, i.e.,  $u \cdot (a_1,\dots,a_n) = (a_{u^{-1}(1)},\dots,a_{u^{-1}(n)})$, in other words, 
\[
u \cdot \left(\sum_{i=1}^{n} a_i \mathbf{e}_i \right) = \sum_{i=1}^{n} a_i \mathbf{e}_{u(i)}.
\]
We define the ordering $\preceq^u$ on $V$ by putting $x \preceq^u y$ for $x,y \in V$ if there exists non-negative constants $c_i \geq 0$ such that
\[
y-x = \sum_{i=1}^{n-1} c_i (\mathbf{e}_{u(i)} - \mathbf{e}_{u(i+1)}),
\] 
which is equivalent to saying $u^{-1} \cdot y - u^{-1} \cdot x = \sum_{i=1}^{n-1} c_i (\mathbf{e}_i - \mathbf{e}_{i+1})$ with the same coefficients~$c_i$. This means
\begin{equation}\label{eq_x_precu_y}
x \preceq^u y \iff u^{-1}\cdot x \preceq^e u^{-1}\cdot y.
\end{equation}  
\begin{lemma}[{\cite[Lemma~6.2.3]{BGW03_Coxeter}}]\label{lemma_bruhat_order_and_prec}
Suppose that $\nu = (\nu_1,\nu_2,\dots,\nu_n) \in V$ satisfies $\nu_1 < \nu_2 < \dots < \nu_n$.
If $v \leq^u w$, then $v \cdot \nu \preceq^u w \cdot \nu$.
\end{lemma}
\begin{proof}
Note that $v \leq^u w$ if and only if $u^{-1}v \leq u^{-1}w$, and $v \cdot \nu \preceq^u w \cdot \nu$ if and only if $u^{-1}v \cdot \nu \preceq^{e} u^{-1} w \cdot \nu$ (see~\eqref{eq_x_precu_y}). Hence it suffices to prove the statement in the case $u = e$. 
Suppose $v \leq w$. Then we have $v^{-1} \leq w^{-1}$ and hence 
$\{v^{-1}(1),\dots,v^{-1}(d)\}\!\!\uparrow\: \leq \{w^{-1}(1),\dots,w^{-1}(d)\}\!\!\uparrow$ for $1 \leq d \leq n$ (cf.~\cite[\S3.2]{BL20Singular}). 
This implies $\sum_{i=1}^d v^{-1}(i) \leq \sum_{i=1}^d w^{-1}(i)$ for $1 \leq d \leq n$,
so we get
\begin{equation}\label{eq_compare_Bruhat_order_and_sums}
\sum_{i=1}^d \nu_{v^{-1}(i)} \leq \sum_{i=1}^d \nu_{w^{-1}(i)} \quad \text{ for }1 \leq d \leq n.
\end{equation}

On the other hand, we have
\[
\begin{split}
w \cdot \nu - v \cdot \nu &= (\nu_{w^{-1}(1)},\dots,\nu_{w^{-1}(n)}) - (\nu_{v^{-1}(1)},\dots,\nu_{v^{-1}(n)}) \\
&=  \sum_{d=1}^{n-1}\left(\sum_{i=1}^d \nu_{w^{-1}(i)} - \sum_{i=1}^d \nu_{v^{-1}(i)} \right)(\mathbf{e}_d - \mathbf{e}_{d+1}).
\end{split}
\]
Accordingly, by~\eqref{eq_compare_Bruhat_order_and_sums}, we obtain $v \cdot \nu \preceq^e w \cdot \nu$ and we are done.
\end{proof}

The converse in Lemma~\ref{lemma_bruhat_order_and_prec} does not hold  in general.  For instance, if $v=1432$ and $w=4123$, then $v\not\le w$ but since
\[
\begin{split}
w\cdot\nu-v\cdot\nu&=(\nu_2,\nu_3,\nu_4,\nu_1)-(\nu_1,\nu_4,\nu_3,\nu_2)\\
&=(\nu_2-\nu_1)(\mathbf{e}_1-\mathbf{e}_2)+(\nu_2+\nu_3-\nu_1-\nu_4)(\mathbf{e}_2-\mathbf{e}_3)+(\nu_2-\nu_1)(\mathbf{e}_3-\mathbf{e}_4),
\end{split}
\]
we have $v\cdot\nu\preceq^e w\cdot\nu$ if $\nu_2+\nu_3\ge 0$ (note that $\nu_1+\nu_2+\nu_3+\nu_4=0$ because $\nu\in V$).  

However, the converse holds in the following special case.
\begin{lemma}[{\cite[Lemma~6.2.5]{BGW03_Coxeter}}] \label{lemm:converse}
Let $\nu$ be as in Lemma~\ref{lemma_bruhat_order_and_prec}.  
Suppose that $w\cdot \nu-v\cdot \nu$ is parellel to a root in $\Phi$.  Then $v\le^u w$ if $v\cdot\nu\preceq^u w\cdot\nu$.   
\end{lemma}
\begin{proof}
As noted in the proof of Lemma~\ref{lemma_bruhat_order_and_prec}, we may assume $u=e$.  If $v\cdot\nu\preceq^e w\cdot\nu$, then 
\[
w\cdot\nu-v\cdot\nu=c(\mathbf{e}_i-\mathbf{e}_j) 
\]
for some $c>0$ and $i<j$. This implies that $w^{-1}=v^{-1}t_{i,j}$ and $v^{-1}(i)<v^{-1}(j)$ since $c>0$, where $t_{i,j}$ denotes the transposition interchanging $i$ and $j$. Therefore $v^{-1}\le w^{-1}$ and hence $v\le w$.   
\end{proof}

\begin{proof}[Proof of Theorem~\ref{thm_GS} in type A]
We take  $\nu = (\nu_1,\nu_2,\dots,\nu_n) \in V$ which satisfies $\nu_1 < \nu_2 < \dots < \nu_n$.
We will prove that if $\A \subset \Sn$ is a Coxeter matroid, then $\Delta_{\A}$ is a $\Phi$-polytope.
Assume on the contrary that $\Delta_{\A}$ is not a $\Phi$-polytope. 
Then there exists an edge $l$ with vertices $v_1 \cdot \nu$ and $v_2 \cdot \nu$ that is not parallel to any root. Consider a linear function $f \colon V \to \R$ that is constant on $l$ and takes smaller values on the other points of~$\Delta_{\A}$. 

Since the edge $l$ is not parallel to any root, we may assume that $f$ is not vanishing on any root in~$\Phi$. Accordingly, there is a unique simple system of roots $\tilde{\alpha}_1,\dots,\tilde{\alpha}_{n-1}$ such that $f(\tilde{\alpha}_i) > 0$ for $i=1,\dots,n-1$. 
Moreover, since the group $\Sn$ acts transitively on the set of all simple root systems, there exists an element $u \in \Sn$ which sends $\{ \mathbf{e}_i - \mathbf{e}_{i+1} \mid i=1,\dots,n-1\}$ to $\{\tilde{\alpha}_1,\dots,\tilde{\alpha}_{n-1}\}$. 

For any $w \in \A \setminus \{v_1 \}$ we have $f(w \cdot \nu) \leq f(v_1 \cdot \nu)$ and the vector $w \cdot \nu - v_1 \cdot \nu$ has at least one negative coefficient with respect to $\{\tilde{\alpha}_1,\dots,\tilde{\alpha}_{n-1}\} = \{\mathbf{e}_{u(i)} - \mathbf{e}_{u(i+1)} \mid i =1,\dots,n-1\}$.
Therefore, we have $v_1 \not\preceq^u w$ and this implies $v_1 \not\leq^u w$ by Lemma~\ref{lemma_bruhat_order_and_prec}. 
Accordingly, by the Maximality Property, $v_1$ must be the maximal element with respect to the $u$-shifted order. However, by the similar argument, one can show that $v_2$ is also the maximal element with respect to the $u$-shifted order. This contradicts the Maximality Property.

Now we prove that if $\Delta_\A$ is a $\Phi$-polytope, then $\A\subset \Sn$ is a Coxeter matroid.   Assume on the contrary that $\A$ is not a Coxeter matroid.  Then there are (at least) two maximal elements $v_1$ and $v_2$ in $\A$ with respect to $\le^u$ for some $u\in\Sn$.  Let $x_1\cdot \nu,\dots,x_\ell\cdot \nu$ be the vertices of $\Delta_\A$ adjacent to the vertex $v_1\cdot\nu$ and $y_1\cdot \nu,\dots,y_m\cdot \nu$ the vertices adjacent to the vertex $v_2\cdot\nu$. Set 
\begin{equation} \label{eq:xp_yq}
\text{$\beta_p=x_p\cdot \nu-v_1\cdot\nu$ $(p=1,\dots,\ell)$\quad and\quad $\gamma_q=y_q\cdot \nu-v_2\cdot\nu$ $(q=1,\dots,m)$.}
\end{equation}
Since $\Delta_\A$ is a $\Phi$-polytope, each $\beta_p$ or $\gamma_q$ is parallel to a root in $\Phi$, and therefore either $\preceq^u  \boldsymbol{0}$ or $\succeq^u \boldsymbol{0}$.

Suppose that $\beta_p\preceq^u \boldsymbol{0}$ for all $p$.  Then $\Delta_\A$ is contained in the cone $\{\chi\in V\mid \chi\preceq^u v_1\cdot \nu\}$. 
Therefore $z\cdot \nu \preceq^u v_1\cdot\nu$ for all $z \in \A$, in particular, $v_2\cdot\nu\preceq^u v_1\cdot\nu$. 
Similarly, if $\gamma_q\preceq^u \boldsymbol{0}$ for all $q$, then 
$v_1\cdot\nu\preceq^u v_2\cdot\nu$. 
Hence, if $\beta_p\preceq^u  \boldsymbol{0}$ and $\gamma_q\preceq^u  \boldsymbol{0}$ for all $p$ and $q$, then $v_1\cdot\nu=v_2\cdot\nu$ and hence $v_1=v_2$, a contradiction.

Thus, some $\beta_p$ or $\gamma_q$ is $\succeq^u  \boldsymbol{0}$.  Suppose that $\beta_p\succeq^u  \boldsymbol{0}$ for some $p$, i.e., 
$$x_p\cdot\nu   \succeq^u v_1\cdot \nu.$$
Then $v_1\le^u x_p$ by Lemma~\ref{lemm:converse}.  This contradicts the maximality of $v_1$ because $v_1\neq x_p$ and $x_p\in \A$.  The same argument works when $\gamma_q\succeq^u \boldsymbol{0}$ for some $q$, so we deduce a contradiction in any case. Thus $\A$ is a Coxeter matroid. 
\end{proof}

\begin{example}
\label{example_polytope_and_matroid}
Using Theorem~\ref{thm_GS}, one can check that subsets $\{123,213,132,312\}$ and $\{231,321\}$ of $S_3$ are Coxeter matroids while a subset $\{213,132\}$ of $S_3$ is not a Coxeter matroid since the edge joining the vertices $s_1\cdot\nu$ and $s_2\cdot\nu$ is not parallel to any root in $\Phi=\{\pm\alpha_1,\pm\alpha_2,\pm(\alpha_1+\alpha_2)\}$ (see Figure~\ref{fig:M} and Example~\ref{example_not_Coxeter}). 
\end{example}

\begin{figure}[htb]
\centering
\begin{subfigure}[b]{.35\textwidth}
\centering
\begin{tikzpicture}[scale=.8]
\fill[fill=blue!30] (0,1)--(-{sqrt(3)}/2,1/2)--(-{sqrt(3)}/2,-1/2)--(0,-1)--cycle;
\draw[thick,gray,->] (0,0)--({sqrt(3)},-1);
\draw ({sqrt(3)},-1) node[right]{$\alpha_1$};
\draw[thick,gray,->] (0,0)--(0,2);
\draw (0,2) node[above]{$\alpha_2$};
\draw (2,0) node[right]{$s_2$};
\draw (2,0) node{$\updownarrow$};
\draw (1, {sqrt(3)+2/10}) node[right]{$s_1$};
\draw[<->] ({1+sqrt(3)/10},{sqrt(3)-1/10})--({1-sqrt(3)/10},{sqrt(3)+1/10});
\draw[thick] (1,{sqrt(3)})--(-1,-{sqrt(3)});
\draw[thick] (-1,{sqrt(3)})--(1,-{sqrt(3)});
\draw[thick] (2,0)--(-2,0);
\draw (0, 1.3) node{\footnotesize{$s_2s_1\cdot\nu$}};
\draw (0, -1.3) node{\footnotesize{$s_1\cdot\nu$}};
\draw (-1.2,-0.7) node{\footnotesize{$\nu$}};
\draw (-1.3,0.7) node{\footnotesize{$s_2\cdot\nu$}};
\filldraw[blue] (0,1) circle(2pt);
\filldraw[blue] (0,-1) circle(2pt);
\filldraw[blue] (-{sqrt(3)}/2,1/2) circle(2pt);
\filldraw[blue] (-{sqrt(3)}/2,-1/2) circle(2pt);
\draw[thick,blue] (0,1)--(-{sqrt(3)}/2,1/2)--(-{sqrt(3)}/2,-1/2)--(0,-1)--cycle;
\end{tikzpicture}
\subcaption{$\A=\{123,213,132,312\}$}\label{fig_M2}
\end{subfigure}~
\begin{subfigure}[b]{.35\textwidth}
\centering
\begin{tikzpicture}[scale=.8]
\draw[thick,gray,->] (0,0)--({sqrt(3)},-1);
\draw ({sqrt(3)},-1) node[right]{$\alpha_1$};
\draw[thick,gray,->] (0,0)--(0,2);
\draw (0,2) node[above]{$\alpha_2$};
\draw (2,0) node[right]{$s_2$};
\draw (2,0) node{$\updownarrow$};
\draw (1, {sqrt(3)+2/10}) node[right]{$s_1$};
\draw[<->] ({1+sqrt(3)/10},{sqrt(3)-1/10})--({1-sqrt(3)/10},{sqrt(3)+1/10});
\draw[thick] (1,{sqrt(3)})--(-1,-{sqrt(3)});
\draw[thick] (-1,{sqrt(3)})--(1,-{sqrt(3)});
\draw[thick] (2,0)--(-2,0);
\draw (1.2,-0.7) node{\footnotesize{$s_1s_2\cdot\nu$}};
\draw (1.3,0.7) node{\footnotesize{$s_1s_2s_1\cdot\nu$}};
\filldraw[blue] ({sqrt(3)/2},1/2) circle(2pt);
\filldraw[blue] ({sqrt(3)/2},-1/2) circle(2pt);
\draw[thick,blue] ({sqrt(3)/2},1/2)--({sqrt(3)/2},-1/2);
\end{tikzpicture}
\subcaption{$\A=\{231,321\}$}\label{fig_M3}
\end{subfigure}~
\begin{subfigure}[b]{.3\textwidth}
\centering
\begin{tikzpicture}[scale=.8] 
\draw[thick,gray,->] (0,0)--({sqrt(3)},-1);
\draw ({sqrt(3)},-1) node[right]{$\alpha_1$};
\draw[thick,gray,->] (0,0)--(0,2);
\draw (0,2) node[above]{$\alpha_2$};
\draw (2,0) node[right]{$s_2$};
\draw (2,0) node{$\updownarrow$};
\draw (1, {sqrt(3)+2/10}) node[right]{$s_1$};
\draw[<->] ({1+sqrt(3)/10},{sqrt(3)-1/10})--({1-sqrt(3)/10},{sqrt(3)+1/10});
\draw[thick] (1,{sqrt(3)})--(-1,-{sqrt(3)});
\draw[thick] (-1,{sqrt(3)})--(1,-{sqrt(3)});
\draw[thick] (2,0)--(-2,0);
\draw (0, -1.3) node{\footnotesize{$s_1\cdot\nu$}};
\draw (-1.3,0.7) node{\footnotesize{$s_2\cdot\nu$}};
\filldraw[red] (0,-1) circle(2pt);
\filldraw[red] (-{sqrt(3)}/2,1/2) circle(2pt);
\draw[thick,red] (-{sqrt(3)}/2,1/2)--(0,-1);
\end{tikzpicture}
\subcaption{$\A=\{213,132\}$}\label{fig_M1}
\end{subfigure}
\caption{Examples of $\Delta_{\A}$} \label{fig:M}
\end{figure}
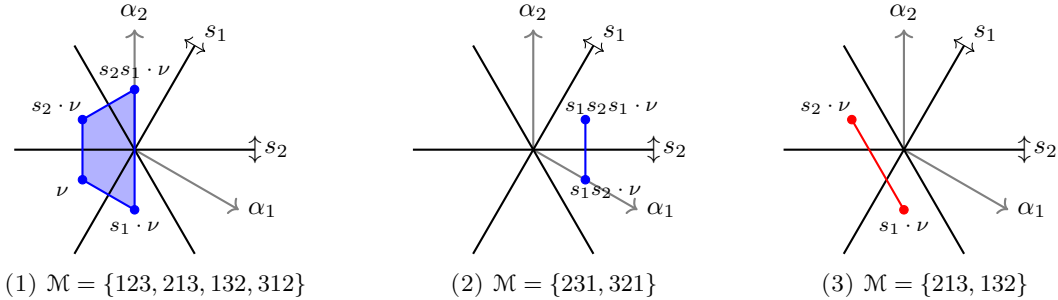

By Theorem~\ref{thm_GS}, torus orbit closures provide Coxeter matroids. A Coxeter matroid~$\A$ of $W$ is said to be \emph{representable} (over $\C$) if $\A$ can be realized as the $T$-fixed point set of a $T$-orbit closure in the flag variety $G/B$, that is,  there exists a point $x\in G/B$ such that $\A=(\overline{T\cdot x})^T$, where $(G/B)^T$ is identified with $W$. See~\cite[\S 1.7.5, \S 3.6.2, \S 3.10.3]{BGW03_Coxeter}. 
A computer check shows that any Coxeter matroid of $\Sn$ is representable when $n\le 4$. On the other hand, there exists a non-representable Coxeter matroid of $\Sn$ when $n=7$, which we explain in the following.

Let 
\[
A = \{\{1,2,4\}, \{1,3,5\}, \{1,6,7\}, \{2,3,6\}, \{2,5,7\}, \{3,4,7\}, \{4,5,6\}\}
\] 
be a collection of subsets of $\{1,2,\ldots,7\}$. Each subset corresponds to one of the six lines or the circle in Figure~\ref{fig_Fano}. Define a subset $\A$ of $S_7$ by
\[
\A = \{w \in S_7 \mid \{w(1), w(2), w(3)\} \notin A\}.
\] 
Then $\A$ is a Coxeter matroid (see \cite{LMP_retractions}).

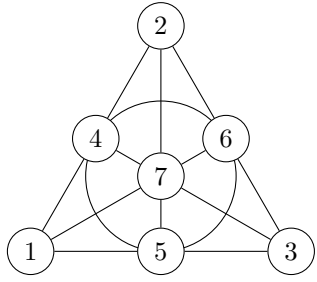
\begin{figure}[H]
\begin{tikzpicture}
\draw (0:0) circle (1cm);
\node[fill=white, draw, circle] at (210:2) (1) {$1$};
\node[fill=white, draw, circle] at (90:2) (2) {$2$};
\node[fill=white, draw, circle] at (330:2) (3) {$3$};
\node[fill=white, draw, circle] at (150:1) (4) {$4$};
\node[fill=white, draw, circle] at (270:1) (5) {$5$};
\node[fill=white, draw, circle] at (30:1) (6) {$6$} ;
\node[fill=white, draw, circle] at (0:0) (7) {$7$};
\draw (1)--(5);
\draw (5)--(3);
\draw (3)--(6);
\draw (6)--(2);
\draw (2)--(4);
\draw (4)--(1);
\draw(1)--(7);
\draw (7)--(6);
\draw(5)--(7);
\draw (7)--(2);
\draw (3)--(7);
\draw (7)--(4);
\end{tikzpicture}
\caption{The Fano plane}\label{fig_Fano}
\end{figure}

Suppose that $\A$ is the $T$-fixed point set of a $T$-orbit closure in ${\rm GL}_7(\C)/B$. Then, by Proposition~\ref{proposition_GS_fixed_points}, there is an element $x$ in ${\rm GL}_7(\C)/B$ such that 
\[
\A = \{w\in S_7\mid \{w(1),\ldots,w(d)\}\!\!\uparrow\: \in I_d(x)\text{ for all }1\leq d\leq 7\}. 
\]
However, it is shown in~\cite[\S16]{Whitney} and easy to check that there is no $7\times 3$ matrix of rank~$3$ whose three rows $v_{j_1}, v_{j_2}, v_{j_3}$ are linearly independent if and only if $\{j_1, j_2, j_3\} \notin A$. Therefore, $\A$ cannot be obtained as the $T$-fixed point set of any $T$-orbit closure in ${\rm GL}_7(\C)/B$.


\subsection{Description of the fan of a torus orbit closure}\label{subsec:Description of the fan of a torus orbit closure}

For a point $x\in G/B$, the closure $Y\colonequals \overline{T\cdot x}$ of its $T$-orbit $T\cdot x$ is a (possibly non-normal) toric variety in $G/B$. 
Although Proposition~\ref{proposition_GS_fixed_points} provides how to find the fixed point set $Y^T$ and the convexity theorem leads us to the moment polytope $\mu(Y)$, it is hard to describe the fan of $Y$ from $\mu(Y)$.
In this subsection and the next subsection, we identify $(G/B)^T=W$ for simplicity, so $Y^T\subset W$.
We define a retraction~$\gret_Y$ (called geometric retraction) of the Weyl group $W$ of $G$ onto the $T$-fixed point set $Y^T \subset W$ of~$Y$ by using the Orbit-Cone correspondence in toric variety and describe the fan of $Y$ using the retraction.

We think of the Lie algebra $\mathfrak{t}_{\R}$ of the maximal compact torus of $T$ as $\Hom(\C^{\ast},T)\otimes \R$ and the lattice $\mathfrak{t}_\Z$ of $\mathfrak{t}_{\R}$ as $\Hom(\C^{\ast},T)$, where $\Hom(\C^*,T)$ denotes the group of algebraic homomorphisms from $\C^*=\C\backslash\{0\}$ to $T$, i.e., the one parameter subgroup of $T$ (see~\cite[\S 1]{CLS11Toric}). The vector space $\mathfrak{t}_{\R}^{\ast}$ dual to $\mathfrak{t}_{\R}$ can be thought of as $\Hom(T,\C^{\ast})\otimes \R$ and the set $\Phi$ of roots of $G$ is a finite subset of $\Hom(T,\C^{\ast})=\mathfrak{t}_\Z^{\ast}$. The Weyl group $W$ of $G$ acts on $\mathfrak{t}_{\R}$ as the adjoint action and on its dual space $\mathfrak{t}_{\R}^{\ast}$ as the coadjoint action, i.e.,  
\[
\text{$(w\cdot f)(x)\colonequals f(\Ad_{w^{-1}}(x))$\quad for $w\in W$, $f\in \mathfrak{t}_{\R}^{\ast}$ and $x\in \mathfrak{t}_{\R}$.}
\]
The Borel subgroup $B$ determines the set $\Phi^+$ of positive roots and the tangent space $T_{uB}(G/B)$ of~$G/B$ at a $T$-fixed point $u\in W$ decomposes as follows: 
\begin{equation} \label{eq:tangent}
\quad T_{uB}(G/B)=\bigoplus_{\alpha\in \Phi^+}\mathfrak{g}_{-u(\alpha)}\quad \text{as $T$-modules for $u\in W$,}
\end{equation}
where $\mathfrak{g}_{\alpha}$ denotes the root space of $\alpha$, which is the eigenspace of $\mathfrak{g}$ for $\alpha\in \Phi$ (see, for example,~\cite[\S3]{GHZ06_GKM}). 
We recall that for $G = \GL_n(\C)$, the set $\Phi$ of roots is given by 
\[
\Phi = \{ \mathbf{e}_i - \mathbf{e}_j \mid 1 \leq i \neq j \leq n\}.
\]
Here, $\mathbf{e}_i \in \mathfrak{t}_{\R}^{\ast}$ is the linear functional sending $(a_1,
\dots,a_n) \in \mathfrak{t}_{\R}$ to $a_i$ for $i=1,\dots,n$. 
For each root $\mathbf{e}_i - \mathbf{e}_j$, the root space $\mathfrak{g}_{\mathbf{e}_i - \mathbf{e}_j}$ is spanned by the matrix $E_{ij}$, which is the $n\times n$ matrix with $1$ at the $(i,j)$ entry and $0$ otherwise. Moreover, the set $\Phi^+ = \{ \mathbf{e}_i - \mathbf{e}_j \mid 1 \leq i < j \leq n\}$ of positive roots corresponds to the Borel subgroup $B$ consisting of upper triangular matrices in $G$. The set~$\Delta$ of simple roots is $\Delta = \{\mathbf{e}_i - \mathbf{e}_{i+1} \mid 1 \leq i \leq n-1\}$. See, for example,~\cite{FultonHarris} for more details.

For each $u\in W$, define 
\begin{equation*} \label{eq:cone}
C(u) \colonequals \{ \lambda\in \mathfrak{t}_{\R}\mid \langle u(\alpha),\lambda\rangle \le 0 \quad\text{for all simple roots $\alpha$}\}
\end{equation*}
where $\langle \ ,\ \rangle$ denotes the natural pairing between $\mathfrak{t}_{\R}^{\ast}$ and $\mathfrak{t}_{\R}$.
The interiors $\Int(C(u))$ of the cones $C(u)$ above form the Weyl chambers. 
The identity in~\eqref{eq:tangent} implies that \lee{$T$-weights on the tangent spaces at fixed points are given by roots. This implies that }
if $\lambda \in \mathfrak{t}_\Z$ satisfies $\langle \alpha,\lambda\rangle\not=0$ for any $\alpha\in \Phi$, then 
\begin{equation} \label{eq:lambda}
(G/B)^{\lambda(\C^{\ast})}=(G/B)^T.
\end{equation}

\begin{example}
Let $G = \GL_3(\C)$ and $u = 231 \in S_3$. Following the convention that we adopt, we have two simple roots, $\alpha_1 = \mathbf{e}_1 - \mathbf{e}_2$ and $\alpha_2 = \mathbf{e}_2 - \mathbf{e}_3$. 
Then we have 
\[
u(\alpha_1) = \mathbf{e}_{u(1)} - \mathbf{e}_{u(2)} = \mathbf{e}_2 - \mathbf{e}_3,\quad 
u(\alpha_2) = \mathbf{e}_{u(2)} - \mathbf{e}_{u(3)} = \mathbf{e}_3 - \mathbf{e}_1.
\]
Accordingly, we obtain
\[
\begin{split}
C(u) &= \{\lambda = (\lambda_1, \lambda_2, \lambda_3) \mid \langle \mathbf{e}_2 - \mathbf{e}_3, \lambda \rangle \leq 0, \langle  \mathbf{e}_3 - \mathbf{e}_1, \lambda \rangle \leq 0 \} \\
&= \{ \lambda = (\lambda_1, \lambda_2, \lambda_3) \mid \lambda_2 \leq \lambda_3 \leq \lambda_1 \}.
\end{split}
\]
\end{example}

For each $w\in W=(G/B)^T$, we choose an element $\lambda_w\in \Int (C(w))\cap \mathfrak{t}_\Z$ and define 
\[
S_w\colonequals \left\{ x\in G/B \,\,\middle|\,\, \lim_{t\to 0} \lambda_w(t)\cdot x=w \right\},
\]
which is independent of the choice of $\lambda_w$. Then $S_w$ is a $T$-invariant affine open subset of $G/B$ and isomorphic to $T_{wB}(G/B)$ as a $T$-variety (see~\cite{BB73_some}). 

\begin{proposition}[{\cite[Proposition~3.1]{LMP_retractions}}]\label{lemm:limit_Y}
Let $x$ be a point of $G/B$ and $Y=\overline{T\cdot x}$. For any $u\in W$ and $\lambda_u \in \Int (C(u))\cap \mathfrak{t}_\Z$, the limit point ${\lim_{t\to 0}\lambda_u(t)\cdot x}$ is an element of~$Y^T$ depending only on $u$ and $Y$. Furthermore, if $u\in Y^T$, then $\lim_{t\to 0}\lambda_u(t)\cdot x=u$.
\end{proposition}

\begin{proof}
Since $Y$ is closed, the limit point ${\lim_{t\to 0}\lambda_u(t)\cdot x}$ belongs to $Y$ and clearly remains fixed under the action of $\lambda_u(\C^{\ast})$. Therefore, the limit point is indeed in $Y^T$ by~\eqref{eq:lambda}. Denote the limit point by $w$. Since $\lambda_u(t)\cdot x \in S_{w}$ and $S_{w}$ is $T$-invariant, $x$ belongs to $S_{w}$. Moreover, because $S_w$ is isomorphic to $T_{wB}(G/B)$ as a $T$-variety, it follows from \eqref{eq:tangent} that ${\lim_{t\to 0}\lambda_u(t)\cdot x}$ is independent of the choice of $\lambda_u\in \Int (C(u))\cap \mathfrak{t}_\Z$. 

If $u\in Y^T$, then $x$ belongs to $S_u$ because otherwise $Y=\overline{T\cdot x}$ does not contain $u$ (note that $S_u$ is a $T$-invariant open subset of $G/B$). Therefore, we obtain $\lim_{t\to 0}\lambda_u(t)\cdot x=u$.
\end{proof}

By Proposition~\ref{lemm:limit_Y}, the map $\gret_{Y}\colon W\to Y^T\subset W$ defined by 
\begin{equation} \label{eq:limit_Y}
\gret_{Y}(u)\colonequals \lim_{t\to 0}\lambda_u(t)\cdot x
\end{equation}
is a retraction of $W$ onto $Y^T$, which we call a \emph{geometric retraction}. 

\begin{example}
\label{example_fan_Y_SL3}
Take a point $x= \begin{pmatrix}
\alpha & 1 & 0 \\ \beta & 0 & 1 \\ 1 & 0 & 0
\end{pmatrix}B\in {\rm GL}_3(\C)/B$, where $\alpha,\beta\in \C^*$. 
For $Y = \overline{T \cdot x}$, we have $Y^T=\{123,132,213,312\}$ as observed in Example~\ref{example_fixed_points_3}.
Choose an element $\lambda = (\lambda_1, \lambda_2, \lambda_3) \in \Int(C(231)) \cap \mathfrak{t}_{\Z}$, that is,  $\lambda_2 < \lambda_3 < \lambda_1$. 
Since 
\begin{equation}\label{eq_computation_x_in_Auw}
x= \begin{pmatrix}
\alpha\beta^{-1} & 1 & 0 \\ 1 & 0 & 1 \\ \beta^{-1} & 0 & 0
\end{pmatrix}B 
= \begin{pmatrix}
\alpha\beta^{-1} & 1 & -\alpha\beta^{-1} \\ 1 & 0 & 0 \\ \beta^{-1} & 0 & -\beta^{-1}
\end{pmatrix}B=\begin{pmatrix}\alpha\beta^{-1}&1&0\\
1&0&0\\
\beta^{-1}&0&1\end{pmatrix}B,
\end{equation}
we have 
\[
\begin{split}
\lambda(t) \cdot x
&= \begin{pmatrix}
t^{\lambda_1} & 0 & 0 \\
0 & t^{\lambda_2} & 0 \\
0 & 0 & t^{\lambda_3}
\end{pmatrix} \cdot
\begin{pmatrix}
\alpha\beta^{-1} & 1 & 0 \\ 1 & 0 & 0 \\ \beta^{-1} & 0 & 1
\end{pmatrix}B \\
&= \begin{pmatrix}
t^{\lambda_1}\alpha\beta^{-1} & t^{\lambda_1} & 0 \\ t^{\lambda_2} & 0 & 0\\ t^{\lambda_3}\beta^{-1} & 0 & t^{\lambda_3}
\end{pmatrix}B \\
&= \begin{pmatrix}
t^{\lambda_1 - \lambda_2}\alpha\beta^{-1} & 1 & 0 \\
1 & 0 & 0 \\
t^{\lambda_3 - \lambda_2}\beta^{-1} & 0 & 1
\end{pmatrix}B 
\stackrel{t \to 0}{\longrightarrow}
\begin{pmatrix}
0 & 1 & 0 \\ 1 & 0 & 0 \\ 0 & 0 & 1
\end{pmatrix}B\qquad(\because \lambda_2<\lambda_3<\lambda_1).
\end{split}
\]
Therefore, we get $\gret_Y(231)=213$. By a similar computation, we obtain Table~\ref{table_ug_1}.
\begin{table}[htb]
\centering
\begin{tabular}{c|cccccc}
\toprule 
$u$ & $123$ & $213$ & $231$ & $321$ & $312$ & $132$ \\
\midrule 
$\gret_Y(u)$ for $Y$ & $123$ & $213$ & $213$ & $312$ & $312$ & $132$ \\
\bottomrule
\end{tabular}

\smallskip
\caption{$\gret_Y(u)$ for $Y = \overline{T\cdot x}$ in Example~\ref{example_fan_Y_SL3}.}\label{table_ug_1}
\end{table}
\end{example}

\vspace{-.5cm}

The following corollary follows from Proposition~\ref{lemm:limit_Y} and the Orbit-Cone correspondence of toric varieties (see \cite[Proposition 3.2.2 \& Theorem 3.A.5]{CLS11Toric}).

\begin{corollary}[{\cite[Corollary 3.7]{LM2020}}] \label{coro:fan_of_Y}
The maximal cone $C_Y(y)$ corresponding to $y\in Y^T$ in the fan of \textup{(}the normalization of\textup{)} $Y$ is given by $\bigcup_{u\in(\gret_Y)^{-1}(y)}C(u)$. 
\end{corollary}

\begin{remark}[{\cite[Remark~3.4]{LMP_retractions}}]\label{rmk_normal}
\begin{enumerate}
\item Since the action of $T$ on $Y$ is not effective, the ambient space of the fan of $Y$ is the quotient of $\mathfrak{t}_{\R}$ by the subspace $\Hom(\C^{\ast}, T_Y) \otimes \R$, where $T_Y$ is the toral subgroup of $T$ which fixes $Y$ pointwise. Therefore, to be precise, we need to project the cones $C_Y(y)$ to this quotient space in the corollary above. 

\item When $G$ is of type $A_n$, $D_4$, or $B_2$, every $T$-orbit closure in $G/B$ is normal (\cite[Proposition~4.8]{ca-ku2000}) while when $G = G_2$, non-normal torus orbit closures exist (see~\cite[Example~6.1]{ca-ku2000} and references therein). 
\end{enumerate}
\end{remark}

\begin{example}\label{exam:w312}
Let $Y$ be the torus orbit closure in Example~\ref{example_fan_Y_SL3}. Corollary~\ref{coro:fan_of_Y} together with Table~\ref{table_ug_1} shows that the fan of $Y$ consists of four maximal cones:
\[
C({123}),\quad C({132}),\quad C({213}) \cup C({231}),\quad C({312}) \cup C({321}).
\]
See Figure~\ref{fig_fan_312}. 
{Here, the ambient space of the fan of $Y$ is the quotient space $\R^3/\R(1,1,1)$ and the identification $\R^3/\R(1,1,1)  \stackrel{\cong}{\longrightarrow} \R^2$ given by $[a_1,a_2,a_3] \mapsto (a_1-a_2, a_2-a_3)$ is used in Figures~\ref{fig_fan_Fl3} and~\ref{fig_fan_312}. For instance, the cone $C(213)$ consists of points $[a_1,a_2,a_3]$ satisfying $a_2 \leq a_1 \leq a_3$, so that it corresponds to the set of points $\{(b_1,b_2)\in\R^2 \mid b_1 \geq 0, b_1+b_2 \leq 0\}$ under the identification.
}

\end{example}
\begin{figure}[ht]
\begin{minipage}{0.55\textwidth}
\centering
\begin{tikzpicture}
\draw(-2,0)--(2,0);
\draw(0,-2)--(0,2);
\draw(-2,2)--(2,-2);
\node at (-1,-1) {\small $C(123)$};
\node at (0.6,-1.3) {\small $C(213)$};
\node at (1.3,-0.5) {\small $C(231)$};
\node at (1,1) {\small $C(321)$};
\node at (-0.6,1.3) {\small $C(312)$};
\node at (-1.3,0.5) {\small $C(132)$};

\end{tikzpicture}
\caption{Cones $\Cv$ for $v \in {S}_3$} 
\label{fig_fan_Fl3}
\end{minipage}~
\begin{minipage}{0.45\textwidth}
\centering
\begin{tikzpicture}

\filldraw[fill=yellow!40!white, draw opacity = 0] (-2,2)--(2,2)--(2,0)--(0,0)--cycle;
\filldraw[fill=green!10!white, draw opacity = 0] (-2,2)--(0,0)--(-2,0)--cycle;
\filldraw[fill=blue!10!white, draw opacity = 0] (-2,0)--(0,0)--(0,-2)--(-2,-2)--cycle;
\filldraw[fill=purple!10!white, draw opacity = 0] (0,0)--(2,0)--(2,-2)--(0,-2)--cycle;

\draw (-2,0)--(2,0);
\draw (0,0)--(0,-2);
\draw (-2,2)--(0,0);
\draw[dotted] (0,2)--(0,0);
\draw[dotted] (0,0)--(2,-2); 
\node at (-1,-1) {\small $C(123)$};
\node at (1.3,-1) {\small $C(213) \cup C(231)$};
\node at (0.5,1) {\small $C(312) \cup C(321)$};
\node at (-1.3,0.5) {\small $C(132)$};

\end{tikzpicture}
\caption{The fan of $Y$}
\label{fig_fan_312}
\end{minipage}
\end{figure}

Now we reformulate the geometric retraction $\gret_Y$ using a Bruhat decomposition of $G/B$. It makes the meaning of the geometric retraction more transparent. 

For $u \in W$, we set $B_u \colonequals u B^{-}u^{-1}$, where $B^-$ is the opposite Borel subgroup $w_0Bw_0$ and $w_0$ is the longest element of $W$. The Lie algebra of the Borel subgroup $B_u$ is given as follows:
\begin{equation*}\label{eq_Lie_Bu}
\Lie(B_u) = \mathfrak{t} \oplus \bigoplus_{\alpha \in \Phi^+} \mathfrak{g}_{-u(\alpha)}. 
\end{equation*}

With respect to $B_u$, we obtain the following Bruhat decomposition:
\begin{equation*} \label{eq:decomposition_of_G/B}
G/B = \bigsqcup_{w \in W} B_u \cdot w B/B,
\end{equation*}
which follows from the Bruhat decomposition $G/B = \bigsqcup_{w \in W} BwB/B$ with respect to $B$ (see~\cite[\S8.3]{Springer_linear_algebraic_groups} for Bruhat decompositions). 
Indeed, since $G=uw_0G$ and $G/B=\bigsqcup_{w\in W}Bw_0u^{-1}wB/B$, we have
\[
G/B=uw_0G/B=\bigsqcup_{w\in W}uw_0Bw_0u^{-1}wB/B=\bigsqcup_{w\in W}B_uwB/B.
\]
We set 
\begin{equation} \label{eq:Auw}
A^u_w \colonequals B_u \cdot w B/B=u\cdot B^{-}u^{-1}wB/B.
\end{equation}
Note that $A^{w_0}_w$ is the Schubert cell $BwB/B$ and $A^e_w$ is the opposite Schubert cell $B^- w B/B$, where $e$ denotes the identity element of $W$. Similarly to (opposite) Schubert cells, we have 
\begin{equation}\label{eq_closure_of_Auw}
\overline{A^u_w} = \bigsqcup_{ w \leq^u v \leq^u u w_0} A^u_v \quad\text{and hence $(\overline{A^u_w})^T = \{  v\in W \mid w \leq^u v \leq^u uw_0\}$. }
\end{equation}
Indeed, since $\overline{B^- wB/B}$ is the disjoint union of opposite Schubert cells indexed by elements $z\in W$ satisfying $w \leq z \leq w_0$, it follows from \eqref{eq:Auw} that 
\[
\begin{split}
\overline{A^u_w} &= u \cdot \overline{B^{-}u^{-1}wB/B} \\
&= u \cdot \bigsqcup_{ u^{-1}w \leq z \leq w_0} B^- z B/B \\
&= \bigsqcup_{u^{-1}w\le u^{-1}(uz) \le u^{-1}(uw_0)}u\cdot B^-u^{-1}(uz)B/B\\
&= \bigsqcup_{w \leq^u v \leq^u uw_0} A^u_v,
\end{split}
\]
which shows \eqref{eq_closure_of_Auw}.

\begin{proposition}[{\cite[Proposition~3.5]{LMP_retractions}}]\label{prop1}
Let $x$ be a point of $G/B$ and $Y=\overline{T\cdot x}$. Then $x \in A^u_w$ if and only if $\gret_Y(u) = w$.
\end{proposition}

\begin{example} 
Take $u = 231\in S_3$. Then the corresponding Borel subgroup $B_u$ is given by 
\[
B_u = u B^- u^{-1} 
= \begin{pmatrix}
0 & 0 & 1 \\1 & 0 & 0 \\ 0 & 1 & 0
\end{pmatrix}
\begin{pmatrix} \star & 0 & 0 \\ * & \star & 0 \\ * & * & \star \end{pmatrix}
\begin{pmatrix} 0 & 1 & 0 \\ 0 & 0 & 1 \\ 1 & 0 & 0 \end{pmatrix}
= \begin{pmatrix}
\star & * & * \\ 0 & \star & 0 \\ 0 & * & \star
\end{pmatrix}.
\]
Here, $* \in \C$ and $\star \in \C^{\ast}$.
For $w = 213\in S_3$, the cell $A^u_w=A^{231}_{213}$ is given by 
\[
B_u w B/B = \begin{pmatrix}
\star & * & * \\ 0 & \star & 0 \\ 0 & * & \star
\end{pmatrix} \begin{pmatrix}
0 & 1 & 0 \\ 1 & 0 & 0 \\ 0 & 0 & 1
\end{pmatrix} B 
= \begin{pmatrix}
* & \star & * \\ \star & 0 & 0 \\ * & 0 & \star 
\end{pmatrix}B
= \begin{pmatrix}
* & 1 & 0 \\ 1 & 0 & 0 \\ * & 0 & 1
\end{pmatrix} B.
\]
Therefore, 
$x = \begin{pmatrix} \alpha& 1 & 0 \\ \beta & 0 & 1 \\ 1 & 0 & 0 \end{pmatrix}B\in {\rm GL}_3(\C)/B$ 
in Example~\ref{example_fixed_points_3} is contained in $A^{231}_{213}$ by \eqref{eq_computation_x_in_Auw}. 
On the other hand, $\gret_Y(231) = 213$ for $Y = \overline{T\cdot x}$ as observed in Example~\ref{example_fixed_points_3}. 
\end{example}


\subsection{Retractions and metric on finite Coxeter groups}\label{subsec:Retractions and metric on finite Coxeter groups}

For a $T$-orbit closure $Y=\overline{T\cdot x}$ in the flag variety $G/B$, we defined the geometric retraction 
\[
\gret_Y\colon W\to Y^T\subset W
\]
in \eqref{eq:limit_Y} by looking at the limit points of the trajectory of $x$ by $1$-parameter subgroups. Note that $Y^T$ is a Coxeter matroid by Theorem~\ref{thm_GS}. 

In general, 
for a Coxeter matroid $\A$ of a Coxeter group $W$, there is a unique $\le^u$-minimal element in $\A$ for any element $u$ of $W$ (see Remark~\ref{rem:minimality_property}).
If we denote the minimal element by $\mret_{\A}(u)$, then we obtain a map
\[
\mret_{\A}\colon W\to \A\subset W.
\]
One can easily check that $\mret_{\A}$ is a retraction on $W$ and we call $\mret_{\A}$ a \emph{matroid retraction}.\footnote{In \cite{BGW03_Coxeter}, a map $\mu\colon W\to W$ satisfying the inequality $\mu(u)\leq^v\mu(v)$ for all $u,v\in \A$ is called a \emph{matroid map}. For a Coxeter matroid $\A$, the map $W\to W$ sending $u$ to the $\leq^u$-maximal element of $\A$ is a matroid map. Note that our matroid retraction satisfies the opposite inequality $\mret_\A(u)\geq^v \mret_\A(v)$ for all $u,v\in\A$.}

\begin{theorem}[{\cite[Theorem~3.7]{LMP_retractions}}]\label{theo:torus-coxeter}
Let $G$ be a semisimple algebraic group over $\C$, $B$ a Borel subgroup of $G$, and $T$ a maximal torus of $G$ contained in $B$. Then $\gret_Y=\mret_{Y^T}$ for any $T$-orbit closure $Y$ in $G/B$. 
\end{theorem}

\begin{proof}
Let $x$ be a point of $G/B$ such that $Y=\overline{T\cdot x}$. For $u\in W$, let $\gret_{Y}(u)=w$. By Proposition~\ref{prop1}, $T\cdot x \subseteq A^{u}_{w}$ and hence $Y=\overline{T\cdot x}\subseteq \overline{A^{u}_{w}}$. Therefore, we have 
\[
Y^T\subset \overline{A^u_w}^T=\{v\in W \mid w\le^u v \le^u uw_0\},
\] 
where the equation above follows from \eqref{eq_closure_of_Auw}. Hence, $w$ is the unique $\le^u$-minimal element in $Y^T$, proving $\gret_Y=\mret_{Y^{T}}$.
\end{proof}

A finite Coxeter group $W$ has a metric $d$ defined by 
\begin{equation*} \label{eq:metric}
d(v,w)\colonequals \ell(v^{-1}w)=\ell(w^{-1}v) \qquad\text{for $v, w\in W$}
\end{equation*}
where $\ell(\ )$ denotes the length function on $W$. Note that the metric $d$ is invariant under the left multiplication of $W$. For a subset $\A$ of $W$, we define 
\begin{equation*} 
d(v,\A)\colonequals \min\{d(v,w)\mid w\in \A\}.
\end{equation*}

The metric $d$ can be interpreted in terms of the $W$-permutohedron $\Delta_W$ as follows. 
As mentioned in Subsection~\ref{section:Coxeter matroids}, the $W$-permutohedron $\Delta_W$ is the convex hull of the $W$-orbit of a generic point $\nu$ in the vector space $V$, 
and two vertices $v\cdot \nu$ and $w\cdot \nu$ are joined by an edge in $\Delta_W$ if and only if 
$v^{-1}w$ is a generator of $W$. 
Therefore, if we identify $w\cdot \nu$ with $w\in W$, then the distance $d(v,w)$ can be thought of as the minimum length of the paths in $\Delta_W$ connecting $v$ and $w$ through edges of $\Delta_W$. In other words, the metric $d$ is the graph metric on the graph obtained as the $1$-skeleton of $\Delta_W$. 
For example, if $W=S_4$ and $(v,w)=(1243,3214)$, then $v^{-1}w=4213$ and hence $d(v,w)=\ell(4213)=4$. Figure~\ref{fig:length} shows a minimal-length path joining $v$ and $w$ in $\Delta_{S_4}$. In Figure~\ref{fig:length}, the vertex $v\cdot\nu$ is labeled by $\vertex{v}$ for each $v\in S_4$. 
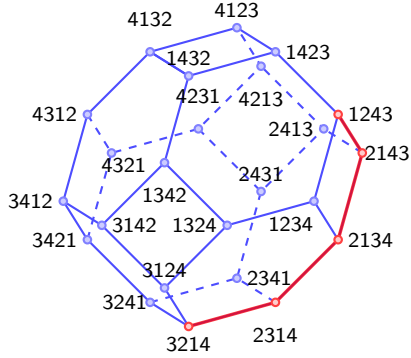
\begin{figure}[H]
\begin{tikzpicture}[scale=5]
\tikzset{every node/.style={draw=blue!50,fill=blue!20, circle, thick, inner sep=1pt,font=\footnotesize}}
\tikzset{red node/.style = {fill=red!20!white, draw=red!75!white}}
\tikzset{red line/.style = {line width=0.3ex, red, nearly opaque}}

\coordinate (4231) at (1/3, 1/2, 1/6); 
\coordinate (2413) at (2/3, 1/2, 1/6); 
\coordinate (1243) at (5/6, 2/3, 1/2); 
\coordinate (2143) at (5/6, 1/2, 1/3); 
\coordinate (2134) at (5/6, 1/3, 1/2); 
\coordinate (1423) at (2/3, 5/6, 1/2); 
\coordinate (3142) at (1/3, 1/2, 5/6); 
\coordinate (1324) at (2/3, 1/2, 5/6); 
\coordinate (1234) at (5/6, 1/2, 2/3); 
\coordinate (1342) at (1/2, 2/3, 5/6); 
\coordinate (4123) at (1/2, 5/6, 1/3); 
\coordinate (4213) at (1/2, 2/3, 1/6); 
\coordinate (1432) at (1/2, 5/6, 2/3); 
\coordinate (4132) at (1/3, 5/6, 1/2); 
\coordinate (2314) at (2/3, 1/6, 1/2); 
\coordinate (3214) at (1/2, 1/6, 2/3); 
\coordinate (3124) at (1/2, 1/3, 5/6); 
\coordinate (3241) at (1/3, 1/6, 1/2); 
\coordinate (2341) at (1/2, 1/6, 1/3); 
\coordinate (2431) at (1/2, 1/3, 1/6);
\coordinate (3421) at (1/6, 1/3, 1/2); 
\coordinate (4321) at (1/6, 1/2, 1/3); 
\coordinate (3412) at (1/6, 1/2, 2/3); 
\coordinate (4312) at (1/6, 2/3, 1/2);

\draw[thick, draw=blue!70] (1432)--(4132)--(4123)--(1423)--cycle;
\draw[thick, draw=blue!70] (4132)--(1432)--(1342)--(3142)--(3412)--(4312)--(4132);
\draw[dashed, thick, draw=blue!70] (4312)--(4321)--(4231)--(4213)--(4123);
\draw[dashed, thick, draw=blue!70] (3421)--(4321);
\draw[thick, draw=blue!70] (1342)--(1324)--(3124)--(3142);
\draw[dashed, thick, draw=blue!70] (4231)--(2431)--(2413)--(4213);
\draw[thick, draw=blue!70] (1423)--(1243)--(2143);
\draw[dashed, thick, draw=blue!70] (2143)--(2413);
\draw[thick, draw=blue!70] (1324)--(1234)--(1243);
\draw[thick, draw=blue!70] (1234)--(2134)--(2143);
\draw[thick, draw=blue!70] (2314)--(2134);
\draw[dashed, thick, draw=blue!70] (2314)--(2341)--(2431);
\draw[thick, draw=blue!70] (3412)--(3421)--(3241)--(3214)--(3124);
\draw[thick, draw=blue!70] (3214)--(2314);
\draw[dashed,thick, draw=blue!70] (3241)--(2341);

\draw[red line] (1243)--(2143)--(2134)--(2314)--(3214);

\node [label = {[label distance = 0cm] below left:$\vertex{1234}$}] at (1234) {};
\node[label = {[label distance = 0cm]right:$\vertex{1243}$}, red node] at (1243) {};
\node[label = {[label distance = 0cm]left:$\vertex{1324}$}] at (1324) {};
\node[label = {[label distance = 0cm]below:$\vertex{1342}$}] at (1342) {};
\node [label = {[label distance = 0cm]right:$\vertex{1423}$}] at (1423) {};
\node[label = {[label distance = -0.2cm]above:$\vertex{1432}$}] at (1432) {};
\node [label = {[label distance = 0cm]right:$\vertex{2134}$}, red node] at (2134) {};
\node[label = {[label distance = -0.1cm]right:$\vertex{2143}$}, red node] at (2143) {};
\node[label = {[label distance = 0cm]below:$\vertex{2314}$}, red node] at (2314) {};
\node[label = {[label distance = 0cm]right:$\vertex{2341}$}] at (2341) {};
\node[label = {[label distance = 0cm]left:$\vertex{2413}$}] at (2413) {};
\node[label = {[label distance = -0.2cm]above:$\vertex{2431}$}] at (2431) {};
\node[label = {[label distance = -0.2cm]above:$\vertex{3124}$}] at (3124) {};
\node[label = {[label distance = 0cm]right:$\vertex{3142}$}] at (3142) {};
\node[label = {[label distance = -0.2cm]below:$\vertex{3214}$}, red node] at (3214) {};
\node [label = {[label distance = -0.1cm]left:$\vertex{3241}$}] at (3241) {};
\node[label = {[label distance = 0cm] left:$\vertex{3412}$}] at (3412) {};
\node[label = {[label distance = 0cm]left:$\vertex{3421}$}] at (3421) {};
\node[label = {[label distance = -0.2cm]above:$\vertex{4123}$}] at (4123) {};
\node [label = {[label distance = 0cm]above:$\vertex{4132}$}] at (4132) {};
\node[label = {[label distance = 0cm]below:$\vertex{4213}$}] at (4213) {};
\node[label = {[label distance = 0cm]above:$\vertex{4231}$}] at (4231) {};
\node[label = {[label distance = 0cm]left:$\vertex{4312}$}] at (4312) {};
\node [label = {[label distance = -0.2cm]below right:$\vertex{4321}$}] at (4321) {};

\end{tikzpicture}
\caption{A minimal-length path between $\vertex{1243}$ and $\vertex{3214}$ in $\Delta_{S_4}$}\label{fig:length}
\end{figure}

For every subset $\A$ of $W$ and every $u\in W$, there exists an element $y\in \A$ such that $d(u,y)=d(u,\A)$ since $\A$ is a finite set, but such an element $y$ may not be unique. However, the following proposition says that such an element $y$ is unique if $\A$ is a Coxeter matroid. 

\begin{proposition}[{\cite[Proposition~2.6]{LMP_retractions}}] \label{prop_mret_minimal_distance}
If $\A$ is a Coxeter matroid of a finite Coxeter group~$W$, then $q=\mret_{\A}(u)$ is the unique element satisfying $d(u,q)=d(u,\A)$. 
\end{proposition}

\begin{proof}
First we remark that if $v<^u w$, then $d(u,v)<d(u,w)$ for $u,v,w\in W$. Indeed, this is shown by the following observation: 
\[
\begin{split}
v<^u w 
\Longleftrightarrow \ &u^{-1}v< u^{-1}w\\
\Longrightarrow \ & d(e,u^{-1}v)< d(e,u^{-1}w)\\
\Longleftrightarrow \ & d(u,v)< d(u,w),
\end{split}
\]
where $e$ denotes the identity element of $W$ and the last equivalence follows from the invariance of the metric $d$ under the left multiplication of $W$. 

Since $\mret_{\A}(u)$ is the unique $\le^u$-minimal element in $\A$, it follows from the above observation that $d(u,\mret_{\A}(u))\leq d(u,w)$ for every $w\in\A$ and the equality holds only when $w=\mret_{\A}(u)$. Hence $\mret_{\A}(u)$ is the unique element in $W$ closest to $u$. This proves the proposition.
\end{proof}

There is an algorithm to find $\mret_\A(u)$ when $W$ is a Weyl group of classical Lie type. 
The Weyl group $W$ of classical Lie type is of the following form: 
\[
W=\begin{cases} {S}_n \qquad &\text{if $W$ is of type $A_{n-1}$},\\
(\Z/2\Z)^n\rtimes {S}_n \qquad &\text{if $W$ is of type $B_n$ or $C_n$},\\
(\Z/2\Z)^{n-1}\rtimes {S}_n\qquad &\text{if $W$ is of type $D_n$}.
\end{cases}
\]
We denote the set $\{\bar{1},\dots,\bar{n}\}$ by $[\overline{n}]$ and regard $\bar{\bar{i}}=i$. In each type, we will use one-line notation for $u\in W$, i.e.,
\[
u=u(1)u(2)\cdots u(n)
\]
where $u(i)\in [n]\cup[\bar{n}]$ and $u(1)u(2)\cdots u(n)$ is a permutation on $[n]$ if we forget the bars. There is no bar in type $A$ and the number of bars in $u(1),\dots,u(n)$ is even (possibly zero) in type $D$. In types $B$, $C$ and $D$, we have $u(\bar{i})=\overline{u(i)}$.

For $u\in W$, we define a linear order $\prec^u$ on the set $[n]\cup[\overline{n}]$ by
\begin{equation*}
u(1)\prec^u \dots \prec^u u(n)\prec^u u(\overline{n})\prec^u\dots \prec^u u(\overline{1}).
\end{equation*}
This induces a $u$-lexicographic order $\prec^u_{\mathrm{lex}}$ on the set of words of length $n$ in the alphabet $[n]\cup [\overline{n}]$. Then we obtain a linear order $\prec^u$ on $W$, where $v\prec^u w$ if and only if $v(1)\cdots v(n)\prec^u_{\mathrm{lex}} w(1)\cdots w(n)$. Note that $u$ is the minimal element of $W$ with respect to $\prec^u$.

\begin{definition} \label{defi:algebraic_retraction}
Let $W$ be a Weyl group of classical Lie type and $\A$ an \emph{arbitrary} subset of $W$. For each $u\in W$, we define $\aret_\A(u)$ as the $u$-minimal element of $\A$ with respect to the order~$\prec^u$. Then the map
\[
\aret_\A\colon W\to \A\, (\subset W)
\]
is a retraction of $W$ onto $\A$, which we call an \emph{algebraic retraction}.
\end{definition}

\begin{example}
\begin{enumerate}
\item We take a subset $\A=\{123, 132, 213, 312\}$ of $S_3$. For $u=231$, we have a linear order
$2\prec^u 3\prec^u 1$. 
Since
\[
213 \prec^u_{\mathrm{lex}}312\prec^u_{\mathrm{lex}} 123\prec^u_{\mathrm{lex}} 132,
\]
we have $\aret_\A(231)=213$.

\item We take a subset $\A=\{1\bar{4}23, 14\bar{3}\bar{2}, 2413, \bar{3}\bar{4}1\bar{2}\}$ of $(\Z/2\Z)^4 \rtimes {S}_4$. For $u=\bar{2}3\bar{1}4$, we have a linear order 
\[
\bar{2} \prec^u 3 \prec^u \bar{1} \prec^u 4 \prec^u \bar{4} \prec^u 1 \prec^u \bar{3} \prec^u 2.
\]
Since
\[
14\bar{3}\bar{2} \prec^u_{\mathrm{lex}} 1\bar{4}23\prec^u_{\mathrm{lex}} \bar{3}\bar{4}1\bar{2}\prec^u_{\mathrm{lex}} 2413,
\]
we have $\aret_\A(\bar{2}3\bar{1}4)=14\bar{3}\bar{2}$. 
\end{enumerate}
\end{example}

The following theorem together with Proposition~\ref{prop_mret_minimal_distance} shows that 
if $\A$ is a Coxeter matroid of a Weyl group $W$ of classical Lie type, then $\aret_\A(u)$ for $u\in W$ provides the point in $\A$ closest to $u$. 

\begin{theorem}[{\cite[Theorem~4.7]{LMP_retractions}}]\label{thm:aret-mret}
If $\A$ is a Coxeter matroid of a Weyl group $W$ of classical Lie type, then $\aret_\A=\mret_\A.$ 
\end{theorem}

Unless $\A$ is a Coxeter matroid, $\aret_\A(u)$ does not necessarily provide a point in $\A$ closest to $u$ as shown in the following example. 

\begin{example}\label{example_not_unique_closest}
Let $W={S}_4$ and $\A=\{1423,2134\}$. Then $\A$ is not a Coxeter matroid because the convex hull $\Delta_\A$ of $\A$,  the dotted red line in Figure~\ref{fig:aret_not_closest}, is not a $\Phi$-polytope since the dotted red line is not parallel to any edge of the $S_4$-permutohedron. On the other hand, one can check that $\A$ has the following property: for each $u\in {S}_4$, there is a unique element in $\A$ closest to $u$. However, if we take $u=1324$ for instance, then $d(u,2134)=2$ and $d(u,1423)=3$, but $\aret_\A(u)=1423$, see Figure~\ref{fig:aret_not_closest}. Therefore, $\aret_\A(1324)$ is not an element of $\A$ closest to $1324$. 
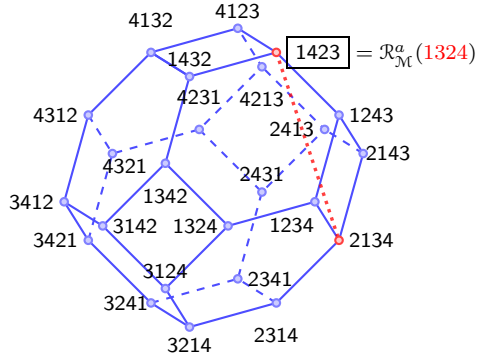
\begin{figure}[ht]
\begin{tikzpicture}[scale=5]
\tikzset{every node/.style={draw=blue!50,fill=blue!20, circle, thick, inner sep=1pt,font=\footnotesize}}
\tikzset{red node/.style = {fill=red!20!white, draw=red!75!white}}
\tikzset{red line/.style = {line width=0.3ex, red, nearly opaque}}

\coordinate (4231) at (1/3, 1/2, 1/6); 
\coordinate (2413) at (2/3, 1/2, 1/6); 
\coordinate (1243) at (5/6, 2/3, 1/2); 
\coordinate (2143) at (5/6, 1/2, 1/3); 
\coordinate (2134) at (5/6, 1/3, 1/2); 
\coordinate (1423) at (2/3, 5/6, 1/2); 
\coordinate (3142) at (1/3, 1/2, 5/6); 
\coordinate (1324) at (2/3, 1/2, 5/6); 
\coordinate (1234) at (5/6, 1/2, 2/3); 
\coordinate (1342) at (1/2, 2/3, 5/6); 
\coordinate (4123) at (1/2, 5/6, 1/3); 
\coordinate (4213) at (1/2, 2/3, 1/6); 
\coordinate (1432) at (1/2, 5/6, 2/3); 
\coordinate (4132) at (1/3, 5/6, 1/2); 
\coordinate (2314) at (2/3, 1/6, 1/2); 
\coordinate (3214) at (1/2, 1/6, 2/3); 
\coordinate (3124) at (1/2, 1/3, 5/6); 
\coordinate (3241) at (1/3, 1/6, 1/2); 
\coordinate (2341) at (1/2, 1/6, 1/3); 
\coordinate (2431) at (1/2, 1/3, 1/6);
\coordinate (3421) at (1/6, 1/3, 1/2); 
\coordinate (4321) at (1/6, 1/2, 1/3); 
\coordinate (3412) at (1/6, 1/2, 2/3); 
\coordinate (4312) at (1/6, 2/3, 1/2); 

\draw[thick, draw=blue!70] (1432)--(4132)--(4123)--(1423)--cycle;
\draw[thick, draw=blue!70] (4132)--(1432)--(1342)--(3142)--(3412)--(4312)--(4132);
\draw[dashed, thick, draw=blue!70] (4312)--(4321)--(4231)--(4213)--(4123);
\draw[dashed, thick, draw=blue!70] (3421)--(4321);
\draw[thick, draw=blue!70] (1342)--(1324)--(3124)--(3142);
\draw[dashed, thick, draw=blue!70] (4231)--(2431)--(2413)--(4213);
\draw[thick, draw=blue!70] (1423)--(1243)--(2143);
\draw[dashed, thick, draw=blue!70] (2143)--(2413);
\draw[thick, draw=blue!70] (1324)--(1234)--(1243);
\draw[thick, draw=blue!70] (1234)--(2134)--(2143);
\draw[thick, draw=blue!70] (2314)--(2134);
\draw[dashed, thick, draw=blue!70] (2314)--(2341)--(2431);
\draw[thick, draw=blue!70] (3412)--(3421)--(3241)--(3214)--(3124);
\draw[thick, draw=blue!70] (3214)--(2314);
\draw[dashed,thick, draw=blue!70] (3241)--(2341);

\draw[red line, dotted] (1423)--(2134);

\node [label = {[label distance = 0cm] below left:$\vertex{1234}$}] at (1234) {};
\node[label = {[label distance = 0cm]right:$\vertex{1243}$}] at (1243) {};
\node[label = {[label distance = 0cm]left:$\vertex{1324}$}] at (1324) {};
\node[label = {[label distance = 0cm]below:$\vertex{1342}$}] at (1342) {};
\node [label = {[label distance = 0cm]right:$\fbox{$\vertex{1423}$}=\aret_{\A}(\textcolor{red}{1324})$}, red node] at (1423) {};
\node[label = {[label distance = -0.2cm]above:$\vertex{1432}$}] at (1432) {};
\node [label = {[label distance = 0cm]right:$\vertex{2134}$},red node] at (2134) {};
\node[label = {[label distance = -0.1cm]right:$\vertex{2143}$}] at (2143) {};
\node[label = {[label distance = 0cm]below:$\vertex{2314}$}] at (2314) {};
\node[label = {[label distance = 0cm]right:$\vertex{2341}$}] at (2341) {};
\node[label = {[label distance = 0cm]left:$\vertex{2413}$}] at (2413) {};
\node[label = {[label distance = -0.2cm]above:$\vertex{2431}$}] at (2431) {};
\node[label = {[label distance = -0.2cm]above:$\vertex{3124}$}] at (3124) {};
\node[label = {[label distance = 0cm]right:$\vertex{3142}$}] at (3142) {};
\node[label = {[label distance = -0.2cm]below:$\vertex{3214}$}] at (3214) {};
\node [label = {[label distance = -0.1cm]left:$\vertex{3241}$}] at (3241) {};
\node[label = {[label distance = 0cm] left:$\vertex{3412}$}] at (3412) {};
\node[label = {[label distance = 0cm]left:$\vertex{3421}$}] at (3421) {};
\node[label = {[label distance = -0.2cm]above:$\vertex{4123}$}] at (4123) {};
\node [label = {[label distance = 0cm]above:$\vertex{4132}$}] at (4132) {};
\node[label = {[label distance = 0cm]below:$\vertex{4213}$}] at (4213) {};
\node[label = {[label distance = 0cm]above:$\vertex{4231}$}] at (4231) {};
\node[label = {[label distance = 0cm]left:$\vertex{4312}$}] at (4312) {};
\node [label = {[label distance = -0.2cm]below right:$\vertex{4321}$}] at (4321) {};
\end{tikzpicture}
\caption{$\aret_{\A}(1324)$ is not the element of $\A$ closest to $1324$.}
\label{fig:aret_not_closest}
\end{figure}
\end{example}

As mentioned in Proposition~\ref{prop_mret_minimal_distance}, the following is a necessary condition for $\A\subset W$ to be a Coxeter matroid.
\begin{quote}
$(\ast)$\quad For each $u\in W$, there is a \emph{unique} $q\in \A$ such that $d(u,q) = d(u,\A)$.
\end{quote}
However, it is not a sufficient condition as is shown in Example~\ref{example_not_unique_closest}. On the other hand, Proposition~\ref{prop_mret_minimal_distance} and Theorem~\ref{thm:aret-mret} show that if $\A$ is a Coxeter matroid of $W$, then the unique element $q$ in $(\ast)$ above must be given by $\aret_\A(u)$. We ask whether these two necessary conditions are sufficient:

\begin{Problem}[{\cite{LMP_retractions}}]
Let $W$ be a Weyl group of classical Lie type. Suppose that a subset $\A$ of $W$ satisfies the following two conditions:
\begin{enumerate}
\item for each $u\in W$, there is a unique $q\in \A$ such that $d(u,q) = d(u,\A)$, and
\item $q=\aret_\A(u)$.
\end{enumerate}
Then, is $\A$ a Coxeter matroid?
\end{Problem}

For $\A\subset S_n$, the answer of the above problem is yes when $|\A|=2$ (see \cite[Proposition 4.9]{LMP_retractions}) or $n\leq 4$ (by a computer check), but we do not know the answer for an arbitrary subset $\A$ of $S_n$ with $n\geq 5$.


\section{Generic torus orbit closures in the flag variety}
\label{sec_generic_orbit_in_flag}

In this section, we observe that generic torus orbit closures in the flag variety $\flag(n)$ are in fact permutohedral varieties and then discuss their topology.  Their Poincar\'e polynomials turn out to be Eulerian polynomials.  We also discuss Klyachko's theorem which describes the restriction image of the cohomology of $\flag(n)$ to that of the generic torus orbit closure in $\flag(n)$.  Finally we discuss a  generalization of Klyachko's theorem to Hessenberg varieties.    

The image of the moment map 
\[
\mu\colon \flag(n)=G/B\to \R^{n}
\]
defined in \eqref{eq_moment_map} is the permutohedron
\[
\Pi_n=\mathrm{Conv}\{(w(1),\dots,w(n))\mid w\in\Sn\}.
\]
Note that $\Pi_n$ is the $S_n$-permutohedron in Subsection~\ref{section:Coxeter matroids}.
The permutohedron $\Pi_n$ lies on the hyperplane
\begin{equation} \label{eq:hyperplane}
\{(x_1,\dots,x_n)\in \R^n\mid x_1+\cdots+x_n=n(n+1)/2\}
\end{equation}
and one can easily see that $\Pi_n$ is of dimension $n-1$.

\begin{definition}\label{def_generic_point}
We say that a $T$-orbit $O$ in $\flag(n)$ is \emph{generic} if $\mu(\overline{O})=\Pi_n$, in other words if $\overline{O}^T=\Sn$. 
\end{definition}


\subsection{Faces of the permutohedron \texorpdfstring{$\Pi_n$}{Pin}}

We recall some facts on the faces of the permutohedron~$\Pi_n$. We refer the reader to~\cite[Theorem~6.1]{Zelevinsky06}, \cite{PostnikovReinerWilliams08}, \cite{Postnikov09}, or \cite[Section~5.A]{KLSS} for more detail.

The permutohedron $\Pi_n$ is contained in the half space of $\R^n$ defined by 
\[
\left\{(x_1,\dots,x_n)\in\R^n \,\,\middle|\,\, \sum_{i=1}^k x_i\ge \sum_{i=1}^ki \right\}\quad (k=1,\dots,n-1)
\]
and the intersection of its boundary with $\Pi_n$, that is, 
\begin{equation} \label{eq:facet_e}
\left\{(x_1,\dots,x_n)\in \Pi_n \,\,\middle|\,\, \sum_{i=1}^kx_i=\sum_{i=1}^ki \right\}\quad (k=1,\dots,n-1)
\end{equation}
is a facet of $\Pi_n$ that contains the vertex $(1,\dots,n)$. This provides all the facets of $\Pi_n$ which meet at the vertex $(1,\dots,n)$. Since $\Pi_n$ is invariant under permutations of the coordinates of $\R^n$, it follows that $\Pi_n$ is a simple polytope of dimension $n-1$ and the facets of $\Pi_n$ meeting at the vertex $(w(1),\dots,w(n))$ of $\Pi_n$ $(w\in \Sn)$ are written as 
\begin{equation} \label{eq:facet_w}
\left\{(x_1,\dots,x_n)\in \Pi_n \,\,\middle|\,\, \sum_{i=1}^kx_{w^{-1}(i)}=\sum_{i=1}^ki \right\}\quad (k=1,\dots,n-1).
\end{equation} 

The facet in \eqref{eq:facet_w} is determined by the subset $\{w^{-1}(1),\dots,w^{-1}(k)\}$ of $[n]$ and the observation above shows that there is a bijective correspondence between non-empty proper subsets of $[n]$ and facets of $\Pi_n$. Indeed, the facet of $\Pi_n$ associated to a non-empty proper subset $A$ of $[n]$ is defined by 
\begin{equation} \label{eq:normal_A}
F(A) \colonequals \left\{(x_1,\dots,x_n)\in \Pi_n \,\,\middle|\,\, \sum_{a\in A}x_a= \frac{|A|(|A|+1)}{2}\right\}.
\end{equation}
Therefore, there are $2^n-2$ facets in $\Pi_n$ and \eqref{eq:facet_w} shows that the $n-1$  facets of $\Pi_n$ meeting at the vertex $(w(1),\dots,w(n))$ are $F(A)$'s with $w(A)=\{1,\dots,|A|\}$. See Figure~\ref{fig:permuto-facet}. 

\begin{figure}[H]
\centering
\begin{tikzpicture}[x=1cm, y=1cm, z=-0.6cm,scale=.8]
\filldraw[fill=yellow] (1,2,3)--(2,1,3)--(3,1,2)--(3,2,1)--(2,3,1)--(1,3,2)--cycle;
\draw [->] (0,0,0) -- (4,0,0) node [right] {$y$};
\draw [->] (0,0,0) -- (0,4,0) node [left] {$z$};
\draw [->] (0,0,0) -- (0,0,4) node [left] {$x$};
\draw[thick] (1,2,3)--(2,1,3)--(3,1,2)--(3,2,1)--(2,3,1)--(1,3,2)--cycle;
\draw [dotted]
(0,3,0) -- (3,3,0) -- (3,0,0)
(3,0,0) -- (3,0,3) -- (0,0,3)
(0,3,0) -- (0,3,3) -- (0,0,3)
(3,3,0) -- (3,0,3)
(3,3,0) -- (0,3,3)
(0,3,3) -- (3,0,3)
(1,3,2) -- (1,2,3)
(2,1,3) -- (3,1,2)
;
\draw [ thick, red]
(2.6,0.4,3)--(0.4,2.6,3)
(2.6,3,0.4)--(0.4,3,2.6)
(3,2.7,0.3)--(3,0.3,2.7)
;
\draw [ thick,blue]
(2.6,0.4,-.65)--(0.4,2.6,-.65)
(2.7,.25,0.3)--(0.4,.25,2.6)
(.25,2.6,0.4)--(.25,0.4,2.6)
;

\draw (1.8,1.5,3.3) node[above]{\scriptsize$F(\{2,3\})$};
\draw (1.6,2.8,1.4) node[above]{\scriptsize$F(\{1,2\})$};
\draw (2.9,1.5,1.5) node[right]{\scriptsize$F(\{1,3\})$};
\draw (1.2,1.5,-0.9) node[right]{\scriptsize$F(\{1\})$};
\draw (1.5,.25,1.7) node[right]{\scriptsize$F(\{3\})$};
\draw (0,1.6,1.6) node[above]{\scriptsize$F(\{2\})$};

\draw (3,0,0) node[below]{$3$};
\draw (0,3,0) node[left]{$3$};
\draw (0,0,3) node[right]{$3$};
\draw (1,2,3) node[left]{\tiny{$(3,1,2)$}};
\draw (2,1,3) node[below]{\tiny{$(3,2,1)$}};
\draw (2.2,3,1) node[above]{\tiny{$(1,2,3)$}};
\draw (1,3,2) node[left]{\tiny{$(2,1,3)$}};
\draw (3,2,1) node[right]{\tiny{$(1,3,2)$}};
\draw (3,1,2) node[right]{\tiny{$(2,3,1)$}};
\end{tikzpicture}
\caption{Facets of the permutohedron $\Pi_3$}\label{fig:permuto-facet}
\end{figure}

The following lemma can easily be proved. 

\begin{lemma} \label{lemm:tauAtauB}
Let $A$ and $B$ be non-empty proper subsets of $[n]$. Then $F(A)\cap F(B)\not=\emptyset$ if and only if $A\subset B$ or $B\subset A$. Therefore, a codimension $k$ face of $\Pi_n$ is $\bigcap_{i=1}^d F(A_i)$ with $A_1\subsetneq\cdots \subsetneq A_k$. 
\end{lemma}

Since $\Pi_n$ lies on the hyperplane in \eqref{eq:hyperplane} which is perpendicular to the vector $(1,\dots,1)$  in $\R^n$, the normal fan $\Sigma(\Pi_n)$ of $\Pi_n$ is defined in the quotient space $N_\R\colonequals\R^n/\R(1,\dots,1)$, where the lattice $N$ of $N_\R$ is the one induced from the standard lattice $\Z^n$ of $\R^n$. By \eqref{eq:normal_A}, the primitive (inward) normal vector $\delta_A$ to the hyperplane in \eqref{eq:normal_A} defining the facet $F(A)$ is the $(0,1)$-vector with $1$ in the $a$th coordinate for each $a\in A$ and $0$ otherwise. Therefore we have the following. 

\begin{lemma} \label{lemm:linear_relation}
The quotient image of the vector $\delta_A$ on $N_\R=\R^n/\R(1,\dots,1)$ is the primitive ray vector in the normal fan $\Sigma(\Pi_n)$ corresponding to the facet $F(A)$. 
\end{lemma} 
 
The maximal cone in the fan $\Sigma(\Pi_n)$ corresponding to the vertex $(1,\dots,n)$ is spanned by the ray vectors corresponding to the $n-1$ facets in \eqref{eq:facet_e} and those ray vectors are the projection image of the following $n-1$ vectors on $N_\R$: 
\[
(1,0,\dots,0),\quad (1,1,0,\dots,0),\quad \dots,\quad (1,\dots,1,0). 
\]
The other maximal cones in $\Sigma(\Pi_n)$ are obtained by permuting the above cone. This shows that the collection of maximal cones coincides with the (closures of) Weyl chambers in type $A$. 
The maximal cones determine the fan $\Sigma(\Pi_n)$ and the permutohedral variety ${\Perm}_n$ is the toric variety whose fan has the (closure of) Weyl chambers as maximal cones. Therefore, we obtain the following. 

\begin{proposition}
The maximal cones in the fan of a generic torus orbit closure in the flag variety $\flag(n)$ are the \textup{(}closures of\textup{)} Weyl chambers. Therefore, the closure of a generic torus orbit in the flag variety $\flag(n)$ is isomorphic to the permutohedral variety ${\Perm}_n$. 
\end{proposition}

The cohomology ring of a compact smooth toric variety is explicitly described in terms of the associated fan. Applying the general result to our setting together with Lemmas~\ref{lemm:tauAtauB} and~\ref{lemm:linear_relation}, we obtain the following theorem. 

\begin{theorem} \label{theo:perm_cohomology}
The cohomology ring of the permutohedral variety ${\Perm}_n$ has the following presentation:
\[
H^*(\Perm_n;\Z)=\Z[\tau_A\mid \emptyset \subsetneq A\subsetneq [n]]/\mathcal{I}
\]
where $\deg\tau_A=2$ and $\mathcal{I}$ is the ideal generated by the following two types of elements:
\begin{enumerate}
\item $\tau_A\tau_B$ for $A\not\subset B$ or $B\not\subset A$,
\item $\sum_{p\in A}\tau_A-\sum_{q\in B}\tau_B$ for $p,q\in [n]$.
\end{enumerate}
\end{theorem}

\begin{example}
Take $n=3$. Then $H^*({\Perm}_3;\Z)$ is generated by six elements 
\[
\tau_{\{1\}},\ \tau_{\{2\}},\ \tau_{\{3\}},\ \tau_{\{1,2\}},\ \tau_{\{1,3\}},\ \tau_{\{2,3\}}
\]
with relations
\begin{enumerate}
\item $\tau_{\{1\}}\tau_{\{2\}}=\tau_{\{1\}}\tau_{\{3\}}=\tau_{\{2\}}\tau_{\{3\}}=\tau_{\{1\}}\tau_{\{2,3\}}=\tau_{\{2\}}\tau_{\{1,3\}}=\tau_{\{3\}}\tau_{\{1,2\}}$\\
$\quad =\tau_{\{1,2\}} \tau_{\{1,3\}} =\tau_{\{1,3\}}\tau_{\{2,3\}} = \tau_{\{1,2\}} \tau_{\{2,3\}}
=0$,
\item $\tau_{\{1\}}+ \tau_{\{1,2\}}+\tau_{\{1,3\}}=\tau_{\{2\}}+ \tau_{\{1,2\}}+\tau_{\{2,3\}}=\tau_{\{3\}}+ \tau_{\{1,3\}}+\tau_{\{2,3\}}$.
\end{enumerate}
It follows that $H^2({\Perm}_3)$ is freely generated by four classes, say
$\tau_{\{1\}}, \ \tau_{\{2\}},\ \tau_{\{1,2\}},\ \tau_{\{2,3\}}$, 
and $H^4({\Perm_3})$ is freely generated by one class, say $\tau_{\{1\}}\tau_{\{1,2\}}$. Therefore, the Poincar\'e polynomial of ${\Perm}_3$ is given by
\begin{equation} \label{eq:Poin_Perm_3}
{\Poin}({\Perm}_3,t)=1+4t^2+t^4.
\end{equation}
In fact, since the fan of ${\Perm}_3$ is isomorphic to the fan of the toric variety obtained by blowing up $\C P^2$ with the standard $T^2$-action at the three fixed points, ${\Perm}_3$ is isomorphic to 
\park{$\C P^2\# 3\overline{\C P^2}$,  the equivariant connected sum of $\C P^2$ with three copies of $\overline{\C P^2}$, where $\overline{\C P^2}$ stands for $\C P^2$ with the opposite orientation.}
\end{example}


\subsection{Eulerian polynomial}

The \lee{set of} \emph{ascents} ${\rm asc}(w)$ of $w\in \Sn$ is defined by 
\begin{equation} \label{eq:Eulerian_polynomial}
{\rm asc}(w)\colonequals\#\{ i\in [n-1]\mid w(i)<w(i+1)\}
\end{equation}
and the Eulerian polynomial $A_n(t)$ is defined by 
\[
A_n(t)\colonequals\sum_{w\in \Sn}t^{{\rm asc}(w)}. 
\]
Here, for a set $A$, we denote by $\# A$ the cardinality  of $A$.
Note that $0\le {\asc}(w)\le n-1$ for any $w\in \Sn$, where the former equality is attained only when $w=[n,n-1,\dots,1]$ and the latter equality is attained only when $w$ is the identity. Therefore, the constant term of $A_n(t)$ is $1$ and the highest degree term is $t^{n-1}$. Eulerian polynomials $A_n(t)$ for $n\le 5$ are given as follows: 
\begin{equation} \label{eq:Eulerian_polynomial_5}
\begin{split} 
&A_1(t)=1,\quad A_2(t)=1+t,\quad A_3(t)=1+4t+t^2,\\
&A_4(t)=1+11t+11t^2+t^3,\quad A_5(t)=1+26t+66t^2+26t^3+t^4. 
\end{split}
\end{equation}

\begin{remark}
\begin{enumerate}
\item The \emph{descent} ${\rm des}(w)$ of $w\in \Sn$ is defined by
\[
{\rm des}(w)\colonequals\#\{ i\in[n-1]\mid w(i)>w(i+1)\}
\]
and it holds that 
\[
\sum_{w\in \Sn}t^{{\rm asc}(w)}=\sum_{w\in \Sn}t^{{\rm des}(w)}.
\]
Therefore, we may define the Eulerian polynomial $A_n(t)$ using the descents. 

\item The Eulerian polynomial $A_n(t)$ was originally defined by Euler as the polynomial which appears in the numerator of the following formula:
\[
\sum_{k=1}^\infty k^nt^k=\frac{tA_n(t)}{(1-t)^{n+1}}.
\]
As easily checked, the Eulerian polynomials defined this way have the following recurrence relation 
\[
A_{n+1}(t)=(nt+1)A_n(t)-t(t-1)\frac{dA_n(t)}{dt}
\]
while one can see that the Eulerian polynomials defined in \eqref{eq:Eulerian_polynomial} satisfy the same recurrence relation as above. Since $A_1(t)=1$ in both definitions, they define the same family of polynomials. 
\end{enumerate}
\end{remark}

It follows from \eqref{eq:Poin_Perm_3} and \eqref{eq:Eulerian_polynomial_5} that ${\Poin}(\Perm_3,t)=A_3(t^2)$. This is not a coincidence.

\begin{theorem}[{\cite[Theorem~2]{Klyachko85Orbits}}] \label{theo:perm&eulerian}
${\Poin}(\Perm_n,t)=A_n(t^2)$ for any $n\ge 1$. 
\end{theorem}

\begin{remark}
The polynomials $A_n(t)$ for $n\le 5$ in \eqref{eq:Eulerian_polynomial_5} are palindromic and unimodal. Indeed, this holds for any $A_n(t)$ 
since the Poincar\'e duality implies the palindromicity and the hard Lefschetz theorem implies the unimodality because $\Perm_n$ is a smooth projective toric variety.
\end{remark}

In order to give a proof of Theorem~\ref{theo:perm&eulerian}, we recall a fact on edges of the permutohedron $\Pi_n$.

\begin{lemma} \label{lemm:edge_Pi_n}
Two vertices $(u(1),\dots,u(n))$ and $(v(1),\dots,v(n))$ of $\Pi_n$ $(u,v\in \Sn)$ are joined by an edge of $\Pi_n$ if and only if there is an adjacent transposition $s_i=(i,i+1)$ for $i\in [n-1]$ such that $v=s_iu$. 
\end{lemma}

\begin{proof}
We note that 
\begin{equation} \label{eq:v_minus_id}
(v(1),\dots,v(n))-(1,\dots,n)=\sum_{i=1}^{n-1}\left(\sum_{k=1}^iv(k)-\sum_{k=1}^ik\right)(\mathbf{e}_i-\mathbf{e}_{i+1})
\end{equation}
where $\mathbf{e}_1,\dots,\mathbf{e}_n$ denote the standard basis vectors of $\R^n$. The coefficient of $\mathbf{e}_i-\mathbf{e}_{i+1}$ in \eqref{eq:v_minus_id} is non-negative for any $i$ and $v\in \Sn$. Moreover, when $v=s_i$, the right hand side at \eqref{eq:v_minus_id} is $\mathbf{e}_i-\mathbf{e}_{i+1}$. This implies the lemma when $u$ is the identity. The general case follows from this special case and the invariance of $\Pi_n$ under permutations of the coordinates of $\R^n$. 
\end{proof}

\begin{proof}[Proof of Theorem~\ref{theo:perm&eulerian}]
Equation~\eqref{eq:v_minus_id} is generalized to 
\begin{equation} \label{eq:v_minus_u}
(v(1),\dots,v(n))-(u(1),\dots,u(n))=\sum_{i=1}^{n-1}\left(\sum_{k=1}^iv(k)-\sum_{k=1}^i u(k)\right)(\mathbf{e}_i-\mathbf{e}_{i+1}).
\end{equation}
By Lemma~\ref{lemm:edge_Pi_n}, the endpoints of the edges emanating from the vertex $(u(1),\dots,u(n))$ are $(s_iu(1),\dots,s_iu(n))$ for $i\in [n-1]$. Set $p=u^{-1}(i)$ and $q=u^{-1}(i+1)$. Then, since $u(p)=i$ and $u(q)=i+1$ and the action of $s_i$ from the left interchanges $i$ and $i+1$, \eqref{eq:v_minus_u} reduces to 
\begin{equation} \label{eq:sku_minus_u}
(s_iu(1),\dots,s_iu(n))-(u(1),\dots,u(n))=\mathbf{e}_p-\mathbf{e}_{q}=\mathbf{e}_{u^{-1}(i)}-\mathbf{e}_{u^{-1}(i+1)}.
\end{equation}

Take real numbers $a_1>\cdots >a_n$ and consider a linear function $f\colon \R^n\to \R$ defined by 
\[
f(x_1,\dots,x_n)=a_1x_1+\cdots+a_nx_n.
\]
It follows from \eqref{eq:sku_minus_u} that 
\[
f(s_iu(1),\dots,s_iu(n))-f(u(1),\dots,u(n))=a_{u^{-1}(i)}-a_{u^{-1}(i+1)}\begin{cases} >0\ \ &\text{ if }u^{-1}(i)<u^{-1}(i+1),\\
<0 \ \ &\text{ if }u^{-1}(i)>u^{-1}(i+1).\end{cases}
\]

Recall that $\mu(wB)=(w^{-1}(1),\dots,w^{-1}(n))$ for the moment map $\mu\colon \flag(n)\to\R^n$ and we labeled the vertex $(w^{-1}(1),\dots,w^{-1}(n))$ as $\vertex{w}$. 
Then, the observation above shows that ${\rm asc}(w)$ is
the number of edges emanating from the vertex $\vertex{w}$ and increasing with respect to the function $f$. 
Therefore, discrete Morse theory applied to the function $f$ on the permutohedron $\Pi_n$ produces a decomposition 
\[
\Pi_n=\bigsqcup_{w\in\Sn}P_w, \qquad P_w\approx (\R_{\ge 0})^{{\rm asc}(w)}
\]
where $\approx$ means homeomorphic as a manifold with corners. Moreover, one sees that $\mu^{-1}(P_w)\approx \C^{{\rm asc}(w)}$ for $\mu$ restricted to $\Perm_n$ and the decomposition 
\[
\Perm_n=\mu^{-1}(\Pi_n)=\bigsqcup_{w\in \Sn}\mu^{-1}(P_w)
\]
gives a cell decomposition of $\Perm_n$. Since the dimension of the cell $\mu^{-1}(P_w)$ is $2{\rm asc}(w)$, we obtain
\[
{\Poin}(\Perm_n,t)=\sum_{w\in \Sn}t^{2{\rm asc}(w)}=A_n(t^2),
\]
proving the theorem. 
\end{proof}

\begin{example}
Take $n=3$. Consider a linear function $f(x_1,x_2,x_3) = 5x_1 + 4x_2 + x_3$. The values~$f(\mu(wB))$ and $\asc(w)$ are given as follows. Note that the vertex $\mu(wB)$ of $\Pi_3$ is labeled by~$\vertex{w}$ as before. 

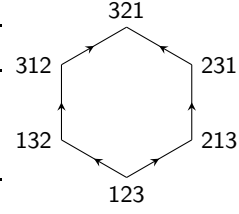
\begin{tabular}{ccccccc}
\toprule 
$w$ & $123$ & $213$ & $132$ & $231$ & $312$ & $321$\\
\midrule
$\mu(wB)$ & $(1,2,3)$ & $(2,1,3)$ & $(1,3,2)$ & $(3,1,2)$ & $(2,3,1)$ & $(3,2,1)$\\
$f(\mu(wB))$ & $16$ & $17$ & $19$ & $21$ & $23$ & $24$\\
$\asc(w)$ & $2$ & $1$ & $1$ & $1$ & $1$ & $0$ \\
\bottomrule
\end{tabular}
\begin{minipage}{0.3\linewidth}   
\begin{tikzpicture}
[decoration={markings, 
    mark= at position 0.5 with {\arrow{stealth}}}
] 
\foreach \x in {1,...,6}{
	\coordinate (\x) at (\x*60-150:1cm) ;
}

\node[below] at (1) {\small $\vertex{123}$};
\node[right] at (2) {\small $\vertex{213}$};
\node[right] at (3) {\small $\vertex{231}$};
\node[above] at (4) {\small $\vertex{321}$};
\node[left] at (5) {\small $\vertex{312}$};
\node[left] at (6) {\small $\vertex{132}$};

\draw[postaction={decorate}] (1) -- (2) ;
\draw[postaction={decorate}] (2) -- (3) ;
\draw[postaction={decorate}] (3) -- (4) ;
\draw[postaction={decorate}] (1) -- (6) ;
\draw[postaction={decorate}] (6) -- (5) ;
\draw[postaction={decorate}] (5) -- (4) ;
\end{tikzpicture}
\end{minipage}
Accordingly, we have
\[
\Poin(\Perm_3, t) = A_3(t^2) = 1 + 4t^2 + t^4.
\]
\end{example}


\subsection{Klyachko's result}

Since 
\[
\flag(n) = \{ V_\bullet=(\{0\}\subset V_1 \subset V_2 \subset \cdots \subset V_n = \C^n) \mid \dim_{\C} V_i = i \: \text{ for all }i = 1,\dots,n\}, 
\]
there is a sequence of tautological complex vector bundles on $\flag(n)$: 
\[
E_1\subset E_2\subset \cdots \subset E_n=\flag(n)\times \C^n
\]
where 
\[
E_i\colonequals\{(V_\bullet,v)\in \flag(n)\times\C^n\mid v\in V_i\}\qquad (i=1,\dots,n).
\]
We define 
\[
x_i\colonequals c_1(E_i/E_{i-1})\qquad \text{for $i=1,\dots,n$,}
\]
where $E_0=\flag(n)$, i.e.,  the $0$-dimensional vector bundle over $\flag(n)$ and $c_1(E)$ denotes the first Chern class of a vector bundle $E$. 
Since 
\[
E_n\cong \bigoplus_{i=1}^n E_i/E_{i-1}
\]
and $E_n$ is the trivial vector bundle $\flag(n)\times \C^n$, we have 
\[
1=c(E_n)=\prod_{i=1}^nc(E_i/E_{i-1})=\prod_{i=1}^n(1+x_i).
\]
Here, $c(E)$ denotes the total Chern class of a vector bundle $E$.
Therefore, the $i$th elementary symmetric polynomial $e_i(x)$ in $x_1,\dots,x_n$ vanishes in $H^*(\flag(n);\Z)$ for $i=1,\dots,n$. The following theorem shows that $H^*(\flag(n);\Z)$ is generated by $x_1,\dots,x_n$ as a ring and any polynomial in $x_1,\dots,x_n$ which vanishes in $H^*(\flag(n);\Z)$ is generated by $e_i(x)$'s. 

\begin{theorem}[Borel~\cite{Borel53}] \label{theo:flag_cohomology}
Let $x_1,\dots,x_n$ be as above. Then 
\begin{equation*} \label{eq:cohomology_flag}
H^*(\flag(n);\Z)=\Z[x_1,\dots,x_n]/(e_1(x), \dots, e_n(x)) \quad\text{as a ring}. 
\end{equation*}
\end{theorem}

We denote by $\Perm_n$ a generic torus orbit closure in the flag variety $\flag(n)$ because it is isomorphic to the permutohedral variety $\Perm_n$. The natural action of $\Sn$ on the fan of $\Perm_n$ induces an action of $\Sn$ on $\Perm_n$ and hence on its cohomology ring.

\begin{theorem}[Klyachko \cite{Klyachko85Orbits}] \label{theo:Klyachko}
Let $\iota\colon \Perm_n\to \flag(n)$ be the inclusion map. Then the image of the restriction map 
\[
\iota^*\colon H^*(\flag(n);\mathbb{Q})\to H^*(\Perm_n;\mathbb{Q}) 
\]
agrees with the ring of $\Sn$-invariants $H^*(\Perm_n;\mathbb{Q})^{\Sn}$ and 
\begin{equation} \label{eq:quadratic_relation}
H^*(\Perm_n;\mathbb{Q})^{\Sn}=\mathbb{Q}[\omega_1,\dots,\omega_{n-1}]/(\omega_i(\omega_{i-1}-2\omega_i+\omega_{i+1})\mid i=1,\dots,n-1)
\end{equation}
where $\omega_i=x_1+\cdots+x_i$ and $\omega_0=\omega_n=0$. 
\end{theorem}

\begin{remark}
\begin{enumerate}
\item The cohomology ring of $\Perm_n$ is explicitly described in Theorem~\ref{theo:perm_cohomology}. The $\Sn$-orbit of $\tau_A$ with $|A|=i$ is $\sum_{|B|=i}\tau_B$ which is equal to $x_i-x_{i+1}$. 

\item If we set $\alpha_i\colonequals x_i-x_{i+1}$, then $\omega_{i-1}-2\omega_i+\omega_{i+1}=-\alpha_i$. Therefore, we may express the ideal in \eqref{eq:quadratic_relation} as $(\omega_i\alpha_i\mid i=1,\dots,n-1)$. Note that the $\omega_i$'s correspond to the fundamental weights while $\alpha_i$'s correspond to the roots in type $A$. 
\end{enumerate}
\end{remark}


\subsection{Relation to Hessenberg varieties}
\label{sec_Hess}
Theorem~\ref{theo:Klyachko} is generalized to the setting of Hessenberg varieties. Given a square matrix $A$ of size $n$ with complex entries and a function $h\colon [n]\to [n]$ (called a \emph{Hessenberg function}) satisfying 
\[
h(1)\le h(2)\le \cdots\le h(n)\quad \text{and} \quad h(j)\ge j\quad \text{for all } j\in [n], 
\]
the Hessenberg variety ${\rm Hess}(A,h)$ is defined as 
\[
{\rm Hess}(A,h)\colonequals\{V_\bullet\in \flag(n)\mid A(V_j)\subset V_{h(j)}\: \text{ for all } j\in [n]\},
\]
where the matrix $A$ is regarded as a linear transformation on $\C^n$. 
We notice that ${\rm Hess}(A,h)$ is isomorphic to ${\rm Hess}(gAg^{-1},h)$ via the action of $G$ on $\flag(n)$. Accordingly, for a fixed Hessenberg function $h$, the geometry of Hessenberg variety ${\rm Hess}(A,h)$ depends only on the conjugacy class of $A$.  
We often express the Hessenberg function $h$ as a vector $(h(1),\dots,h(n))$ by listing the values of $h$. When $h=(n,\dots,n)$, 
it is obvious from the definition that ${\rm Hess}(A,h)$ is the flag variety $\flag(n)$ regardless of $A$. 

The Hessenberg variety ${\rm Hess}(S,h)$ for a square matrix $S$ of size $n$ with distinct eigenvalues is called \emph{regular semisimple}. 
It is known that 
\begin{enumerate}
\item ${\rm Hess}(S,h)$ is smooth, \label{Hess_property_1}
\item $\dim_\C{\rm Hess}(S,h)=\sum_{j=1}^n(h(j)-j)$, \label{Hess_property_2}
\item ${\rm Hess}(S,h)$ is connected if and only if $h(j)\ge j+1$ for all $j\in [n-1]$. \label{Hess_property_3}
\end{enumerate}
The topology of ${\rm Hess}(S,h)$ is independent of the choice of $S$ and we take $S$ to be a diagonal matrix with distinct diagonal entries in the following. Indeed, for regular semisimple matrics $S$ and $S'$, the corresponding Hessenberg varieties ${\rm Hess}(S,h)$ and ${\rm Hess}(S',h)$ are diffeomorphic.\footnote{For regular semisimple matrics $S$ and $S'$, there is a path connecting them in the space of regular semsimple matrices and the family of regular semisimple Hessenberg varieties are fiber bundle over the path where the Hessenberg function $h$ is fixed.  Since a path is contractible, the fiber bundle is trivial, which implies the desired result.} 
Since $S$ commutes with the diagonal torus $T$ of~${\rm GL}_n(\C)$, the restricted action of $T$ on $\flag(n)$ leaves ${\rm Hess}(S,h)$ invariant. One sees that 
\begin{equation} \label{eq:Hess(S,h)_fixed}
{\rm Hess}(S,h)^T=\flag(n)^T=\Sn.
\end{equation}
When $h=(2,3,\dots,n,n)$, it follows from~\eqref{Hess_property_2} and \eqref{Hess_property_3} above that $\dim_\C{\rm Hess}(S,h)=n-1$ and ${\rm Hess}(S,h)$ is connected. 
Although $\dim_\C T=n$, the subgroup $D$ of $T$ consisting of scalar matrices acts on $\flag(n)$ trivially and the induced action of $T/D$ on ${\rm Hess}(S,h)$ is effective. Therefore, ${\rm Hess}(S,h)$ for $h=(2,3,\dots,n,n)$ is a smooth toric variety and hence it is a torus orbit closure in $\flag(n)$. Moreover, the orbit is generic by \eqref{eq:Hess(S,h)_fixed}. Therefore, ${\rm Hess}(S,h)$ for $h=(2,3,\dots,n,n)$ is isomorphic to the permutohedral variety $\Perm_n$. 

Using the $T$-action on ${\rm Hess}(S,h)$, Tymoczko \cite{Tymoczko08permutation} constructed an action of $\Sn$ (called \emph{dot action}) on $H^*({\rm Hess}(S,h);\mathbb{Z})$ and when $h=(2,3,\dots,n,n)$, this action agrees with that on $H^*(\Perm_n;\mathbb{Z})$ induced from the action of $\Sn$ on $\Perm_n$. Therefore, the following theorem is a generalization of the former part of Theorem~\ref{theo:Klyachko}. 

\begin{theorem}[{\cite[Theorem~B]{AHHM19}}]\label{thm_AHHM}
Let $\iota\colon {\rm Hess}(S,h)\to \flag(n)$ be the inclusion map. Then the image of the restriction map 
\[
\iota^*\colon H^*(\flag(n);\mathbb{Q})\to H^*({\rm Hess}(S,h);\mathbb{Q}) 
\]
agrees with the ring of $\Sn$-invariants $H^*({\rm Hess}(S,h);\mathbb{Q})^{\Sn}$ for any Hessenberg function $h$. 
\end{theorem}

In fact, it is known from~\cite{BrosnanChow} that 
\begin{equation}\label{eq_Sn_invariant_Hess_Sh}
H^*({\rm Hess}(S,h);\mathbb{Q})^{\Sn}\cong H^*({\rm Hess}(N,h);\mathbb{Q}),
\end{equation}
where $N$ is conjugate to a nilpotent matrix with one Jordan block (${\rm Hess}(N,h)$ is called \emph{regular nilpotent}) and an explicit ring presentation of $H^*({\rm Hess}(N,h);\mathbb{Q})$, which reduces to \eqref{eq:quadratic_relation} when $h=(2,3,\dots,n,n)$, is known. 
Therefore, Theorem~\ref{theo:Klyachko} is completely generalized to the setting of Hessenberg varieties, see \cite{AHHM19, AbeHoriguchi20,AHMMS20, HoriguchiHarada} for more details.

\begin{remark}
It is shown in~\cite{AbeZeng} that when $h=(2,3,\dots,n)$ (so that ${\rm Hess}(S,h)=\Perm_n$), Theorem~\ref{thm_AHHM} and~\eqref{eq_Sn_invariant_Hess_Sh} hold with $\Z$-coefficient but~\eqref{eq:quadratic_relation} does not. 
\end{remark}

\section{Generic torus orbit closures in Schubert varieties}\label{sec:Schubert}
In this section, we consider geometric and topological properties of generic torus orbit closures in Schubert varieties with respect to the $T$-action from the introduction. A generic torus orbit closure $Y_w$ in a Schubert variety $X_w$ is not smooth in general and we discuss how to determine the smoothness of $Y_w$ at a fixed point by considering a certain graph. We study the Poincar\'e polynomial of $Y_w$ which is a generalization of the Eulerian polynomial. Moreover, we study the fan of a toric Schubert variety and discuss the classification of toric Schubert varieties. Finally, we summarize the properties of Schubert varieties of complexity one. 


\subsection{Generic torus orbit closures in Schubert varieties}

\begin{definition}
For $w \in S_n$, the \emph{Schubert variety} $X_w$ is a subvariety of $G/B$ defined by $X_w \colonequals \overline{BwB/B}$. 
\end{definition}
The action of $T$ on $G/B$ leaves $X_w$ invariant. 
\begin{definition}
We say that a $T$-orbit $O$ in~$X_w$ is \emph{generic} if $\mu(\overline{O}) = \mu(X_w)$, in other words, if $\overline{O}^T = X_w^T$. 
We denote by $Y_w$ the closure of a generic torus orbit in the Schubert variety $X_w$.
\end{definition}
Since $X_{w_0} = G/B$, the generic torus orbit closure $Y_{w_0}$ in $X_{w_0}$ is the permutohedral variety~$\Perm_n$ considered in Section~\ref{sec_generic_orbit_in_flag} (cf. Definition~\ref{def_generic_point}), in other words, $Y_{w_0} = \Perm_n$.

We notice that a fixed point~$uB$ is contained in the Schubert variety $X_w$ if and only if $u \leq w$ in the Bruhat order; $uB$ is contained in the opposite Schubert variety $X^w$ if and only if $u \geq w$ in the Bruhat order. 
By Lemma~\ref{lem:moment_map}, we have
\[
\mu(X_w) = \text{Conv}\{ (u^{-1}(1),\dots,u^{-1}(n)) \mid u \leq w\} \subset \R^n.
\] 
Motivated by this fact, we provide the following definition. 
\begin{definition}
For $w \in S_n$, we define a polytope $\Q_w$ by
\[
\Q_w \colonequals \text{Conv} \{ (u^{-1}(1),\dots,u^{-1}(n)) \mid u \leq w\}
\]
and label the vertex $(u^{-1}(1),\dots,u^{-1}(n))$ as $\vertex{u}$ as before.  
\end{definition}
Therefore we have
\[
\mu(X_w) = \mu(Y_w) = \Q_{w}.
\]
We depict the Bruhat interval $[e, 3412] = \{ u \in S_4 \mid u \leq 3412 \}$ and the polytope $\Q_{3412}$ in Figure~\ref{fig:BIP-3412}.
\begin{figure}[hbt!]
\begin{subfigure}[c]{0.51\textwidth}
\begin{tikzpicture}
\tikzset{every node/.style={font=\footnotesize}}
\tikzset{red node/.style = {fill=red!20!white, draw=red!75!white}}
\tikzset{red line/.style = {line width=1ex, red,nearly transparent}} 
\matrix [matrix of math nodes,column sep={0.48cm,between origins},
row sep={1cm,between origins},
nodes={circle, draw=blue!50,fill=blue!20, thick, inner sep = 0pt , minimum size=1.2mm}]
{
& & & & & \node[label = {above:{4321}}] (4321) {} ; & & & & & \\
& & & 
\node[label = {above left:4312}] (4312) {} ; & & 
\node[label = {above left:4231}] (4231) {} ; & & 
\node[label = {above right:3421}] (3421) {} ; & & & \\
& \node[label = {above left:4132}] (4132) {} ; & & 
\node[label = {left:4213}] (4213) {} ; & & 
\node[label = {above:3412}, red node] (3412) {} ; & & 
\node[label = {[label distance = 0.1cm]0:2431}] (2431) {} ; & & 
\node[label = {above right:3241}] (3241) {} ; & \\
\node[label = {left:1432}, red node] (1432) {} ; & & 
\node[label = {left:4123}] (4123) {} ; & & 
\node[label = {[label distance = 0.01cm]180:2413}, red node] (2413) {} ; & & 
\node[label = {[label distance = 0.01cm]0:3142}, red node] (3142) {} ; & & 
\node[label = {right:2341}] (2341) {} ; & & 
\node[label = {right:3214}, red node] (3214) {} ; \\
& \node[label = {below left:1423}, red node] (1423) {} ; & & 
\node[label = {[label distance = 0.1cm]182:1342}, red node] (1342) {} ; & & 
\node[label = {below:2143}, red node] (2143) {} ; & & 
\node[label = {right:3124}, red node] (3124) {} ; & & 
\node[label = {below right:2314}, red node] (2314) {} ; & \\
& & & \node[label = {below left:1243}, red node] (1243) {} ; & & 
\node[label = {[label distance = 0.01cm]190:1324}, red node] (1324) {} ; & & 
\node[label = {below right:2134}, red node] (2134) {} ; & & & \\
& & & & & \node[label = {below:1234}, red node] (1234) {} ; & & & & & \\
};
\draw (4321)--(4312)--(4132)--(1432)--(1423)--(1243)--(1234)--(2134)--(2314)--(2341)--(3241)--(3421)--(4321);
\draw (4321)--(4231)--(4132);
\draw (4231)--(3241);
\draw (4231)--(2431);
\draw (4231)--(4213);
\draw (4312)--(4213)--(2413)--(2143)--(3142)--(3241);
\draw (4312)--(3412)--(2413)--(1423)--(1324)--(1234);
\draw (3421)--(3412)--(3214)--(3124)--(1324);
\draw (3421)--(2431)--(2341)--(2143)--(2134);
\draw (4132)--(4123)--(1423);
\draw (4132)--(3142)--(3124)--(2134);
\draw (4213)--(4123)--(2143)--(1243);
\draw (4213)--(3214);
\draw (3412)--(1432)--(1342)--(1243);
\draw (2431)--(1432);
\draw (2431)--(2413)--(2314);
\draw (3142)--(1342)--(1324);
\draw (4123)--(3124);
\draw (2341)--(1342);
\draw (2314)--(1324);
\draw (3412)--(3142);
\draw (3241)--(3214)--(2314);
\draw[red line] (1234)--(1243)--(2143)--(2413);
\draw[red line] (1234)--(2134)--(2143)--(3142);
\draw[red line] (1243)--(1423);
\draw[red line] (1243)--(1342);
\draw[red line] (1234)--(1324);
\draw[red line] (3124)--(2134)--(2314);
\draw[red line] (1324)--(1423)--(1432)--(3412);
\draw[red line] (1324)--(2314)--(3214)--(3412);
\draw[red line] (1324)--(1342)--(1432);
\draw[red line] (2314)--(2413)--(3412);
\draw[red line] (1423)--(2413);
\draw[red line] (1342)--(3142)--(3412);
\draw[red line] (1324)--(3124)--(3214);
\draw[red line] (3124)--(3142);
\end{tikzpicture}
\end{subfigure}
\begin{subfigure}[c]{0.45\textwidth}
\begin{tikzpicture}[scale=5.5]
\tikzset{every node/.style={draw=blue!50,fill=blue!20, circle, thick, inner sep=1pt,font=\footnotesize}}
\tikzset{red node/.style = {fill=red!20!white, draw=red!75!white}}
\tikzset{red line/.style = {line width=0.3ex, red, nearly opaque}}
\coordinate (3142) at (1/3, 1/2, 1/6); 
\coordinate (4231) at (2/3, 1/2, 1/6); 
\coordinate (4312) at (5/6, 2/3, 1/2); 
\coordinate (4321) at (5/6, 1/2, 1/3); 
\coordinate (3421) at (5/6, 1/3, 1/2); 
\coordinate (4213) at (2/3, 5/6, 1/2); 
\coordinate (1324) at (1/3, 1/2, 5/6); 
\coordinate (2413) at (2/3, 1/2, 5/6); 
\coordinate (3412) at (5/6, 1/2, 2/3); 
\coordinate (2314) at (1/2, 2/3, 5/6); 
\coordinate (4123) at (1/2, 5/6, 1/3); 
\coordinate (4132) at (1/2, 2/3, 1/6); 
\coordinate (3214) at (1/2, 5/6, 2/3); 
\coordinate (3124) at (1/3, 5/6, 1/2); 
\coordinate (2431) at (2/3, 1/6, 1/2); 
\coordinate (1432) at (1/2, 1/6, 2/3); 
\coordinate (1423) at (1/2, 1/3, 5/6); 
\coordinate (1342) at (1/3, 1/6, 1/2); 
\coordinate (2341) at (1/2, 1/6, 1/3); 
\coordinate (3241) at (1/2, 1/3, 1/6);
\coordinate (1243) at (1/6, 1/3, 1/2); 
\coordinate (2143) at (1/6, 1/2, 1/3); 
\coordinate (1234) at (1/6, 1/2, 2/3); 
\coordinate (2134) at (1/6, 2/3, 1/2); 
\draw[thick, draw=blue!70] (4213)--(4312)--(3412)--(2413)--(2314)--(3214)--cycle;
\draw[thick, draw=blue!70] (4312)--(4321)--(3421)--(3412);
\draw[thick, draw=blue!70] (3421)--(2431)--(1432)--(1423)--(2413);
\draw[thick, draw=blue!70] (1423)--(1324)--(2314);
\draw[thick, draw=blue!70] (1432)--(1342)--(1243)--(1234)--(1324);
\draw[thick, draw=blue!70] (1234)--(2134)--(3124)--(3214);
\draw[thick, draw=blue!70] (3124)--(4123)--(4213); 
\draw[thick, draw=blue!70, dashed] (2134)--(2143)--(3142)--(4132)--(4123);
\draw[thick, draw=blue!70, dashed] (2143)--(1243);
\draw[thick, draw=blue!70, dashed] (3142)--(3241)--(2341)--(1342);
\draw[thick, draw=blue!70, dashed] (2341)--(2431);
\draw[thick, draw=blue!70, dashed] (3241)--(4231)--(4132);
\draw[thick, draw=blue!70, dashed] (4231)--(4321);

\node [label = {[label distance = 0cm]left:$\vertex{1234}$}] at (1234) {};
\node[label = {[label distance = 0cm]left:$\vertex{1243}$}] at (1243) {};
\node[label = {[label distance = 0cm]right:$\vertex{1324}$}] at (1324) {};
\node[label = {[label distance = 0cm]left:$\vertex{1342}$}] at (1342) {};
\node [label = {[label distance = 0cm]above:$\vertex{1423}$},red node] at (1423) {};
\node[label = {[label distance = -0.2cm]below:$\vertex{1432}$}] at (1432) {};
\node [label = {[label distance = 0cm]left:$\vertex{2134}$}] at (2134) {};
\node[label = {[label distance = -0.1cm]below right:$\vertex{2143}$}] at (2143) {};
\node[label = {[label distance = 0cm]below:$\vertex{2314}$}] at (2314) {};
\node[label = {[label distance = 0cm]right:$\vertex{2341}$}] at (2341) {};
\node[label = {[label distance = 0cm]left:$\vertex{2413}$}] at (2413) {};
\node[label = {[label distance = -0.2cm]below:$\vertex{2431}$}] at (2431) {};
\node[label = {[label distance = -0.2cm]above:$\vertex{3124}$}] at (3124) {};
\node[label = {[label distance = -0.2cm]above:$\vertex{3142}$}] at (3142) {};
\node[label = {[label distance = -0.2cm]above:$\vertex{3214}$}] at (3214) {};
\node [label = {[label distance = -0.1cm]above:$\vertex{3241}$}] at (3241) {};
\node[label = {[label distance = 0cm]below left:$\vertex{3412}$}] at (3412) {};
\node[label = {[label distance = 0cm]right:$\vertex{3421}$}] at (3421) {};
\node[label = {[label distance = -0.2cm]above:$\vertex{4123}$}] at (4123) {};
\node [label = {[label distance = 0cm]below:$\vertex{4132}$}] at (4132) {};
\node[label = {[label distance = 0cm]right:$\vertex{4213}$}] at (4213) {};
\node[label = {[label distance = 0cm]left:$\vertex{4231}$}] at (4231) {};
\node[label = {[label distance = 0cm]right:$\vertex{4312}$}] at (4312) {};
\node [label = {[label distance = 0cm]right:$\vertex{4321}$}] at (4321) {};

\draw[red line] (3214)--(2314)--(2413)--(3412)--cycle;
\draw[red line] (2314)--(1324)--(1423)--(2413);
\draw[red line] (1234)--(1243)--(1342)--(1432)--(1423);
\draw[red line] (1432)--(3412);
\draw[red line] (1324)--(1234)--(2134)--(3124);
\draw[red line](3124)--(3214);
\draw[red line, dashed] (3124)--(3142)--(2143)--(1243);
\draw[red line, dashed] (3142)--(3412);
\draw[red line, dashed] (2134)--(2143);
\draw[red line, dashed] (3142)--(1342);

\foreach \x in {3124, 3214, 2134, 3142, 2143, 2314, 1234, 1243, 1324, 2413, 3412, 1342, 1423, 1432}{
\node[red node] at (\x) {};
}

\end{tikzpicture}
\end{subfigure}
\caption{The Bruhat interval $[e,3412]$ and the polytope $\Q_{3412}$}\label{fig:BIP-3412}
\end{figure}
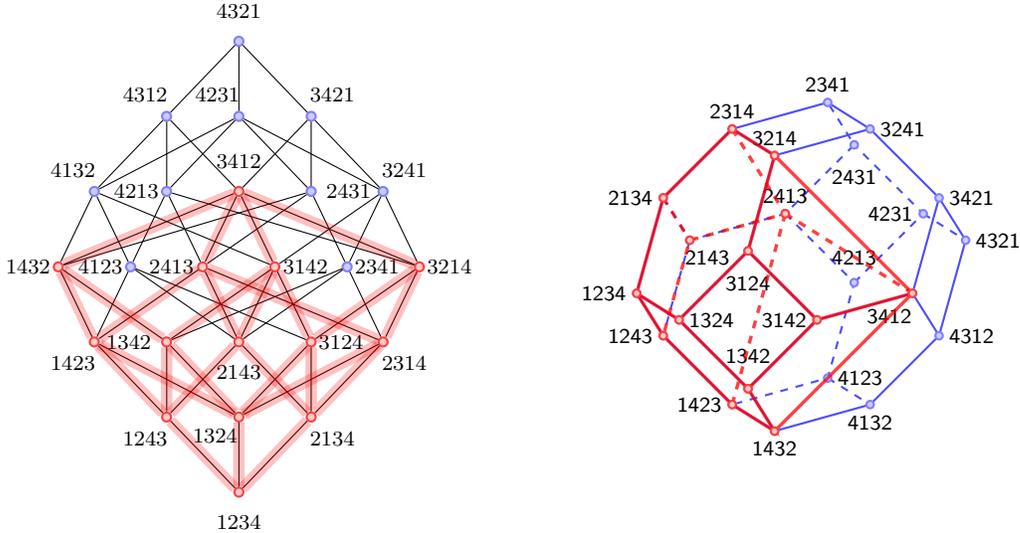

\begin{remark}
Note that the polytope $\Q_{w}$ is the $\Phi$-polytope $\Delta_{[e,w^{-1}]}$ in Subsection~\ref{section:Coxeter matroids}, where $W=S_n$ and $\nu=(1,2,\dots,n)$.
\end{remark}

By definition, $Y_w$ is the toric variety whose fan is the normal fan of the polytope $\Q_{w}$. 
The polytope $\Q_{w}$ is not simple in general. For example, the polytope $\Q_{3412}$ is $3$-dimensional, but there are four facets meeting at the vertex $\vertex{3412}$ (see Figure~\ref{fig:BIP-3412}).
At each vertex of $\Q_{w}$, the primitive direction vectors of the edges emanating from the vertex are root vectors, i.e.,  of the form $\mathbf{e}_i-\mathbf{e}_j$. Therefore, we obtain the following. 

\begin{proposition}[{\cite[\S8]{LM2020}}]
For $w \in S_n$, the toric variety $Y_w$ is smooth if and only if the polytope~$\Q_{w}$ is simple. 
\end{proposition}

To each vertex of $\Q_{w}$ we associate a graph which detects the simpleness of the vertex. 
For $u\le w$, we set
\[
\widetilde{E}_w(u) =\{ (u(i),u(j))\mid i<j,\ t_{u(i),u(j)}u\le w,\ |\ell(u)-\ell(t_{u(i),u(j)}u)|=1\},
\]
where $t_{a,b}$ denotes the transposition interchanging $a$ and $b$.\footnote{
There is a relation between this set and the $T$-weights of
$T_{uB}(G/B)$ which are described in~\cite{BL20Singular}.}
Then the digraph~$\widetilde{\Gamma}_w(u)$ defined by the vertex set $[n]$ and the edge set $\widetilde{E}_w(u)$ is acyclic by~\cite[Theorem~4.19]{TsukermanWilliams}, so it has a unique transitive reduction (see~\cite{a-ga-ul72}). Here, a \defi{transitive reduction} of a digraph is a digraph with the same vertices and as fewer edges as possible, such that if there is a directed path connecting two vertices, then there is also such a path in the reduction.  We denote the transitive reduction of $\widetilde{\Gamma}_w(u)$ by $\Gamma_w(u)$ and the edge set of $\Gamma_w(u)$ by $E_w(u)$. 

\begin{example}[{\cite[Example~3.2]{LMPS_Poincare}}]
Take $w = 3412$ and $u = 2143$. Then we have 
\[
\widetilde{E}_w(u) = \{(1,4), (2,3), (2,1), (4,3)\}.
\] 
The corresponding digraph $\widetilde{\Gamma}_w(u)$ and its transitive reduction $\Gamma_w(u)$ are presented as follows.
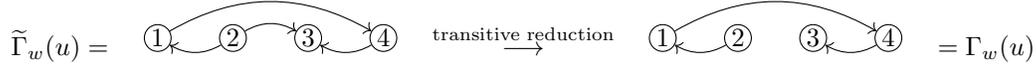
\begin{center}
$\widetilde{\Gamma}_w(u)=\quad $
\begin{tikzpicture}[node/.style={circle,draw, fill=white!20, inner sep = 0.25mm}]
\node[node] (1) at (1,1) {$1$};
\node[node] (2) at (2,1) {$2$};
\node[node] (3) at (3,1) {$3$};
\node[node] (4) at (4,1) {$4$};
\draw[->] (1) to [bend left] (4);
\draw[->] (4) to [bend left] (3);
\draw[->] (2) to [bend left] (1);
\draw[->] (2) to [bend left] (3);
\end{tikzpicture}
$\quad \stackrel{\text{transitive reduction}}{\longrightarrow} \quad $
\begin{tikzpicture}[node/.style={circle,draw, fill=white!20, inner sep = 0.25mm}]
\node[node] (1) at (1,1) {$1$};
\node[node] (2) at (2,1) {$2$};
\node[node] (3) at (3,1) {$3$};
\node[node] (4) at (4,1) {$4$};
\draw[->] (1) to [bend left] (4);
\draw[->] (4) to [bend left] (3);
\draw[->] (2) to [bend left] (1);
\end{tikzpicture} $\quad = \Gamma_w(u)$
\end{center}
\medskip
As a result, we have $E_w(u) = \{(1,4), (2,1), (4,3)\}$.
\end{example}

\begin{example}\label{example_Ew0u}
For the longest element $w_0$, the condition $t_{u(i),u(j)}u \leq w_0$ is always satisfied. Therefore, we have 
\[
\widetilde{E}_{w_0}(u) = \{(u(i), u(j)) \mid i<j,\ |\ell(u) - \ell(t_{u(i),u(j)}u)| = 1 \}.
\]
One can easily check that 
\[
\{(u(i),u(i+1)) \mid i=1,\dots,n-1\} \subset \widetilde{E}_{w_0}(u).
\]
Indeed, ${E}_{w_0}(u) = \{(u(i),u(i+1)) \mid i=1,\dots,n-1\}$.
\end{example}

\begin{proposition}[{\cite[Proposition 7.7]{LM2020}}]\label{prop_edges_of_Qw}
For $w\in S_n$, two vertices $\vertex{u}$ and $\vertex{v}$ of $\Q_{w}$ are joined by an edge if and only if 
\[
v=t_{u(i),u(j)}u(=ut_{i,j}) \quad \text{ for }(u(i),u(j))\in E_w(u).
\] 
\end{proposition}
The proposition above says that the elements in $E_w(u)$ correspond to the edges of $\Q_{w}$ emanating from the vertex $\vertex{u}$. Therefore, one can read from the graph $\Gamma_w(u)$ the number of edges emanating from the vertex $\vertex{u}$ in the polytope $\Q_{w}$. 
\begin{theorem}[{\cite[Theorem~1.2]{LM2020}}]
The generic torus orbit closure $Y_w$ in a Schubert variety $X_w$ is smooth at a fixed point $uB$ in $Y_w$ if and only if the graph $\Gamma_w(u)$ is a forest. Therefore, $Y_w$ is smooth if and only if $\Gamma_w(u)$ is a forest for every $u \leq w$.
\end{theorem}

The graph $\Gamma_w(w)$ appears in a slightly different viewpoint.  Indeed, it is shown in \cite{Woo-Yong} that the Schubert variety $X_w$ is locally factorial if and only if $\Gamma_w(w)$ is a forest.  This motivated the authors in~\cite{bm-bu07} to study the graph and they provided a necessary and sufficient condition for $\Gamma_w(w)$ to be a forest in terms of the pattern avoidance.\footnote{We say that a permutation $w \in S_n$ \emph{contains the pattern $q = q_1q_2 \cdots q_k$} if there is a $k$-element set of indices $1 \leq i_1 < i_2 < \cdots < i_k \leq n$ so that $w(i_r) < w(i_s)$ if and only if $q_r < q_s$ for $1 \leq r < s  \leq k$. If a permutation $w$ does not contain $q$, then we say that $w$ \emph{avoids the pattern $q$}. Similarly, $w$ avoids the pattern $45\bar{3}12$ if every occurrence of the pattern $4512$ is a subsequence of an occurrence of $45312$.}
Combining the theorem above with \cite[Theorem~1.1]{bm-bu07}, we obtain the following. 
\begin{proposition}
The following statements are equivalent.
\begin{enumerate}
\item $\Gamma_w(w)$ is a forest.
\item $w$ avoids the patterns $4231$ and $45\bar{3}12$.
\item $Y_w$ is smooth at $wB$. 
\item $\Q_w$ is simple at $\vertex{w}$.
\end{enumerate}
\end{proposition}

About the smoothness of $Y_w$, it is enough to check the smoothness at the fixed point $wB$ as is shown in the following theorem. It was originally conjectured in~\cite{LM2020}.
\begin{theorem}[{\cite[Theorem~B]{Gaetz22_one-skeleton}}] \label{conj:schubert}
The graph $\Gamma_w(u)$ is a forest for any $u \leq w $ if $\Gamma_w(w)$ is a forest. \textup{(}This is equivalent to saying that the generic torus orbit closure $Y_w$ in the Schubert variety $X_w$ is smooth if it is smooth at the fixed point $wB$, in other words, $\Q_{w}$ is simple if $\Q_w$ is simple at the vertex $\vertex{w}$.\textup{)}
\end{theorem}


\subsection{Generalized Eulerian polynomials}
In this subsection, we introduce a polynomial $A_w(t)$ for each $w\in \Sn$, which agrees with the Eulerian polynomial $A_n(t)$ when $w=w_0$, and explain that the Poincar\'e polynomial of $\Yw$ is given by $A_w(t^2)$. The tool we use to compute the Poincar\'e polynomial of $Y_w$ is discrete Morse theory.  We have the moment map 
\[
\mu\colonequals Y_w\to \Q_{w}\subset \R^n.
\]
The polytope $\Q_{w}$ is not necessarily simple, but we can find a linear function on $\R^n$ by which we decompose $\Q_{w}$ into pieces homeomorphic to orthants. 

Let $\mathsf{P}$ be a polytope in $\R^n$ and 
$h\colon \R^n\to \R$ a linear function satisfying that $h(\mathsf{u}) \neq h(\mathsf{v})$ if two vertices $\mathsf{u}$ and~$\mathsf{v}$ of $\mathsf{P}$ are joined by an edge in $\mathsf{P}$. Then the function $h$ gives an orientation on each edge of~$\mathsf{P}$, namely, we give an orientation of the edge connecting two vertices $\mathsf{u}$ and $\mathsf{v}$ of~$\mathsf{P}$ by $\mathsf{u} \to \mathsf{v}$ if $h(\mathsf{u}) < h(\mathsf{v})$. 
For each vertex $\mathsf{u}$ of~$\mathsf{P}$, we define 
\[
\asc(\mathsf{u})\colonequals \# \{ \mathsf{u} \to \mathsf{v} \}.
\]
The following is well-known for simple polytopes and is a consequence of discrete Morse theory. 

\begin{theorem}[see {\cite[Theorem~2.7]{LMPS_Poincare}} for example]\label{thm_poincare_polynomial_height_ftn}
Let $\mathsf{P} \subset \R^n$ be a lattice polytope.
Suppose that there exists a linear function $h \colon \R^n \to \R$ such that for each vertex $\mathsf{u}$ of $\mathsf{P}$, the direction vectors of ascending edges emanating from $\mathsf{u}$ are linearly independent and form a face $\mathsf{F}(\mathsf{u})$ of $\mathsf{P}$, that is, these edges form the set of edges emanating from $\mathsf{u}$ in $\mathsf{F}(\mathsf{u})$. Then the Poincar\'e polynomial $\Poin(X_{\mathsf P},t)$ of the toric variety $X_{\mathsf P}$ whose fan is the normal fan of $\mathsf P$ is given by 
\[
\Poin(X_{\mathsf P},t) =\sum_{\mathsf{u}\in V(\mathsf{P})}t^{2 \asc(\mathsf{u})},
\] 
where $V(\mathsf{P})$ is the set of all vertices of $\mathsf{P}$.
\end{theorem} 

\begin{remark}\label{rmk_existence_of_F}
The vertex $\mathsf{u}$ may not be simple in $\mathsf{P}$ but has to be simple in $\mathsf{F}(\mathsf{u})$.  Therefore the number of $k$-dimensional faces of $\mathsf{P}$ which are contained in $\mathsf{F}(\mathsf{u})$ and contain $\mathsf{u}$ is $\binom{\asc(\mathsf{u})}{k}$. This implies that 
\[
\sum_{\mathsf{u}\in V(\mathsf{P})}(1+t)^{2 \asc(\mathsf{u})}=\sum_{k=0}^{n}f_k(\mathsf{P})t^k \rotatebox[origin=c]{180}{~\colonequals~} f_{\mathsf{P}}(t),
\]
where $n=\dim\mathsf{P}$, $f_k(\mathsf{P})$ denotes the number of $k$-dimensional faces of $\mathsf{P}$, and $f_n(\mathsf{P})=1$. The polynomial $f_{\mathsf{P}}(t)$ is called the \emph{$f$-polynomial of $\mathsf{P}$} and the polynomial $h_{\mathsf{P}}(t)\colonequals f_{\mathsf{P}}(t-1)$ is called the \emph{$h$-polynomial of $\mathsf{P}$}. Theorem~\ref{thm_poincare_polynomial_height_ftn} can be restated as 
\begin{equation} \label{eq:h-polynomial_poincare}
\Poin(X_{\mathsf{P}},t)=h_{\mathsf{P}}(t^2)
\end{equation}
if there exists a linear function $h$ in Theorem~\ref{thm_poincare_polynomial_height_ftn}. When $\mathsf{P}$ is simple, such a linear function $h$ always exists and \eqref{eq:h-polynomial_poincare} is well-known in this case. However, it is not true that the formula \eqref{eq:h-polynomial_poincare} holds for every polytope $\mathsf{P}$.  For example, if $\mathsf{P}$ is an octahedron, then its $f$-polynomial is 
\[
f_{\mathsf{P}}(t)=6+12t+8t^2+t^3.
\] 
Therefore
\[
h_{\mathsf{P}}(t^2)=6+12(t^2-1)+8(t^2-1)^2+(t^2-1)^3=1-t^2+5t^4+t^6 
\]
which cannot be the Poincar\'e polynomial of $X_{\mathsf{P}}$ because it has a negative coefficient.  In fact, the polynomial $h_{\mathsf{P}}(t^2)$ is the {\it virtual Poincar\'e polynomial} of $X_{\mathsf{P}}$.  It is known that the virtual Poincar\'e polynomial agrees with the ordinary Poincar\'e polynomial for compact smooth toric varieties, see \cite[Section~4.5]{Fulton_toric}, but Theorem~\ref{Thm_Poincare} below implies that they agree for $Y_w$ although $Y_w$ is not necessarily smooth.  
\end{remark}
\begin{example}
Let $\mathsf{P}$ be the following pyramid with five vertices:
\[
\mathsf P = \Conv\{ (1,0,0), (0,1,0), (-1,0,0), (0,0,-1), (0,0,1) \}.
\] 
Choose a linear function $h \colon \R^3 \to \R$ defined by 
\[
h(x_1,x_2,x_3) = -2x_1 - x_2 + 3 x_3.
\] 
The function $h$ gives an orientation on edges of $\mathsf{P}$ as displayed in Figure~\ref{fig_pyramid}.
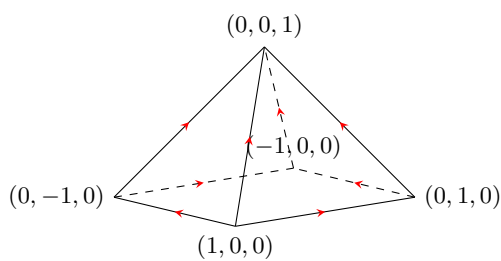
\begin{figure}[H]
\tikzset{ mid arrow/.style={postaction={decorate,decoration={
markings,
mark=at position .5 with {\arrow[#1]{stealth}}
}}}}
\begin{tikzpicture}[line join=bevel,z=-5.5, scale = 2]
\coordinate (A1) at (0,0,-1);
\coordinate (A2) at (-1,0,0);
\coordinate (A3) at (0,0,1);
\coordinate (A4) at (1,0,0);
\coordinate (B1) at (0,1,0);
\coordinate (C1) at (0,-1,0);
\node[above] at (A1) {\small $(-1,0,0)$};
\node[left] at (A2) {\small $(0,-1,0)$};
\node[below] at (A3) {\small $(1,0,0)$};
\node[right] at (A4) {\small $(0,1,0)$};
\node[above] at (B1) {\small$(0,0,1)$};
\draw[dashed, mid arrow=red] (A2)--(A1);
\draw[dashed, mid arrow=red] (A4)--(A1);
\draw[dashed, mid arrow=red] (A1)--(B1);
\draw [mid arrow=red] (A3)--(A4);
\draw [mid arrow=red] (A3)--(A2);
\draw [mid arrow = red] (A3)--(B1);
\draw [mid arrow = red] (A2)--(B1);
\draw [mid arrow = red] (A4)--(B1);
\end{tikzpicture}
\caption{The pyramid with an orientation}
\label{fig_pyramid}
\end{figure}
\noindent
Then, for each vertex~$\mathsf{u}$ of $\mathsf{P}$, the corresponding face $\mathsf{F}(\mathsf{u})$ and the number of ascending edges are as in Table~\ref{table_retraction_pyramid}. 
Therefore, the Poincar\'e polynomial $\Poin(X_{\mathsf P},t)$ of $X_{\mathsf P}$ is given by
\[
\Poin(X_{\mathsf P}, t) = 1 + t^2 + 2 t^4 + t^6. 
\]
\vspace{-0.5cm}
\tikzset{ mid arrow/.style={postaction={decorate,decoration={
markings,
mark=at position .5 with {\arrow[#1]{stealth}}
}}}}
\begin{table}[H]
\renewcommand{\arraystretch}{1.2}
\begin{tabular}{c|c|c|c|c|c}
\toprule
\raisebox{1em}{$\mathsf{F}(\mathsf{u})$} & \begin{tikzpicture}[line join=bevel,z=-5.5, scale = 0.9]
\coordinate (A1) at (0,0,-1);
\coordinate (A2) at (-1,0,0);
\coordinate (A3) at (0,0,1);
\coordinate (A4) at (1,0,0);
\coordinate (B1) at (0,1,0);
\coordinate (C1) at (0,-1,0);
\fill[fill=blue!95!black, fill opacity=0.100000] (A2)--(A3)--(A4)--(B1)--(A2);
\draw[dashed, mid arrow=red] (A2)--(A1);
\draw[dashed, mid arrow=red] (A4)--(A1);
\draw[dashed, mid arrow=red] (A1)--(B1);
\draw [mid arrow=red] (A3)--(A4);
\draw [mid arrow=red] (A3)--(A2);
\draw [mid arrow = red] (A3)--(B1);
\draw [mid arrow = red] (A2)--(B1);
\draw [mid arrow = red] (A4)--(B1);
\node[inner sep=1pt,circle,draw=green!25!black,fill=green!75!black,thick] at (A3) {};
\end{tikzpicture}
& \begin{tikzpicture}[line join=bevel,z=-5.5, scale = 0.9,
edge/.style={color=blue!95!black, thick},
facet/.style={fill=blue!95!black,fill opacity=0.100000}]
\coordinate (A1) at (0,0,-1);
\coordinate (A2) at (-1,0,0);
\coordinate (A3) at (0,0,1);
\coordinate (A4) at (1,0,0);
\coordinate (B1) at (0,1,0);
\coordinate (C1) at (0,-1,0);
\filldraw [fill=blue!95!black, fill opacity=0.100000, draw = blue!95!black, thick] (A1)--(A4)--(B1)--cycle;
\draw[dashed, mid arrow=red] (A2)--(A1);
\draw[dashed, mid arrow=red] (A4)--(A1);
\draw[dashed, mid arrow=red] (A1)--(B1);
\draw [mid arrow=red] (A3)--(A4);
\draw [mid arrow=red] (A3)--(A2);
\draw [mid arrow = red] (A3)--(B1);
\draw [mid arrow = red] (A2)--(B1);
\draw [mid arrow = red] (A4)--(B1);
\node[inner sep=1pt,circle,draw=green!25!black,fill=green!75!black,thick] at (A4) {};
\end{tikzpicture}
&
\begin{tikzpicture}[line join=bevel,z=-5.5, scale = 0.9]
\coordinate (A1) at (0,0,-1);
\coordinate (A2) at (-1,0,0);
\coordinate (A3) at (0,0,1);
\coordinate (A4) at (1,0,0);
\coordinate (B1) at (0,1,0);
\coordinate (C1) at (0,-1,0);
\filldraw [fill=blue!95!black, fill opacity=0.100000, draw = blue!95!black, thick](A1)--(A2)--(B1)--cycle;
\draw[dashed, mid arrow=red] (A2)--(A1);
\draw[dashed, mid arrow=red] (A4)--(A1);
\draw[dashed, mid arrow=red] (A1)--(B1);
\draw [mid arrow=red] (A3)--(A4);
\draw [mid arrow=red] (A3)--(A2);
\draw [mid arrow = red] (A3)--(B1);
\draw [mid arrow = red] (A2)--(B1);
\draw [mid arrow = red] (A4)--(B1);
\node[inner sep=1pt,circle,draw=green!25!black,fill=green!75!black,thick] at (A2) {};
\end{tikzpicture}
&
\begin{tikzpicture}[line join=bevel,z=-5.5, scale = 0.9]
\coordinate (A1) at (0,0,-1);
\coordinate (A2) at (-1,0,0);
\coordinate (A3) at (0,0,1);
\coordinate (A4) at (1,0,0);
\coordinate (B1) at (0,1,0);
\coordinate (C1) at (0,-1,0);
\draw[fill=blue!95!black, fill opacity=0.100000, draw = blue!95!black, thick] (A1)--(B1);
\draw[dashed, mid arrow=red] (A2)--(A1);
\draw[dashed, mid arrow=red] (A4)--(A1);
\draw[dashed, mid arrow=red] (A1)--(B1);
\draw [mid arrow=red] (A3)--(A4);
\draw [mid arrow=red] (A3)--(A2);
\draw [mid arrow = red] (A3)--(B1);
\draw [mid arrow = red] (A2)--(B1);
\draw [mid arrow = red] (A4)--(B1);
\node[inner sep=1pt,circle,draw=green!25!black,fill=green!75!black,thick] at (A1) {};
\end{tikzpicture}
& \begin{tikzpicture}[line join=bevel,z=-5.5, scale = 0.9]
\coordinate (A1) at (0,0,-1);
\coordinate (A2) at (-1,0,0);
\coordinate (A3) at (0,0,1);
\coordinate (A4) at (1,0,0);
\coordinate (B1) at (0,1,0);
\coordinate (C1) at (0,-1,0);
\draw[dashed, mid arrow=red] (A2)--(A1);
\draw[dashed, mid arrow=red] (A4)--(A1);
\draw[dashed, mid arrow=red] (A1)--(B1);
\draw [mid arrow=red] (A3)--(A4);
\draw [mid arrow=red] (A3)--(A2);
\draw [mid arrow = red] (A3)--(B1);
\draw [mid arrow = red] (A2)--(B1);
\draw [mid arrow = red] (A4)--(B1);
\node[inner sep=1pt,circle,draw=green!25!black,fill=green!75!black,thick] at (B1) {};
\end{tikzpicture}
\\
\midrule
$h(\mathsf{u})$ & $-2$ & $-1$ & $1$ & $2$ & $3$ \\ 
\hline
$\asc(\mathsf{u})$ & $3$ & $2$ & $2$ & $1$ & $0$\\
\bottomrule
\end{tabular}
\caption{Faces $\mathsf{F}(\mathsf{u})$ and $\asc(\mathsf{u})$ for the pyramid}
\label{table_retraction_pyramid}
\end{table}
\vspace{-0.3cm}
On the other hand, since $(f_0(\mathsf{P}),f_1(\mathsf{P}),f_2(\mathsf{P}),f_3(\mathsf{P}))=(5,8,5,1)$ for the pyramid $\mathsf{P}$, we have 
\[
h_{\mathsf{P}}(t^2)=5+8(t^2-1)+5(t^2-1)^2+(t^2-1)^3=1+t^2+2t^4+t^6.
\]
Therefore the formula \eqref{eq:h-polynomial_poincare} certainly holds in this case.  
\end{example}

We shall apply Theorem~\ref{thm_poincare_polynomial_height_ftn} to $\Q_{w}$ to find the Poincar\'e polynomial of $\Yw$. What we have to do is to find a linear function $h$ on $\R^n$ which satisfies the conditions in Theorem~\ref{thm_poincare_polynomial_height_ftn}.

\begin{lemma}[{\cite[Lemma~3.4]{LMPS_Poincare}}] \label{lemm:2-1}
Let $h\colon \R^n\to\R$ be a linear function defined by the inner product with a vector $(a_1,\dots,a_n) \in \R^n$ with $a_1>a_2>\dots>a_n$. Then, for $(u(i),u(j))\in E_w(u)$, the edge emanating from the vertex $\vertex{u}$ to the vertex $\vertex{v}$ in $\Q_{w}$, where $v=t_{u(i),u(j)}u$, is ascending with respect to the function~$h$ if and only if $u(i)<u(j)$.
Here, we consider orientations on edges of $\PP$ according to the function $h$.
\end{lemma}

Motivated by Lemma~\ref{lemm:2-1}, we define 
\begin{equation*}
\begin{split}
E_w(u)^+&=\{ (u(i),u(j))\in E_w(u)\mid u(i)<u(j)\},\\
E_w(u)^-&=\{ (u(i),u(j))\in E_w(u)\mid u(i)>u(j)\}.
\end{split}
\end{equation*}
Then 
\[
E_w(u)=E_w(u)^+\sqcup E_w(u)^-.
\]
Note that $E_w(u)^+$ (resp. $E_w(u)^-$) corresponds to the ascending (resp. descending) edges emanating from $\vertex{u}$ by Lemma~\ref{lemm:2-1} with respect to the function $h$. 
Now we set 
\begin{equation*}\label{eq_def_of_awu}
a_w(u) \colonequals |E_w(u)^+|
\end{equation*}
and define
\begin{equation*}\label{eq_def_of_Aw}
A_w(t) \colonequals \sum_{u \leq w} t^{a_w(u)}.
\end{equation*}
\begin{example}\label{example_aw0u_and_ascents}
By Example~\ref{example_Ew0u}, we have
\[
\begin{split}
E_{w_0}(u)^+ &= \{(u(i),u(j)) \in E_{w_0}(u) \mid u(i) < u(j)\} \\
&= \{(u(i),u(i+1)) \mid u(i) < u(i+1) \text{ for } i = 1,\dots,n-1\}.
\end{split}
\]
Accordingly, $a_{w_0}(u)$ is the number of ascents in $u$ and 
\[
A_{w_0}(t) = A_n(t).
\]
On the other hand, $X_{w_0}=\flag(n)$ and $Y_{w_0}$ is the permutohedral variety $\Perm_n$. Therefore, we have 
\[
\Poin(Y_{w_0},t)=A_{w_0}(t^2)
\]
by Theorem~\ref{theo:perm&eulerian}. 
\end{example}

One can check that the linear function $h$ in Lemma~\ref{lemm:2-1} satisfies the conditions in Theorem~\ref{thm_poincare_polynomial_height_ftn} so that we obtain the following theorem extending Theorem~\ref{theo:perm&eulerian}. 

\begin{theorem}[{\cite[Theorem~3.6]{LMPS_Poincare}}]\label{Thm_Poincare}
Let $\Yw$ be the generic torus orbit closure in the Schubert variety~$\Xw$ for $w \in \Sn$. Then the Poincar\'e polynomial of $\Yw$ is given by 
\[
\Poin(\Yw,t) =A_w(t^2).
\] 
\end{theorem}

We close this subsection by presenting the Poincar\'{e} polynomials of singular toric varieties $Y_{4231}$ and $Y_{3412}$ using Theorem~\ref{Thm_Poincare}. The moment polytopes of $Y_{4231}$ and $Y_{3412}$ are non-simple polytopes~$\Q_{4231}$ and $\Q_{3412}$, respectively, see Figure~\ref{figure_BIP}.
In both polytopes, we take $h \colon \R^4 \to \R$ by the inner product with a vector $(12,2,-1,-2)$. For instance, the image of the vertex $\vertex{4132} = (2,4,3,1)$ under $h$ is given by
\[
h(4132) = \langle (2,4,3,1), (12,2,-1,-2) \rangle = 2\times 12 + 4 \times 2 - 3 - 2 = 27.
\]
We number each vertex $\vertex{u}$ as $h(\mu(uB))$. 
Therefore, the vertex $\vertex{4132}=(2,4,3,1)$ is numbered as $27$ (see Figure~\ref{figure_Bruhat_polytope_4231}).
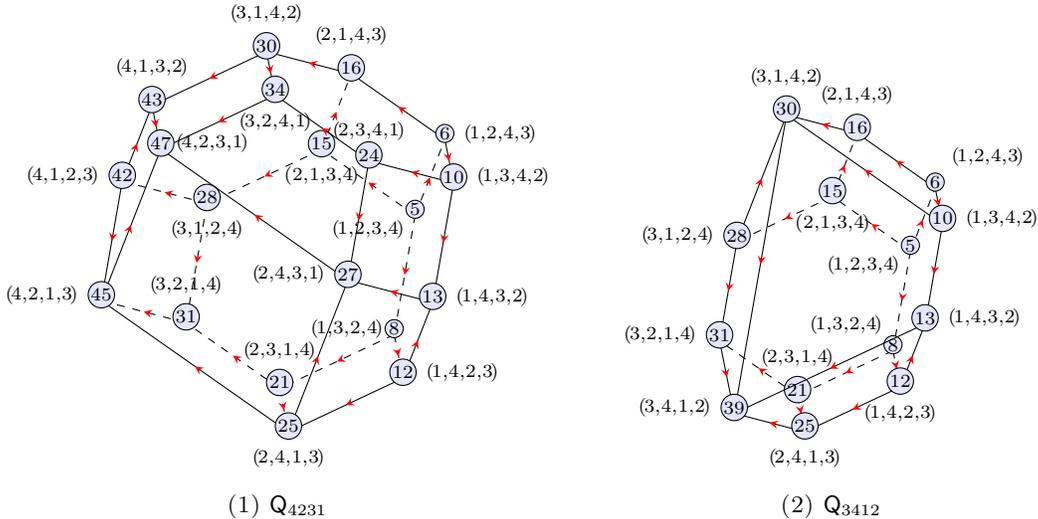
\begin{figure}[H]
\begin{subfigure}[b]{0.5\textwidth}
\centering
\tikzset{ mid arrow/.style={postaction={decorate,decoration={
markings,
mark=at position .5 with {\arrow[#1]{stealth}}
}}}} \begin{tikzpicture}%
[x={(-0.939161cm, 0.244762cm)},
y={(0.097442cm, -0.482887cm)},
z={(0.329367cm, 0.840780cm)},
scale=1.20000,
back/.style={dashed, thin},
edge/.style={color=black},
facet/.style={fill=none},
vertex/.style={inner sep=0.3pt,circle, draw=black, fill=blue!10,anchor=base, font=\scriptsize},
vertex2/.style={inner sep=1.5pt,circle,draw=green!25!black,fill=red,thick,anchor=base},
every label/.append style={text=black, font=\scriptsize}]
%
%
\coordinate (1.00000, 2.00000, 3.00000) at (1.00000, 2.00000, 3.00000);
\coordinate (1.00000, 2.00000, 4.00000) at (1.00000, 2.00000, 4.00000);
\coordinate (1.00000, 3.00000, 2.00000) at (1.00000, 3.00000, 2.00000);
\coordinate (1.00000, 3.00000, 4.00000) at (1.00000, 3.00000, 4.00000);
\coordinate (1.00000, 4.00000, 2.00000) at (1.00000, 4.00000, 2.00000);
\coordinate (1.00000, 4.00000, 3.00000) at (1.00000, 4.00000, 3.00000);
\coordinate (2.00000, 1.00000, 3.00000) at (2.00000, 1.00000, 3.00000);
\coordinate (2.00000, 1.00000, 4.00000) at (2.00000, 1.00000, 4.00000);
\coordinate (2.00000, 3.00000, 1.00000) at (2.00000, 3.00000, 1.00000);
\coordinate (2.00000, 3.00000, 4.00000) at (2.00000, 3.00000, 4.00000);
\coordinate (2.00000, 4.00000, 1.00000) at (2.00000, 4.00000, 1.00000);
\coordinate (2.00000, 4.00000, 3.00000) at (2.00000, 4.00000, 3.00000);
\coordinate (3.00000, 1.00000, 2.00000) at (3.00000, 1.00000, 2.00000);
\coordinate (3.00000, 1.00000, 4.00000) at (3.00000, 1.00000, 4.00000);
\coordinate (3.00000, 2.00000, 1.00000) at (3.00000, 2.00000, 1.00000);
\coordinate (3.00000, 2.00000, 4.00000) at (3.00000, 2.00000, 4.00000);
\coordinate (4.00000, 1.00000, 2.00000) at (4.00000, 1.00000, 2.00000);
\coordinate (4.00000, 1.00000, 3.00000) at (4.00000, 1.00000, 3.00000);
\coordinate (4.00000, 2.00000, 1.00000) at (4.00000, 2.00000, 1.00000);
\coordinate (4.00000, 2.00000, 3.00000) at (4.00000, 2.00000, 3.00000);
\draw[edge,back, mid arrow = red] (1.00000, 2.00000, 3.00000) -- (1.00000, 2.00000, 4.00000);
\draw[edge,back, mid arrow = red] (1.00000, 2.00000, 3.00000) -- (1.00000, 3.00000, 2.00000);
\draw[edge,back, mid arrow = red] (1.00000, 2.00000, 3.00000) -- (2.00000, 1.00000, 3.00000);
\draw[edge,back, mid arrow = red] (1.00000, 3.00000, 2.00000) -- (1.00000, 4.00000, 2.00000);
\draw[edge,back, mid arrow = red] (1.00000, 3.00000, 2.00000) -- (2.00000, 3.00000, 1.00000);
\draw[edge,back, mid arrow = red] (2.00000, 1.00000, 3.00000) -- (2.00000, 1.00000, 4.00000);
\draw[edge,back, mid arrow = red] (2.00000, 1.00000, 3.00000) -- (3.00000, 1.00000, 2.00000);
\draw[edge,back, mid arrow = red] (2.00000, 3.00000, 1.00000) -- (2.00000, 4.00000, 1.00000);
\draw[edge,back, mid arrow = red] (2.00000, 3.00000, 1.00000) -- (3.00000, 2.00000, 1.00000);
\draw[edge,back, mid arrow = red] (3.00000, 1.00000, 2.00000) -- (3.00000, 2.00000, 1.00000);
\draw[edge,back, mid arrow = red] (3.00000, 1.00000, 2.00000) -- (4.00000, 1.00000, 2.00000);
\draw[edge,back, mid arrow = red] (3.00000, 2.00000, 1.00000) -- (4.00000, 2.00000, 1.00000);
\node[vertex, label={[label distance = -3mu] left:{$(\mspace{-2mu}1,\!3,\!2,\!4\mspace{-2mu})$}}] at (1.00000, 3.00000, 2.00000) {$8$};
\node[vertex, label={above:{$(\mspace{-2mu}2,\!3,\!1,\!4\mspace{-2mu})$}}] at (2.00000, 3.00000, 1.00000) {$21$};
\node[vertex,label={[label distance = -3mu] below left:{$(\mspace{-2mu}1,\!2,\!3,\!4\mspace{-2mu})$}}] at (1.00000, 2.00000, 3.00000) {$5$};
\node[vertex,label={below:{$(2,\!1,\!3,\!4\mspace{-2mu})$}}] at (2.00000, 1.00000, 3.00000) {$15$};
\node[vertex,label={above:{$(\mspace{-2mu}3,\!2,\!1,\!4\mspace{-2mu})$}}] at (3.00000, 2.00000, 1.00000) {$31$};
\node[vertex,label={below:{$(\mspace{-2mu}3,\!1,\!2,\!4\mspace{-2mu})$}}] at (3.00000, 1.00000, 2.00000) {$28$};
\fill[facet] (4.00000, 2.00000, 3.00000) -- (2.00000, 4.00000, 3.00000) -- (2.00000, 4.00000, 1.00000) -- (4.00000, 2.00000, 1.00000) -- cycle {};
\fill[facet] (4.00000, 2.00000, 3.00000) -- (2.00000, 4.00000, 3.00000) -- (2.00000, 3.00000, 4.00000) -- (3.00000, 2.00000, 4.00000) -- cycle {};
\fill[facet] (4.00000, 2.00000, 3.00000) -- (3.00000, 2.00000, 4.00000) -- (3.00000, 1.00000, 4.00000) -- (4.00000, 1.00000, 3.00000) -- cycle {};
\fill[facet] (4.00000, 2.00000, 3.00000) -- (4.00000, 1.00000, 3.00000) -- (4.00000, 1.00000, 2.00000) -- (4.00000, 2.00000, 1.00000) -- cycle {};
\fill[facet] (2.00000, 4.00000, 3.00000) -- (1.00000, 4.00000, 3.00000) -- (1.00000, 3.00000, 4.00000) -- (2.00000, 3.00000, 4.00000) -- cycle {};
\fill[facet] (2.00000, 4.00000, 3.00000) -- (1.00000, 4.00000, 3.00000) -- (1.00000, 4.00000, 2.00000) -- (2.00000, 4.00000, 1.00000) -- cycle {};
\fill[facet] (3.00000, 2.00000, 4.00000) -- (2.00000, 3.00000, 4.00000) -- (1.00000, 3.00000, 4.00000) -- (1.00000, 2.00000, 4.00000) -- (2.00000, 1.00000, 4.00000) -- (3.00000, 1.00000, 4.00000) -- cycle {};
\draw[edge, mid arrow = red] (1.00000, 2.00000, 4.00000) -- (1.00000, 3.00000, 4.00000);
\draw[edge, mid arrow = red] (1.00000, 2.00000, 4.00000) -- (2.00000, 1.00000, 4.00000);
\draw[edge, mid arrow = red] (1.00000, 3.00000, 4.00000) -- (1.00000, 4.00000, 3.00000);
\draw[edge, mid arrow = red] (1.00000, 3.00000, 4.00000) -- (2.00000, 3.00000, 4.00000);
\draw[edge, mid arrow = red] (1.00000, 4.00000, 2.00000) -- (1.00000, 4.00000, 3.00000);
\draw[edge, mid arrow = red] (1.00000, 4.00000, 2.00000) -- (2.00000, 4.00000, 1.00000);
\draw[edge, mid arrow = red] (1.00000, 4.00000, 3.00000) -- (2.00000, 4.00000, 3.00000);
\draw[edge, mid arrow = red] (2.00000, 1.00000, 4.00000) -- (3.00000, 1.00000, 4.00000);
\draw[edge, mid arrow = red] (2.00000, 3.00000, 4.00000) -- (2.00000, 4.00000, 3.00000);
\draw[edge, mid arrow = red] (2.00000, 3.00000, 4.00000) -- (3.00000, 2.00000, 4.00000);
\draw[edge, mid arrow = red] (2.00000, 4.00000, 1.00000) -- (2.00000, 4.00000, 3.00000);
\draw[edge, mid arrow = red] (2.00000, 4.00000, 1.00000) -- (4.00000, 2.00000, 1.00000);
\draw[edge, mid arrow = red] (2.00000, 4.00000, 3.00000) -- (4.00000, 2.00000, 3.00000);
\draw[edge, mid arrow = red] (3.00000, 1.00000, 4.00000) -- (3.00000, 2.00000, 4.00000);
\draw[edge, mid arrow = red] (3.00000, 1.00000, 4.00000) -- (4.00000, 1.00000, 3.00000);
\draw[edge, mid arrow = red] (3.00000, 2.00000, 4.00000) -- (4.00000, 2.00000, 3.00000);
\draw[edge, mid arrow = red] (4.00000, 1.00000, 2.00000) -- (4.00000, 1.00000, 3.00000);
\draw[edge, mid arrow = red] (4.00000, 1.00000, 2.00000) -- (4.00000, 2.00000, 1.00000);
\draw[edge, mid arrow = red] (4.00000, 1.00000, 3.00000) -- (4.00000, 2.00000, 3.00000);
\draw[edge, mid arrow = red] (4.00000, 2.00000, 1.00000) -- (4.00000, 2.00000, 3.00000);
\node[vertex,label={right:{$(\mspace{-2mu}1,\!2,\!4,\!3\mspace{-2mu})$}}] at (1.00000, 2.00000, 4.00000) {$6$};
\node[vertex,label={right:{$(\mspace{-2mu}1,\!3,\!4,\!2\mspace{-2mu})$}}] at (1.00000, 3.00000, 4.00000) {$10$};
\node[vertex,label={right:{$(\mspace{-2mu}1,\!4,\!2,\!3\mspace{-2mu})$}}] at (1.00000, 4.00000, 2.00000) {$12$};
\node[vertex,label={right:{$(\mspace{-2mu}1,\!4,\!3,\!2\mspace{-2mu})$}}] at (1.00000, 4.00000, 3.00000) {$13$};
\node[vertex,label={above:{$(\mspace{-2mu}2,\!1,\!4,\!3\mspace{-2mu})$}}] at (2.00000, 1.00000, 4.00000) {$16$};
\node[vertex,label={ [label distance = -3mu] above :{$(\mspace{-2mu}2,\!3,\!4,\!1\mspace{-2mu})$}}] at (2.00000, 3.00000, 4.00000) {$24$};
\node[vertex,label={below:{$(\mspace{-2mu}2,\!4,\!1,\!3\mspace{-2mu})$}}] at (2.00000, 4.00000, 1.00000) {$25$};
\node[vertex,label={ left:{$(\mspace{-2mu}2,\!4,\!3,\!1\mspace{-2mu})$}}] at (2.00000, 4.00000, 3.00000) {$27$};
\node[vertex,label={above:{$(\mspace{-2mu}3,\!1,\!4,\!2\mspace{-2mu})$}}] at (3.00000, 1.00000, 4.00000) {$30$};
\node[vertex,label={below:{$(\mspace{-2mu}3,\!2,\!4,\!1\mspace{-2mu})$}}] at (3.00000, 2.00000, 4.00000) {$34$};
\node[vertex,label={left:{$(\mspace{-2mu}4,\!1,\!2,\!3\mspace{-2mu})$}}]at (4.00000, 1.00000, 2.00000) {$42$};
\node[vertex,label={above:{$(\mspace{-2mu}4,\!1,\!3,\!2\mspace{-2mu})$}}] at (4.00000, 1.00000, 3.00000) {$43$};
\node[vertex,label={left:{$(\mspace{-2mu}4,\!2,\!1,\!3\mspace{-2mu})$}}] at (4.00000, 2.00000, 1.00000) {$45$};
\node[vertex,label={ [label distance = -3mu] 0:{$(\mspace{-2mu}4,\!2,\!3,\!1\mspace{-2mu})$}}] at (4.00000, 2.00000, 3.00000) {$47$};
\end{tikzpicture}%
\caption{$\Q_{4231}$}
\label{figure_Bruhat_polytope_4231}
\end{subfigure}%
\hspace{1.5em}
\begin{subfigure}[b]{0.4\textwidth}
\centering
\tikzset{ mid arrow/.style={postaction={decorate,decoration={
markings,
mark=at position .5 with {\arrow[#1]{stealth}}
}}}}
\begin{tikzpicture}%
[x={(-0.939161cm, 0.244762cm)},
y={(0.097442cm, -0.482887cm)},
z={(0.329367cm, 0.840780cm)},
scale=1.00000,
back/.style={dashed, thin},
edge/.style={color=black},
facet/.style={fill=none},
vertex/.style={inner sep=0.3pt,circle, draw=black, fill=blue!10,anchor=base, font=\scriptsize},
vertex2/.style={inner sep=1.5pt,circle,draw=green!25!black,fill=red,thick,anchor=base},
every label/.append style={text=black, font=\scriptsize}]
%
%
\coordinate (1.00000, 2.00000, 3.00000) at (1.00000, 2.00000, 3.00000);
\coordinate (1.00000, 2.00000, 4.00000) at (1.00000, 2.00000, 4.00000);
\coordinate (1.00000, 3.00000, 2.00000) at (1.00000, 3.00000, 2.00000);
\coordinate (1.00000, 3.00000, 4.00000) at (1.00000, 3.00000, 4.00000);
\coordinate (1.00000, 4.00000, 2.00000) at (1.00000, 4.00000, 2.00000);
\coordinate (1.00000, 4.00000, 3.00000) at (1.00000, 4.00000, 3.00000);
\coordinate (2.00000, 1.00000, 3.00000) at (2.00000, 1.00000, 3.00000);
\coordinate (2.00000, 1.00000, 4.00000) at (2.00000, 1.00000, 4.00000);
\coordinate (2.00000, 3.00000, 1.00000) at (2.00000, 3.00000, 1.00000);
\coordinate (2.00000, 4.00000, 1.00000) at (2.00000, 4.00000, 1.00000);
\coordinate (3.00000, 1.00000, 2.00000) at (3.00000, 1.00000, 2.00000);
\coordinate (3.00000, 1.00000, 4.00000) at (3.00000, 1.00000, 4.00000);
\coordinate (3.00000, 2.00000, 1.00000) at (3.00000, 2.00000, 1.00000);
\coordinate (3.00000, 4.00000, 1.00000) at (3.00000, 4.00000, 1.00000);
\draw[edge,back, mid arrow = red] (1.00000, 2.00000, 3.00000) -- (1.00000, 2.00000, 4.00000);
\draw[edge,back, mid arrow = red] (1.00000, 2.00000, 3.00000) -- (1.00000, 3.00000, 2.00000);
\draw[edge,back, mid arrow = red] (1.00000, 2.00000, 3.00000) -- (2.00000, 1.00000, 3.00000);
\draw[edge,back, mid arrow = red] (1.00000, 3.00000, 2.00000) -- (1.00000, 4.00000, 2.00000);
\draw[edge,back, mid arrow = red] (1.00000, 3.00000, 2.00000) -- (2.00000, 3.00000, 1.00000);
\draw[edge,back, mid arrow = red] (2.00000, 1.00000, 3.00000) -- (2.00000, 1.00000, 4.00000);
\draw[edge,back, mid arrow = red] (2.00000, 1.00000, 3.00000) -- (3.00000, 1.00000, 2.00000);
\draw[edge,back, mid arrow = red] (2.00000, 3.00000, 1.00000) -- (2.00000, 4.00000, 1.00000);
\draw[edge,back, mid arrow = red] (2.00000, 3.00000, 1.00000) -- (3.00000, 2.00000, 1.00000);
\node[vertex, label={[label distance = -3mu] above left:{$(\mspace{-2mu}1,\!3,\!2,\!4\mspace{-2mu})$}}] at (1.00000, 3.00000, 2.00000) {$8$};
\node[vertex, label={above:{$(\mspace{-2mu}2,\!3,\!1,\!4\mspace{-2mu})$}}] at (2.00000, 3.00000, 1.00000) {$21$};
\node[vertex,label={[label distance = -3mu] below left:{$(\mspace{-2mu}1,\!2,\!3,\!4\mspace{-2mu})$}}] at (1.00000, 2.00000, 3.00000) {$5$};
\node[vertex,label={below:{$(2,\!1,\!3,\!4\mspace{-2mu})$}}] at (2.00000, 1.00000, 3.00000) {$15$};
\fill[facet] (3.00000, 4.00000, 1.00000) -- (1.00000, 4.00000, 3.00000) -- (1.00000, 3.00000, 4.00000) -- (3.00000, 1.00000, 4.00000) -- cycle {};
\fill[facet] (3.00000, 4.00000, 1.00000) -- (1.00000, 4.00000, 3.00000) -- (1.00000, 4.00000, 2.00000) -- (2.00000, 4.00000, 1.00000) -- cycle {};
\fill[facet] (3.00000, 4.00000, 1.00000) -- (3.00000, 1.00000, 4.00000) -- (3.00000, 1.00000, 2.00000) -- (3.00000, 2.00000, 1.00000) -- cycle {};
\fill[facet] (3.00000, 1.00000, 4.00000) -- (1.00000, 3.00000, 4.00000) -- (1.00000, 2.00000, 4.00000) -- (2.00000, 1.00000, 4.00000) -- cycle {};
\draw[edge, mid arrow = red] (1.00000, 2.00000, 4.00000) -- (1.00000, 3.00000, 4.00000);
\draw[edge, mid arrow = red] (1.00000, 2.00000, 4.00000) -- (2.00000, 1.00000, 4.00000);
\draw[edge, mid arrow = red] (1.00000, 3.00000, 4.00000) -- (1.00000, 4.00000, 3.00000);
\draw[edge, mid arrow = red] (1.00000, 3.00000, 4.00000) -- (3.00000, 1.00000, 4.00000);
\draw[edge, mid arrow = red] (1.00000, 4.00000, 2.00000) -- (1.00000, 4.00000, 3.00000);
\draw[edge, mid arrow = red] (1.00000, 4.00000, 2.00000) -- (2.00000, 4.00000, 1.00000);
\draw[edge, mid arrow = red] (1.00000, 4.00000, 3.00000) -- (3.00000, 4.00000, 1.00000);
\draw[edge, mid arrow = red] (2.00000, 1.00000, 4.00000) -- (3.00000, 1.00000, 4.00000);
\draw[edge, mid arrow = red] (2.00000, 4.00000, 1.00000) -- (3.00000, 4.00000, 1.00000);
\draw[edge, mid arrow = red] (3.00000, 1.00000, 2.00000) -- (3.00000, 1.00000, 4.00000);
\draw[edge, mid arrow = red] (3.00000, 1.00000, 2.00000) -- (3.00000, 2.00000, 1.00000);
\draw[edge, mid arrow = red] (3.00000, 1.00000, 4.00000) -- (3.00000, 4.00000, 1.00000);
\draw[edge, mid arrow = red] (3.00000, 2.00000, 1.00000) -- (3.00000, 4.00000, 1.00000);
\node[vertex,label={left:{$(\mspace{-2mu}3,\!2,\!1,\!4\mspace{-2mu})$}}] at (3.00000, 2.00000, 1.00000) {$31$};
\node[vertex,label={left:{$(\mspace{-2mu}3,\!1,\!2,\!4\mspace{-2mu})$}}] at (3.00000, 1.00000, 2.00000) {$28$}; 
\node[vertex,label={above right:{$(\mspace{-2mu}1,\!2,\!4,\!3\mspace{-2mu})$}}] at (1.00000, 2.00000, 4.00000) {$6$};
\node[vertex,label={ right:{$(\mspace{-2mu}1,\!3,\!4,\!2\mspace{-2mu})$}}] at (1.00000, 3.00000, 4.00000) {$10$};
\node[vertex,label={below:{$(\mspace{-2mu}1,\!4,\!2,\!3\mspace{-2mu})$}}] at (1.00000, 4.00000, 2.00000) {$12$};
\node[vertex,label={right:{$(\mspace{-2mu}1,\!4,\!3,\!2\mspace{-2mu})$}}] at (1.00000, 4.00000, 3.00000) {$13$};
\node[vertex,label={above:{$(\mspace{-2mu}2,\!1,\!4,\!3\mspace{-2mu})$}}] at (2.00000, 1.00000, 4.00000) {$16$};
\node[vertex,label={below:{$(\mspace{-2mu}2,\!4,\!1,\!3\mspace{-2mu})$}}] at (2.00000, 4.00000, 1.00000) {$25$};
\node[vertex,label={above:{$(\mspace{-2mu}3,\!1,\!4,\!2\mspace{-2mu})$}}] at (3.00000, 1.00000, 4.00000) {$30$};
\node[vertex,label={left:{$(\mspace{-2mu}3,\!4,\!1,\!2\mspace{-2mu})$}}] at (3.00000, 4.00000, 1.00000) {$39$};
\end{tikzpicture}
\caption{$\Q_{3412}$}
\label{figure_Bruhat_polytope_3412}
\end{subfigure}
\caption{Polytopes $\Q_{4231}$ and $\Q_{3412}$}
\label{figure_BIP}
\end{figure}
Using this function $h$, we get the Poincar\'e polynomials of these singular generic torus orbit closures as follows:
\[
\begin{split}
\Poin(Y_{4231},t) &= 1 + 7 t^2 + 11t^4 + t^6, \\
\Poin(Y_{3412},t) &= 1 + 5 t^2 + 7 t^4 + t^6.
\end{split}
\]
One can easily confirm the formula \eqref{eq:h-polynomial_poincare} in these two cases.  Indeed, the formula \eqref{eq:h-polynomial_poincare} holds for every $Y_w$, which implies  
\[
A_w(t)=h_{\mathsf{Q}_{w}}(t) \quad \text{for every $w\in \Sn$}.
\]

\begin{remark}
\begin{enumerate}
\item 
For a rational polytope $\mathsf{P}$ allowing a \emph{retraction sequence}, the Poincar\'e polynomial of the toric variety $X_\mathsf{P}$ is the polynomial $h_{\mathsf{P}}(t^2)$. We refer the reader to~\cite{BSS} for a precise definition of retraction sequences.
It is shown in~\cite{PS_conic} that a linear function $h$ as in Lemma~\ref{lemm:2-1} defines a retraction sequence on the polytope~$\Q_w$. However, it is not known whether all retraction sequences come from linear functions.
\item In~\cite[Section~4.1]{PostnikovReinerWilliams08}, a generalized Eulerian polynomial is defined for a {\it simple} generalized permutohedron $\mathsf{P}$ and it agrees with the $h$-polynomial of $\mathsf{P}$ (see \cite[Theorem~4.2]{PostnikovReinerWilliams08}).  On the other hand, our generalized Eulerian polynomial $A_w(t)$ for $w\in S_n$  agrees with the $h$-polynomial of a Burhat interval polytope $\mathsf{Q}_w$ as mentioned above, where $\mathsf{Q}_w$ is a generalized permutohedron but not necessarily simple.   
\end{enumerate}
\end{remark}

\subsection{Toric Schubert varieties}\label{section_toric_Schubert}
The complexity of a $T$-variety $X$ is the codimension of a maximal dimensional $T$-orbit in $X$. A toric variety is of complexity $0$ and in this case there is a nice relation between geometry and combinatorics (see Appendix~\ref{sec_toric_varieties}). Indeed, a toric variety has a fruitful symmetry given by the $T$-action and we can extract the essential geometric information (e.g., smoothness) from the corresponding combinatorial object. One can expect a similar nice relation when the complexity is small enough. In this section, we present such kind of results for Schubert varieties of complexity $0$ or $1$.

\begin{definition}\label{def:c(w)}
For $w \in \Sn$, we define 
\[
c(w) \colonequals \dim_\C X_w-\dim_\C Y_w=\dim_{\C} X_w - \dim_{\R} \Q_{w}=\ell(w)- \dim_{\R} \Q_{w},
\]
and call $c(w)$ the \textit{complexity} of the Schubert variety $X_w$ (or the \emph{complexity} of $w$). 
\end{definition}
When $c(w) =0$, we have $X_w=Y_w$; so the Schubert variety $X_w$ is a toric variety. 
Let $w = s_{i_1} \cdots s_{i_m}$ be a reduced decomposition of $w$. 
Then it is known from~\cite{Fan98, Karu13Schubert} that the Schubert variety $X_w$ is a toric variety if and only if $i_1,\dots,i_m$ are distinct. Moreover, all toric Schubert varieties are Bott manifolds, which are smooth projective toric varieties whose moment map images are combinatorially equivalent to cubes.
We refer the reader to Appendix~\ref{appendix} for more details on Bott manifolds. Furthermore, all toric Schubert varieties are Bott--Samelson varieties which are smooth projective varieties. We will explain them in detail in Section~\ref{subsection_complexity_one}.

There are several equivalent conditions characterizing toric Schubert varieties as follows.
\begin{theorem} \cite{Fan98,Karu13Schubert,tenn07,Tenner2012,LMP2021} \label{thm:toric}The following statements are equivalent:
\begin{enumerate}
\item[$(0)$] $X_w$ is a toric variety \textup{(}i.e.,  of complexity zero\textup{)}.
\item[$(1)$] $X_{w}$ is a smooth toric variety.
\item[$(2)$] $w$ avoids the patterns $321$ and $3412$.
\item[$(3)$] A reduced decomposition of $w$ consists of distinct letters.
\item[$(4)$] $X_w$ is isomorphic to a Bott--Samelson variety.
\item[$(5)$] The Bruhat interval $[e,w]$ is isomorphic to the Boolean algebra $\mathfrak{B}_{\ell(w)}$ of rank $\ell(w)$ as posets.
\item[$(6)$] $\Q_{w}$ is combinatorially equivalent to the cube of dimension~$\ell(w)$.
\end{enumerate}
\end{theorem}

\lee{Here, we say that a poset $P$ is a \defi{Boolean algebra} if there is a set $X$ such that $P$ is isomorphic to the set of all subsets of $X$, partially ordered by inclusion.}

The fan of a toric Schubert variety $X_w$ is the normal fan of the polytope $\Q_{w}$. Since $\Q_{w}$ is combinatorially equivalent to a cube as stated in Theorem~\ref{thm:toric}, we shall find its normal fan by investigating the primitive direction vectors of the edges emanating from the vertices $\vertex{e}$ and $\vertex{w}$ of~$\Q_{w}$. 
Let $w = s_{i_1} \cdots s_{i_m}$ be a reduced decomposition of $w \in \Sn$ and we assume that $i_1,\dots,i_m$ are distinct. Then $w^{-1}=s_{i_m}\cdots s_{i_1}$ and the Bruhat interval $[e, w^{-1}]$ has $m$ many atoms and coatoms: 
\begin{equation} \label{eq:atom_coatom}
\begin{split}
\text{atoms: } & s_{i_k}\quad (k=1,\dots,m), \\
\text{coatoms: } & s_{i_m} \cdots \hat{s}_{i_k} \cdots s_{i_1}\quad (k=1,\dots,m).
\end{split}
\end{equation}

\begin{lemma}\label{lemma_atoms_and_coatoms_Schubert}
Let $w = s_{i_1} \cdots s_{i_m}$. For $1 \leq k \leq m$, we set 
\[
a_k = s_{i_1}s_{i_2} \cdots s_{i_{k-1}}(i_k), \quad b_k = s_{i_1} s_{i_2} \cdots s_{i_{k-1}} (i_k+1)
\]
where we understand $a_1=i_1$ and $b_1=i_1+1$. Then we have 
\[
s_{i_m} \cdots \hat{s}_{i_k} \cdots s_{i_1} = w^{-1} t_{a_k,b_k}.
\]
\end{lemma}
\begin{proof}
For any $v \in \Sn$ and transposition $t_{p,q}$, we have 
\begin{equation}\label{eq_conjugation}
v t_{p,q} v^{-1} = t_{v(p), v(q)}.
\end{equation}
Noting that $s_{i_k}=t_{i_k,i_k+1}$ and $s_j^{-1}=s_j$, we apply~\eqref{eq_conjugation} to the right hand side of the equation
\[
s_{i_m} \cdots \hat{s}_{i_k} \cdots s_{i_1} = w^{-1} s_{i_1} s_{i_2} \cdots s_{i_{k-1}} s_{i_k} s_{i_{k-1}} \cdots s_{i_2} s_{i_1}.
\]
Then the lemma immediately follows. 
\end{proof}

The \defi{Cartan integers} $c_{i,j}$, which are entries of the Cartan matrix, are given by
\begin{equation}\label{eq_Cartan_integers}
c_{i,j} = \begin{cases}
2 & \text{ if } i = j, \\
-1 & \text{ if } |i-j| = 1,\\
0 & \text{ otherwise}.
\end{cases}
\end{equation}

\begin{theorem}\label{thm_char_matrix_of_toric_Schubert}
Let $w = s_{i_1} \cdots s_{i_m}$ be a reduced decomposition of $w \in \Sn$. 
Assume that $i_1,\dots,i_m$ are distinct. Then the fan of the toric Schubert variety $X_w$ is isomorphic to the fan in $\R^m$ such that the primitive ray vectors are the $2m$ column vectors of the following matrix and a subset of the column vectors forms a cone if and only if it does not contain both the $i$th column vectors in the left submatrix and the right submatrix for each $i=1,\dots,m$: 
\[
\left[ 
\begin{array}{cccc|cccc}
1 & 0 & \cdots & 0 & -1 & 0&\cdots &0 \\
0 & 1 & \cdots & 0 & & -1 & \cdots & 0\\
\vdots & & \ddots & & &a_{j,k} & \ddots & \\
0 & 0 & \cdots & 1 & & & & -1
\end{array}
\right],
\]
where $a_{j,k} = -c_{i_j,i_k}$ for $1 \leq k<j \leq m $. Here, $c_{i,j}$ are Cartan integers \textup{(}see~\eqref{eq_Cartan_integers}\textup{)}.
\end{theorem}
\begin{proof}
Since the fan of $X_w$ is the normal fan of $\Q_{w}$ and $\Q_{w}$ is combinatorially equivalent to an $m$-dimensional cube by Theorem~\ref{thm:toric}, it is enough to consider the edge vectors emanating from the vertices $\vertex{e}$ and $\vertex{w}$ to find the fan of $X_w$ as below. 

Let $\mathbf{e}_1,\dots,\mathbf{e}_n$ be the standard basis vectors of $\R^n$ and we set
\begin{equation} \label{eq:vi}
\mathbf v_i \colonequals \mathbf{e}_i-\mathbf{e}_{i+1} \quad \text{ for }1\leq i \leq n-1.
\end{equation}
The dual vector space to the subspace of $\R^n$ spanned by $\mathbf{v}_1,\dots,\mathbf{v}_{n-1}$ is the quotient space $\R^n/{\R (1,\dots,1)}$ and 
\begin{equation*} \label{eq:vi*}
\mathbf v_i^{\ast} \colonequals \mathbf{e}_1 + \cdots + \mathbf{e}_i\quad \text{ for }1\leq i \leq n-1
\end{equation*}
form the basis of $\R^n/{\R (1,\dots,1)}$ dual to \eqref{eq:vi} through the standard scalar product on $\R^n$. 

At the vertex $\vertex{e}$, the primitive vectors of outgoing edges are $\mathbf{v}_{i_1},\dots,\mathbf{v}_{i_m}$ because the atoms of the Bruhat interval $[e,w^{-1}]$ are $s_{i_k}$'s by \eqref{eq:atom_coatom}. Hence the primitive facet normal vectors at $\vertex{e}$ are $\mathbf{v}_{i_1}^{\ast},\dots,\mathbf{v}_{i_m}^{\ast}$.
On the other hand, in order to find the facet normal vectors at the vertex $\vertex{w}$, we consider the following two elements by \eqref{eq:atom_coatom} and Lemma~\ref{lemma_atoms_and_coatoms_Schubert}: 
\[
\arraycolsep=1.4pt
\begin{array}{rcccccccc}
w^{-1} t_{a_k,b_k} = & w^{-1}(1) & w^{-1}(2) & \cdots & w^{-1}(b_k) & \cdots & w^{-1}(a_k) & \cdots & w^{-1}(n), \\
w^{-1} = & w^{-1}(1) & w^{-1}(2) & \cdots & w^{-1}(a_k) & \cdots & w^{-1}(b_k) & \cdots & w^{-1}(n).
\end{array}
\]
Since $w^{-1} > w^{-1} t_{a_k,b_k}$, we have $w^{-1}(a_k) > w^{-1}(b_k)$.
Moreover, we get $a_k < b_k$ because $a_k \le i_k$ and $b_k \ge i_k+1$ by the definition of $a_k$ and $b_k$ in Lemma~\ref{lemma_atoms_and_coatoms_Schubert}. 
Hence the primitive vector of the outgoing edge at $\vertex{w}$ to the vertex corresponding to $w^{-1} t_{a_k,b_k}$ is $- \mathbf{e}_{a_k} + \mathbf{e}_{b_k}$, which is the same as 
\[
- s_{i_1} s_{i_2} \cdots s_{i_{k-1}} (\mathbf{e}_{i_k} - \mathbf{e}_{i_{k}+1})=- s_{i_1} s_{i_2} \cdots s_{i_{k-1}} (\mathbf{v}_{i_k}).
\]
Here, we regard $s_i$ as the reflection in $\mathbb{R}^n$ which interchanges the $i$th and the $(i+1)$st coordinates, namely $s_i(\mathbf{v})=\mathbf{v}-(\mathbf{v},\mathbf{e}_i-\mathbf{e}_{i+1})(\mathbf{e}_i-\mathbf{e}_{i+1})$ for $\mathbf{v}\in \R^n$, where $(\: ,\: )$ denotes the standard scalar product on $\R^n$.  
Now we set 
\[
\mathbf{w}_k \colonequals - s_{i_1} s_{i_2} \cdots s_{i_{k-1}} (\mathbf{v}_{i_k})
\]
for $1 \leq k \leq m$.
Then,  since $c_{i,j}=(\mathbf{e}_i-\mathbf{e}_{i+1},\mathbf{e}_j-\mathbf{e}_{j+1})$, we have 
\[
\begin{split}
\mathbf{w}_k &= - s_{i_1} s_{i_2} \cdots s_{i_{k-1}} (\mathbf{v}_{i_k})\\
&= -s_{i_1} s_{i_2} \cdots s_{i_{k-2}}(\mathbf{v}_{i_k} - c_{i_k, i_{k-1}} \mathbf{v}_{i_{k-1}}) \\
&= -s_{i_1} s_{i_2} \cdots s_{i_{k-3}}(\mathbf{v}_{i_k} - c_{i_k, i_{k-2}} \mathbf{v}_{i_{k-2}}) - c_{i_k, i_{k-1}} \mathbf{w}_{k-1} \\
&\quad\vdots\\
&= -\mathbf{v}_{i_k} - c_{i_k, i_1} \mathbf{w}_1 - c_{i_k, i_2} \mathbf{w}_{2} - \cdots - c_{i_k, i_{k-2}} \mathbf{w}_{k-2} - c_{i_k, i_{k-1}} \mathbf{w}_{k-1}. 
\end{split}
\]
The facet normal vectors at $\vertex{w}$ are the dual basis $\mathbf{w}_1^*,\dots, \mathbf{w}_m^*$ to $\mathbf{w}_1,\dots, \mathbf{w}_m$, and the above computation implies that 
\[
\mathbf{w}_k^*=-\mathbf{v}_{i_k}^*-c_{i_{k+1},i_k}\mathbf{v}^*_{i_{k+1}}-\cdots-c_{i_m,i_k}\mathbf{v}^*_{i_m}.
\]
This proves the theorem.
\end{proof}

We call the right half of the matrix in Theorem~\ref{thm_char_matrix_of_toric_Schubert} the \emph{reduced characteristic matrix} of $X_w$. 

\begin{example}\label{exam:1234}
For $w = s_1 s_2 s_3s_4$, the reduced characteristic matrix of $X_{w}$ is 
\[
\begin{bmatrix}
-1 & 0 & 0 &0 \\
1 & -1 & 0 & 0 \\
0 & 1 & -1 & 0 \\
0 & 0 & 1 & -1
\end{bmatrix}
\]
In general, the reduced characteristic matrix of $X_w$ for $w = s_1 s_2 \cdots s_{n-1}$ (or $w=s_{n-1}s_{n-2}\cdots s_1$) has $1$'s just below the diagonal and $0$ at the other off diagonal entries.
The transpose inverse of this matrix is the upper triangular matrix with $-1$ on the diagonal and above the diagonal, which is the reduced characteristic of the bounded flag manifold $BF_{n-1}$ mentioned in~\cite[Proposition~7.7.3]{BP15toric}.  Since the isomorphism class of the variety associated to a reduced characteristic matrix does not change by taking inverse and transpose of the  reduced characteristic matrix, $X_w$ for the $w$ above  
is isomorphic to the bounded flag manifold $BF_{n-1}$. 
\end{example}

A reduced characteristic matrix of a toric Schubert variety $X_w$ depends on the choice of a reduced decomposition of $w$ as is seen in Example~\ref{exam:1324}, but it determines $X_w$ up to isomorphism. 

\begin{example} \label{exam:1324}
For $w = s_1 s_3 s_2 s_4$, the reduced characteristic matrix of $X_w$ is 
\begin{equation*}\label{eq_char_mat_1324_BIP}
\begin{bmatrix}
-1 & 0 & 0 &0 \\
0 & -1 & 0 & 0 \\
1 & 1 & -1 & 0 \\
0 & 1 & 0 & -1
\end{bmatrix}
\end{equation*}
There are four more reduced decompositions: $w = s_3 s_1 s_2 s_4 = s_1 s_3 s_4 s_2 = s_3 s_1 s_4 s_2 = s_3s_4s_1s_2$. For each case, we have the following reduced characteristic matrices.
\begin{center}
\begin{tabular}{cccc}
$s_3s_1s_2s_4 $ & $s_1s_3s_4s_2$ & $s_3s_1s_4s_2$ & $s_3s_4s_1s_2$ \\
$\scriptstyle\begin{bmatrix}
-1 & 0 & 0 & 0 \\
0 & -1 & 0 & 0 \\
1 & 1 & -1 & 0 \\
1 & 0 & 0 & -1
\end{bmatrix}$
& $
\scriptstyle\begin{bmatrix}
-1 & 0 & 0 & 0 \\
0 & -1 & 0 & 0 \\
0 & 1 & -1 & 0 \\
1 & 1 & 0 & -1
\end{bmatrix}
$ 
&
$\scriptstyle
\begin{bmatrix}
-1 & 0 & 0 & 0 \\
0 & -1 & 0 & 0 \\
1 & 0 & -1 & 0 \\
1 & 1 & 0 & -1
\end{bmatrix}
$
& $
\scriptstyle
\begin{bmatrix}
-1 & 0 & 0 & 0 \\
1 & -1 & 0 & 0 \\
0 & 0 & -1 & 0\\
1 & 0 & 1 & -1
\end{bmatrix}
$
\end{tabular}
\end{center}
Applying Proposition~\ref{prop:batyrev-iso} in Appendix~\ref{appendix} to the above five matrices, one can easily see that all of them determine the same toric variety up to isomorphism.
\end{example}

A reduced decomposition $w=s_{i_1}\cdots s_{i_m}$ defines a sequence $\mathbf i=(i_1,\dots,i_m)$ and vice versa. So we also call the sequence a reduced decomposition of $w$. We introduce a digraph $\G_{\mathbf i}$ associated with a sequence $\mathbf i$, which does not depend on the choice of a reduced decomposition of $w$, i.e.,  depends only on $w$. 

\begin{definition}
For $\mathbf i = (i_1,\dots,i_m)$, the vertex set and the edge set of a digraph $\G_{\mathbf i}$ are defined by
\begin{itemize}
\item $V(\G_{\mathbf i}) = [m]$;
\item $(k,j) \in E(\G_{\mathbf i})$ if and only if $|i_k - i_j| = 1$ for $1 \leq k < j \leq m$.
\end{itemize}
\end{definition}

\begin{example} \label{exam:digraph_Gi}
For $\mathbf i = (1,2,3,4)$ and $\mathbf i' = (1,3,2,4)$ corresponding to Examples~\ref{exam:1234} and \ref{exam:1324}, we have the following digraphs:
\begin{center}
$\mathcal G_{\mathbf i} = $
\begin{tikzpicture}[node/.style={circle,draw, fill=white!20, inner sep = 0.25mm}, baseline = -0.5ex]
\node[node] (1) at (1,0) {$1$};
\node[node] (2) at (2,0) {$2$};
\node[node] (3) at (3,0) {$3$};
\node[node] (4) at (4,0) {$4$};
\draw[->] (1) to (2);
\draw[->] (2) to (3);
\draw[->] (3) to (4);
\end{tikzpicture}
$\quad \quad$
$\mathcal G_{\mathbf i'} = $
\begin{tikzpicture}[node/.style={circle,draw, fill=white!20, inner sep = 0.25mm}, baseline = -0.5ex]
\node[node] (1) at (1,0) {$1$};
\node[node] (2) at (2,0) {$2$};
\node[node] (3) at (3,0) {$3$};
\node[node] (4) at (4,0) {$4$};
\draw[->] (1) to [bend left] (3);
\draw[->] (2) to (3);
\draw[->] (2) to [bend right] (4);
\end{tikzpicture}
\end{center}
For $\mathbf i=(1,2,4,5)$, the digraph $\mathcal G_{\mathbf i}$ is not connected as follows: 
\begin{center} 
$\mathcal G_{\mathbf i} = $
\begin{tikzpicture}[node/.style={circle,draw, fill=white!20, inner sep = 0.25mm}, baseline = -0.5ex]
\node[node] (1) at (1,0) {$1$};
\node[node] (2) at (2,0) {$2$};
\node[node] (3) at (3,0) {$3$};
\node[node] (4) at (4,0) {$4$};
\draw[->] (1) to (2);
\draw[->] (3) to (4);
\end{tikzpicture}
\end{center}
\end{example}

A smooth projective variety $X$ is called \emph{Fano} (resp. \emph{weak Fano}) if the anti-canonical divisor~$-K_X$ is ample (resp. nef and big). 
\park{Fano varieties play an important role as building blocks in the classification problem in algebraic geometry:}
\lee{the minimal model program. We refer the readers to~\cite{Kollar1996, Kollar_Rational92}. }

\lee{We consider the Fano or weak Fano characterization of toric Schubert varieties in terms of the digraphs.}
The digraphs in Example~\ref{exam:digraph_Gi} are unions of directed path graphs so that each vertex has at most two outgoing edges. This always holds when $\mathbf{i}$ has distinct components and it implies the following.     
\begin{theorem}
Let $\mathbf i = (i_1,\dots,i_m)$ be a reduced decomposition of $w \in \Sn$.
Suppose that $i_1,\dots,i_m$ are distinct. Then the toric Schubert variety $X_w$ is weak Fano.  Moreover, it is Fano if and only if each vertex of the graph $\G_{\mathbf i}$ has at most one outgoing edge.
\end{theorem}

\begin{proof}
We apply the criterion of toric Fano or toric weak Fano by Batyrev to our case, see Proposition~\ref{prop:batyrev} for the criterion.  It follows from Theorem~\ref{thm_char_matrix_of_toric_Schubert} that a primitive collection of the fan of $X_w$ consists of the $k$th column in the left matrix and the $k$th column in the right matrix in Theorem~\ref{thm_char_matrix_of_toric_Schubert} for each $k=1,\dots,m$. Then the criterion by Batyrev says that $X_w$ is Fano if and only if $1$ appears in the $k$th column at most once for each $k$ in the reduced characteristic matrix of $X_w$. Moreover, $X_w$ is weak Fano if and only if $1$ appears in the $k$th column at most twice for each $k$ in the reduced characteristic matrix of $X_w$.
This implies the theorem. 
\end{proof}

We assume $n\ge 3$ in the following. 
An element of $\Sn$ is called a \emph{Coxeter element} if it can be written as a product of all adjacent transpositions $s_1,\dots,s_{n-1}$. Let $\Cox_n$ denote the set of all Coxeter elements in $\Sn$. 

\begin{theorem}[\cite{LMP_directed_Dynkin}]\label{thm:dynkin}
Let $w, w' \in \Cox_n$ and let $\mathbf i, \mathbf i'$ be reduced decompositions of $w,w'$, respectively.
The following statements are equivalent:
\begin{enumerate}
\item $X_w$ and $X_{w'}$ are isomorphic as toric varieties \textup{(}with respect to the $T$-action on $G/B$\textup{)}.
\item $H^{\ast}(X_w;\Z) \cong H^{\ast}(X_{w'};\Z)$ as graded rings.
\item $\mathcal G_{\mathbf i} \cong \mathcal G_{\mathbf i'}$ as digraphs.
\item $w' = w$ or $w' = w_0 w w_0$.
\end{enumerate}
\end{theorem}
\begin{remark}
We notice that the isomorphism classes of Schubert varieties are studied in~\cite{richmond2021isomorphism}. \end{remark}

\begin{remark}
The digraph $\G_{\mathbf i}$ in the theorem above is a directed path graph with $n-1$ vertices. This can be thought of as a directed Dynkin diagram of type $A_{n-1}$. It turns out that directed Dynkin diagrams appear in the classification of toric Schubert varieties for other Lie types (see~\cite{LMP_directed_Dynkin}). 
\end{remark}


\subsection{Schubert varieties of complexity one}
\label{subsection_complexity_one}

In this subsection, we consider Schubert varieties of complexity one. All toric Schubert varieties are smooth while Schubert varieties of complexity one are not necessarily smooth. We provide similar statements for Schubert varieties of complexity one to Theorem~\ref{thm:toric}. 

We recall generalized Bott--Samelson varieties from~\cite{Jant}.
Let $G={\rm GL}_n(\C)$ as before. For a permutation $w \in \Sn$, the subvariety $P_w$ of $G$ corresponding to $w$ is defined by
\begin{equation*}
P_w =\overline{B w B} \subseteq G.
\end{equation*}

\begin{definition}[{cf. \cite[\S13.4]{Jant}, \cite[Definition~2.1]{FLS}, and \cite{GK94Bott}}]\label{def_g_Bott_Samelson}
Let $(w_1,\dots,w_r)$ be a sequence of elements in $\Sn$. The \emph{generalized Bott--Samelson variety} $Z_{(w_1,\dots,w_r)}$ is defined by the orbit space
\[
Z_{(w_1,\dots,w_r)} \colonequals (P_{w_1} \times \cdots \times P_{w_r})/\Theta,
\]
where the right action $\Theta$ of $B^r \colonequals \underbrace{B \times \cdots \times B}_{r}$ on $\prod_{k=1}^r P_{w_k}$ is defined by
\begin{equation*}\label{eq:action}
\Theta((p_1,\dots,p_r), (b_1,\dots,b_r))
= (p_1b_1, b_1^{-1}p_2b_2,\dots,b_{r-1}^{-1}p_rb_r)
\end{equation*}
for $(p_1,\dots,p_r) \in \prod_{k=1}^r P_{w_k}$ and $(b_1,\dots,b_r)\in B^r$.
When $\ell(w_1)=\dots=\ell(w_r)=1$, $Z_{(w_1,\dots,w_r)}$ is called a \emph{Bott--Samelson variety}. 
\end{definition}
We notice that there are simpler descriptions of (generalized) Bott--Samelson varieties in~\cite{Magyar98, EPTY18} in type $A$. 
\lee{We recall from~\cite{EPTY18} the description of Bott--Samelson varieties in type $A$ for readers' convenience. Suppose that a sequence $(w_1,\dots,w_r)$ of elements in $S_n$ consists of permutations having length one, that is, $w_k = s_{i_k}$ the simple reflection for some $i_k \in [n-1]$. Moreover, suppose that it represents a reduced word of a permutation. In other words, $s_{i_1} \cdots s_{i_r}$ is a reduced decomposition of a permutation. In this case, we simply denote the sequence by $\mathbf i  =(i_1,\dots,i_r)$. For a permutation $w$, the \emph{Elnitsky $2n$-gon} $E(w)$ has sides of length one, and these are labeled, in order, by $1,2,\dots,n,w(n),w(n-1),\dots,w(1)$, in which the first $n$ labels form half of a regular $2n$-gon, and sides with the same label are parallel. In Figure~\ref{Fig_Elnitsky_3241}, we provide the Elnitsky $8$-gon for the permutation $3241 \in S_4$. We denote by $v_0$ the lowest vertex $v_0$ of  $E(w)$. 
} 
	
\begin{figure}
\begin{subfigure}[t]{0.45\textwidth}
	\centering
	\begin{tikzpicture}[scale = 0.8]
			\foreach \x in {1,...,8}{ 
				\coordinate (\x) at (-\x*45-90:2cm) ;
				}
				
				\coordinate (A) at ($(1)+(67.5:1.531cm)$);
				\coordinate (B) at ($(A)+(-22.5:1.531cm)$);
				\coordinate (C) at ($(3) +(-22.5:1.531cm)$);

				\draw 
				(8)--(1)--(2)--(3)--(4)--(5)--(C)--(B)--cycle;

				\node[below] at (8) {$v_0$};

	\end{tikzpicture}
	\caption{$E(3241)$}\label{Fig_Elnitsky_3241}
\end{subfigure}%
\begin{subfigure}[t]{0.45\textwidth}
	\centering
	\begin{tikzpicture}[scale = 0.8]
	\foreach \x in {1,...,8}{ 
		\coordinate (\x) at (-\x*45-90:2cm) ;
	}
	
	\coordinate (A) at ($(1)+(67.5:1.531cm)$);
	\coordinate (B) at ($(A)+(-22.5:1.531cm)$);
	\coordinate (C) at ($(3) +(-22.5:1.531cm)$);

	\draw 
	(8)--(1)--(2)--(3)--(4)--(5)--(C)--(B)--cycle
	(1)--(A)--(B)
	(A)--(3)--(C);

	\node[below] at (8) {$v_0$};
					\node (n2) at ($(1)!0.5!(B)$) {$s_1$};
	
					\node (n1) at ($(2)!0.5!(A)$) {$s_2$};
					\node (n3) at ($(B)!0.5!(3)$) {$s_2$};
	
					\node (n4) at ($(5)!0.5!(3)$) {$s_3$};

\end{tikzpicture}
\caption{A tiling for $E(3241)$}\label{Fig_Elnitsky_tiling_3241}
\end{subfigure}
\caption{Elnitsky $8$-gon for the permutation $3241 \in S_4$ and a tiling}
\end{figure}

\lee{Elnitsky~\cite[Theorem~2.2]{Elnitsky97} proved that there is a bijective correspondence between the set of rhombic tilings of $E(w)$ and the commutation classes of reudced words of $w$. More precisely, for a given rhombic tiling $\mathcal{T}$, one can associate a commutation class in an inductive way as follows. We first label each tile $\gamma$ by a simple reflection $s_i$ if 
\[
i = \min_{v \in V(\gamma)} \{d(v,v_0)\} +1 
\]
where $V(\gamma)$ is the set of vertices in $\gamma$ and $d$ denotes the graph-theoretic distance between two vertices. Let $B_1$ be the border consisting of edges $b_1,\dots,b_{n+1}$ and $\mathcal{T}_1$ the set of tiles in $\mathcal{T}$ intersecting  $B_1$  in two edges. We cahgne the border into the new border $B_2$ by replacing the two edges of each $\gamma \in \mathcal{T}_1$ intersecting $B_1$ by the other two edges of $\gamma$. Then define $\mathcal{T}_2$ by the set of all tiles $\mathcal T \setminus \mathcal{T}_1$ intersecting $B_2$ with two edges. Continuing this process inductively until $B_k$ consists only of edges lying on the boundary of $E(w)$, we get a partition $\mathcal{T} = \bigsqcup_k \mathcal{T}_k$. Then we assign a reduced word $\mathbf i(\mathcal{T}) = (i_1,\dots,i_r)$ to the partition by reading the labeling of each tile in $\mathcal{T}_k$ in ascending order in $k$. Note that some partition $\mathcal{T}_k$ may have more than one tile, however, the associated commutation class is well-defined. For instance, the tiling $\mathcal{T}$ in Figure~\ref{Fig_Elnitsky_tiling_3241} corresponds to a reduced word $\mathbf i = \mathbf i (\mathcal T)=(2,1,2,3)$.  
}
	
\lee{
To define a Bott--Samelson variety, we associate a vector space to each vertex of a tiling $\mathcal T$. Starting with the vertex between the edges labeled $1$ and $w(1)$, label the vertices of $E(w)$ in clockwise order by 
\[
\C^0, \C^1,\dots,\C^n, G_{n-1},G_{n-2},\dots,G_1. 
\]
Moreover, associate vector space $V_x$ to a vertex $x$ in the tiling. The dimension of $V_x$ is $d(x,v_0)$. For a given tiling $\mathcal{T}$, define 
\[
Z_{\mathcal{T}} \colonequals 
\{(V_x)_{x \in V(\mathcal{T})} \mid V_y \subseteq V_z \text{ if }\dim(V_y)+1 = \dim(V_z) \text{ and }y,z \text{ are adjacent}\}.
\]
Here, $V(\mathcal{T})$ is the set of vertices in $\mathcal{T}$. It is proved in~\cite[Theorem~1.1]{EPTY18} that 
\[
Z_{\mathcal{T}} \cong Z_{\mathbf i(\mathcal{T})}. 
\]
For instance, the Bott--Samelson variety $Z_{(s_2,s_1,s_2,s_3)} = Z_{(2,1,2,3)}$ is isomorphic to $Z_{\mathcal T}$
\[
\begin{split}
Z_{\mathcal T} = \{(V_1,V_2,V_3,V_4) \mid & \dim(V_k) = i_k, \\
& \C^1 \subseteq V_1 \subseteq \C^3, \C^0 \subseteq V_2 \subseteq V_1, \\
& V_2 \subseteq V_3 \subseteq \C^3, V_3 \subseteq V_4 \subseteq \C^4 
\}.
\end{split}
\]
}

For any $w \in \Sn$, the generalized Bott--Samelson variety $Z_{(w)}$ is the Schubert variety $X_w$. Indeed, not every generalized Bott--Samelson varieties are smooth. If all Schubert varieties $X_{w_1},\dots,X_{w_r}$ corresponding to $w_1,\dots,w_r \in \Sn$ are smooth, then the generalized Bott--Samelson variety $Z_{(w_1,\dots,w_r)}$ is smooth (see~\cite[(6) in~\S13.4]{Jant}).

\begin{theorem}[{\cite{Daly, GreenLosonczy02,L-S1990,LMP_complexity_one,Tenner2012}}] \label{thm:smooth-one}
\label{thm:smooth} For a permutation~$w$ in $\Sn$, the following statements are equivalent:
\begin{enumerate}
\item[$(1^{\prime})$] $X_w$ is smooth and of complexity one.
\item[$(2^{\prime})$] $w$ contains the pattern $321$ exactly once and avoids the pattern $3412$.
\item[$(3^{\prime})$] There exists a reduced decomposition of $w$ containing $s_{i}s_{i+1}s_{i}$ as a factor and no other repetitions.
\item[$(4^{\prime})$] $X_w$ is isomorphic to a generalized Bott--Samelson variety $Z_{(w_{1},\dots,w_{r})}$ such that $r=\ell(w)-2$, $w_{k}=s_{j}s_{j+1}s_{j}$ for some $1\leq k\leq r$, $w_{i}=s_{j_{i}}$ for $i\neq k$, and $j_{1},\dots,j_{k-1},j_{k+1},\dots,j_{r},j,j+1$ are pairwise distinct.
\item[$(5^{\prime})$] The Bruhat interval $[e,w]$ is isomorphic to $S_3\times \mathfrak{B}_{\ell(w)-3}$ as posets. 
\item[$(6^{\prime})$] The polytope $\Q_{w}$ is combinatorially equivalent to the product of the hexagon and the cube of dimension $\ell(w)-3$. 
\end{enumerate}
\end{theorem}

\begin{theorem}[{\cite{Daly, GreenLosonczy02,L-S1990,LMP_complexity_one,Tenner2012}}]\label{thm:singular-one}
\label{thm:singular} For a permutation~$w$ in $\Sn$, the following statements are equivalent:
\begin{enumerate}
\item[$(1^{\prime\prime})$] $X_w$ is singular and of complexity one.
\item[$(2^{\prime\prime})$] $w$ contains the pattern $3412$ exactly once and avoids the pattern $321$.
\item[$(3^{\prime\prime})$] There exists a reduced decomposition of $w$ containing $s_{i+1}s_{i}s_{i+2}s_{i+1}$ as a factor and no other repetitions.
\item[$(4^{\prime\prime})$] $X_w$ is isomorphic to a generalized Bott--Samelson variety $Z_{(w_{1},\dots,w_{r})}$ such that $r=\ell(w)-3$, $w_{k}=s_{j+1}s_{j}s_{j+2}s_{j+1}$ for some $1\leq k \leq r$, $w_{i}=s_{j_{i}}$ for $i\neq k$, and $j_{1},\dots,j_{k-1},j_{k+1},\dots,j_{r},j,j+1,j+2$ are pairwise distinct.
\item[$(5^{\prime\prime})$] The Bruhat interval $[e,w]$ is isomorphic to $[e,3412]\times \mathfrak{B}_{\ell(w)-4}$ as posets. 
\item[$(6^{\prime\prime})$] The polytope $\Q_{w}$ is combinatorially equivalent to the product of $\Q_{3412}$ and the cube of dimension $\ell(w)-4$. 
\end{enumerate}
\end{theorem}


\section{Generic torus orbit closures in Richardson varieties}\label{sec:Richardson}

In this section, we consider generic torus orbit closures in Richardson varieties. The topology and geometry of generic torus orbit closures in Richardson varieties are related to the combinatorics of Bruhat interval polytopes. Studying the combinatorial properties of Bruhat interval polytopes, we show that every smooth toric Richardson variety is a Bott manifold. 
We also give a sufficient condition for a pair $(v,w)\in S_n\times S_n$ to give rise to a \emph{smooth} toric Richardson variety $X^v_w$. Motivated by this sufficient condition, we define a toric variety of Catalan type and then show that the number of isomorphism classes of smooth toric Richardson varieties of Catalan type is the Wedderburn--Etherington number.


\subsection{Richardson varieties and Bruhat interval polytopes}
For $v,w\in S_n$ with $v\leq w$ in the Bruhat order, the Richardson variety $X^v_w$ is defined to be the intersection of the Schubert variety $X_w$ and the opposite Schubert variety $X^v=w_0X_{w_0v}$. It is known that every Richardson variety
$X^v_w$ is a $T$-invariant irreducible subvariety of $\flag(n)$ with respect to the $T$-action on $\flag(n)$ and 
\[
\dim_\C X^v_w =\ell(w)-\ell(v).
\] 
See~\cite[Section~1.3]{Brion_Lecture} for more details. Moreover, a $T$-fixed point $uB$ is contained in $X^v_w$ if and only if $v\leq u\leq w$ in the Bruhat order. 
Accordingly,
the set of $T$-fixed points of $X^v_w$ is
identified with the Bruhat interval 
\[
[v,w] = \{u\in S_n\mid v\leq u\leq w\}.
\]

Recall that the moment map $\mu \colon G/B \to \R^n$ in \eqref{eq:moment-map} sends $uB$ to $\vertex{u}=(u^{-1}(1),\dots,u^{-1}(n))$. Hence we have 
\[
\mu(X^v_w)=\Conv\{(u^{-1}(1),\dots,u^{-1}(n))\mid u\in [v,w]\}.
\]
Motivated by this fact, we provide the following definition. 

\begin{definition}[{\cite{KodamaWilliams}}]
For elements $v$ and $w$ in ${S}_{n}$ with $v\leq w$, we define a polytope $\Q^v_w$ by 
\[
\Q^v_w\colonequals\Conv\{(u^{-1}(1),\dots,u^{-1}(n))\mid u\in [v,w]\}.
\]
This polytope is called a \emph{Bruhat interval polytope}. Note that $\Q^e_w$ is the polytope $\Q_w$ introduced in Section~\ref{sec:Schubert}.
\end{definition}

By the definition of $\Q^v_w$, we have  
\[
\mu\colon X_w^v\twoheadrightarrow \Q^{v}_{w}\subset \R^n.
\]

\begin{remark}
Bruhat interval polytopes were introduced by Kodama--Williams \cite{KodamaWilliams} and their combinatorial properties are studied in \cite{TsukermanWilliams} and \cite{LMP2021}.  For $v,w\in S_n$ with $v\le w$, the notation   
\[
\Q_{v,w}\colonequals \Conv\{(u(1),\dots,u(n))\mid u\in [v,w]\}
\]
is used in \cite{KodamaWilliams, TsukermanWilliams, LMP2021}.  Therefore we have $\Q^v_w=\Q_{v^{-1},w^{-1}}$.  
\end{remark}

Bruhat interval polytopes $\Q^v_w$ and $\Q^{v^{-1}}_{w^{-1}}$ are not necessarily combinatorially equivalent even though the intervals $[v,w]$ and $[v^{-1},w^{-1}]$ are isomorphic as posets. For example, one can check that the polytopes $\Q^{12345}_{35412}$ and $\Q^{12345}_{45132}$ have different numbers of edges using a computer program like SageMath. However, they have the same dimension.

\begin{proposition}[{\cite[Proposition 3.4]{LMP2021}}]\label{lem:dim}
Bruhat interval polytopes $\Q^v_w$ and $\Q^{v^{-1}}_{w^{-1}}$ have the same dimension. 
\end{proposition}


\subsection{Generic torus orbit closures in Richardson varieties}
A torus orbit $T\cdot x$ for $x\in X^v_w$ is said to be \emph{generic} if $(\overline{T\cdot x})^T=(X^v_w)^T$. Every Richardson variety admits a generic torus orbit. Its proof is given in~\cite[Proposition~3.8]{LM2020} when $v = e$ and a similar argument works for any Richardson variety $X^v_w$. We will denote by $Y^v_w$ the closure of a generic ${T}$-orbit in $X^v_w$. Then $Y^v_w$ is the projective toric variety defined by the polytope $\Q^{v}_{w}$ since $\mu(Y^v_w) = \mu(X^v_w)=\Q^{v}_{w}$.

Note that the family of Bruhat interval polytopes forms a subfamily of the $\Phi$-polytopes. Hence every edge of a Bruhat interval polytope is parallel to a vector of the form $\mathbf{e}_i-\mathbf{e}_j$. This implies that $Y^v_w$ is smooth at a fixed point $uB$ if and only if $\Q_{w}^{v}$ is simple at the vertex $\vertex{u}$ (see~\cite[Section~8]{LM2020}).

While $Y_w=Y^e_w$ is smooth at $eB$, not every toric variety $Y^v_w$ is smooth at $vB$. For instance, the polytope $\Q^{1324}_{3412}$ is not simple at $\vertex{1324}$, see Figure~\ref{fig_BIP_1324-3412}. 
A natural generalization of Theorem~\ref{conj:schubert} is the following. 

\begin{conjecture}\label{conjecture_smoothness_Ywv}
The generic torus orbit closure $Y_w^v$ in the Richardson variety $X_w^v$ is smooth if it is smooth at the fixed points $vB$ and $wB$, in other words, $\Q_{w}^{v}$ is simple if it is simple at the vertices $\vertex{v}$ and $\vertex{w}$. 
\end{conjecture}

\begin{figure}[H]
\begin{tikzpicture}[scale=6]
\tikzset{every node/.style={draw=blue!50,fill=blue!20, circle, thick, inner sep=1pt,font=\footnotesize}}
\tikzset{red node/.style = {fill=red!20!white, draw=red!75!white}}
\tikzset{red line/.style = {line width=0.3ex, red, nearly opaque}}
\coordinate (3142) at (1/3, 1/2, 1/6);
\coordinate (4231) at (2/3, 1/2, 1/6);
\coordinate (4312) at (5/6, 2/3, 1/2);
\coordinate (4321) at (5/6, 1/2, 1/3);
\coordinate (3421) at (5/6, 1/3, 1/2);
\coordinate (4213) at (2/3, 5/6, 1/2);
\coordinate (1324) at (1/3, 1/2, 5/6);
\coordinate (2413) at (2/3, 1/2, 5/6);
\coordinate (3412) at (5/6, 1/2, 2/3);
\coordinate (2314) at (1/2, 2/3, 5/6);
\coordinate (4123) at (1/2, 5/6, 1/3);
\coordinate (4132) at (1/2, 2/3, 1/6);
\coordinate (3214) at (1/2, 5/6, 2/3);
\coordinate (3124) at (1/3, 5/6, 1/2);
\coordinate (2431) at (2/3, 1/6, 1/2);
\coordinate (1432) at (1/2, 1/6, 2/3);
\coordinate (1423) at (1/2, 1/3, 5/6);
\coordinate (1342) at (1/3, 1/6, 1/2);
\coordinate (2341) at (1/2, 1/6, 1/3);
\coordinate (3241) at (1/2, 1/3, 1/6);
\coordinate (1243) at (1/6, 1/3, 1/2);
\coordinate (2143) at (1/6, 1/2, 1/3);
\coordinate (1234) at (1/6, 1/2, 2/3);
\coordinate (2134) at (1/6, 2/3, 1/2);
\draw[thick, draw=blue!70] (4213)--(4312)--(3412)--(2413)--(2314)--(3214)--cycle;
\draw[thick, draw=blue!70] (4312)--(4321)--(3421)--(3412);
\draw[thick, draw=blue!70] (3421)--(2431)--(1432)--(1423)--(2413);
\draw[thick, draw=blue!70] (1423)--(1324)--(2314);
\draw[thick, draw=blue!70] (1432)--(1342)--(1243)--(1234)--(1324);
\draw[thick, draw=blue!70] (1234)--(2134)--(3124)--(3214);
\draw[thick, draw=blue!70] (3124)--(4123)--(4213);
\draw[thick, draw=blue!70, dashed] (2134)--(2143)--(3142)--(4132)--(4123);
\draw[thick, draw=blue!70, dashed] (2143)--(1243);
\draw[thick, draw=blue!70, dashed] (3142)--(3241)--(2341)--(1342);
\draw[thick, draw=blue!70, dashed] (2341)--(2431);
\draw[thick, draw=blue!70, dashed] (3241)--(4231)--(4132);
\draw[thick, draw=blue!70, dashed] (4231)--(4321);

\node[label = {[label distance = 0cm]right:$\vertex{2341}$}] at (2341) {};

\draw[red line] (3214)--(2314)--(2413)--(3412)--cycle;
\draw[red line] (2314)--(1324)--(1423)--(2413);
\draw[red line] (1324)--(1342)--(1432)
(1432)--(1423);
\draw[red line] (1432)--(3412);
\draw[red line] (1324)--(3124);
\draw[red line](3124)--(3214);
\draw[red line, dashed] (3124)--(3142)--(1342);
\draw[red line, dashed] (3142)--(3412);

\node [label = {[label distance = 0cm]left:$\vertex{1234}$}] at (1234) {};
\node[label = {[label distance = 0cm]left:$\vertex{1243}$}] at (1243) {};
\node[label = {[label distance = 0cm]right:$\vertex{1324}$}, red node] at (1324) {};
\node[label = {[label distance = 0cm]left:$\vertex{1342}$}, red node] at (1342) {};
\node [label = {[label distance = 0cm]above:$\vertex{1423}$}, red node] at (1423) {};
\node[label = {[label distance = -0.2cm]below:$\vertex{1432}$}, red node] at (1432) {};
\node [label = {[label distance = 0cm]left:$\vertex{2134}$}] at (2134) {};
\node[label = {[label distance = -0.1cm]below right:$\vertex{2143}$}] at (2143) {};
\node[label = {[label distance = 0cm]below:$\vertex{2314}$}, red node] at (2314) {};

\node[label = {[label distance = 0cm]left:$\vertex{2413}$}, red node] at (2413) {};
\node[label = {[label distance = -0.2cm]below:$\vertex{2431}$}] at (2431) {};
\node[label = {[label distance = -0.2cm]above:$\vertex{3124}$}, red node] at (3124) {};
\node[label = {[label distance = -0.2cm]above:$\vertex{3142}$}, red node] at (3142) {};
\node[label = {[label distance = -0.2cm]above:$\vertex{3214}$}, red node] at (3214) {};
\node [label = {[label distance = -0.1cm]above:$\vertex{3241}$}] at (3241) {};
\node[label = {[label distance = 0cm]below left:$\vertex{3412}$}, red node] at (3412) {};
\node[label = {[label distance = 0cm]right:$\vertex{3421}$}] at (3421) {};
\node[label = {[label distance = -0.2cm]above:$\vertex{4123}$}] at (4123) {};
\node [label = {[label distance = 0cm]below:$\vertex{4132}$}] at (4132) {};
\node[label = {[label distance = 0cm]right:$\vertex{4213}$}] at (4213) {};
\node[label = {[label distance = 0cm]left:$\vertex{4231}$}] at (4231) {};
\node[label = {[label distance = 0cm]right:$\vertex{4312}$}] at (4312) {};
\node [label = {[label distance = 0cm]right:$\vertex{4321}$}] at (4321) {};

\end{tikzpicture}
\caption{Bruhat interval polytope $\Q^{1324}_{3412}$} \label{fig_BIP_1324-3412}
\end{figure}
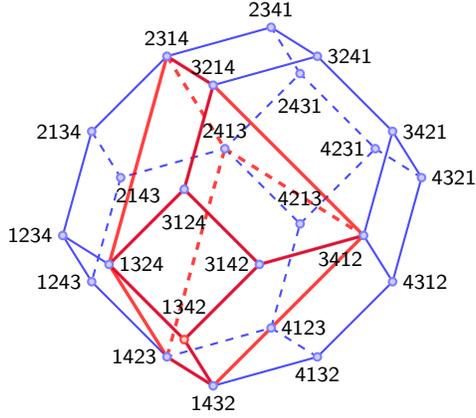


\subsection{Toric Bruhat interval polytopes} 

Since $\Q^{v}_{w}$ is the moment map image of a toric variety~$Y^{v}_{w}$, we get
\begin{equation*}\label{eq_dim_Qvw_and_length}
\dim \Q^v_w=\dim_\C Y^{v}_{w}\leq \dim_\C X^{v}_{w}=\ell(w)-\ell(v).
\end{equation*}
Motivated by this observation, we call a Bruhat interval polytope $\Q^v_w$ \emph{toric} if $\dim \Q^v_w=\ell(w)-\ell(v)$. 
Hence a Richardson variety $X^v_w$ is a toric variety {with respect to the ${T}$-action}, that is,  $X^v_w=Y^v_w$, if and only if the Bruhat interval polytope $\Q^{v}_{w}$ is toric.

Tsukerman and William~\cite{TsukermanWilliams} showed that every face of the Bruhat interval polytope $\Q^v_w$ is realizable as a subinterval of $[v^{-1},w^{-1}]$, but the converse is not true. For instance, the subinterval~$[123,312]$ of $[123,321]=[123^{-1},321^{-1}]$ does not form a face of the polytope $\Q^{123}_{321}$. See Figure~\ref{fig:vertex-labelling}. However, the converse holds when $\Q^v_w$ is toric. More strongly, we have the following.

\begin{theorem} [{\cite[Theorem~5.1]{LMP2021} \lee{ and \cite[Theorem~1.1]{can2023toric}}}] \label{prop:3-2}
For a Bruhat interval polytope $\Q^v_w$,
the following statements are equivalent:
\begin{enumerate}
\item $\Q^v_w$ is toric {\rm(}i.e.,  $\dim \Q^v_w=\ell(w)-\ell(v)${\rm)}. \label{toric_Richardson_1}
\item $\Q^x_y$ is a face of $\Q^v_w$ for any $[x^{-1},y^{-1}]\subset [v^{-1},w^{-1}]$. \label{toric_Richardson_2}
\item \lee{The Bruhat interval $[v,w]$ does not contain any interval that is isomorphic to symmetric group $S_3$ with respect to Bruhat order and $\ell(w)- \ell(v) \leq n-1$.} \label{toric_Richardson_3}
\item \lee{The interval $[v,w]$ is a lattice and $\ell(w)- \ell(v) \leq n-1$. } \label{toric_Richardson_4}
\end{enumerate}
\end{theorem}

\lee{
Here, a poset $P$ is a \defi{lattice} if for all $x,y \in P$, the subposet $\{ z \in P \mid z \leq x, z \leq y\}$ has a top} \park{element}\lee{, and the subposet $\{z \in P \mid z \geq x, z \geq y\}$ has a bottom element.}
	
\lee{
The above statement is obtained by combining the results~\cite[Theorem~5.1]{LMP2021} and \cite{can2023toric}. Indeed, the} \park{equivalence} \lee{between \eqref{toric_Richardson_1} and \eqref{toric_Richardson_2} is proved in~\cite[Theorem~5.1]{LMP2021}; the equivalences between \eqref{toric_Richardson_1}, \eqref{toric_Richardson_3}, and \eqref{toric_Richardson_4} are given in~\cite[Theorem~1.1]{can2023toric}. 
}
Hence, if $\Q^v_w$ is toric, then its combinatorial type is determined by the poset structure of~$[v^{-1},w^{-1}]$, so $\Q^v_w$ and $\Q^{v^{-1}}_{w^{-1}}$ are combinatorially equivalent.
Here, the assumption of ``toric'' cannot be removed. Indeed, the interval $[1324,4231]$ is a Boolean algebra of rank~4, but the corresponding Bruhat interval polytope is of dimension~$3$. See Figure~\ref{fig:1324-4231}. On the other hand, for every positive integer~$m$ there is an example of an interval $[v,w]$ of rank $m$ such that $[v,w]$ is a Boolean algebra and $\Q^v_w$ is toric. For instance, when $w=s_ms_{m-1}\cdots s_2 s_1$ and $v=e$, the Bruhat interval polytope $\Q^v_w$ is combinatorially equivalent to a cube of dimension~$m$ (simply, $m$-cube).

\begin{figure}[H]
\begin{subfigure}[b]{0.49\textwidth}
\begin{tikzpicture}
\tikzset{every node/.style={font=\footnotesize}}
\tikzset{red node/.style = {fill=red!20!white, draw=red!75!white}}
\tikzset{red line/.style = {line width=1ex, red,nearly transparent}}
\matrix [matrix of math nodes,column sep={0.52cm,between origins},
row sep={0.8cm,between origins},
nodes={circle, draw=blue!50,fill=blue!20, thick, inner sep = 0pt , minimum size=1.2mm}]
{
& & & & & \node[label = {above:{4321}}] (4321) {} ; & & & & & \\
& & &
\node[label = {above left:4312}] (4312) {} ; & &
\node[label = {above left:4231}, red node] (4231) {} ; & &
\node[label = {above right:3421}] (3421) {} ; & & & \\
& \node[label = {above left:4132}, red node] (4132) {} ; & &
\node[label = {left:4213}, red node] (4213) {} ; & &
\node[label = {above:3412}] (3412) {} ; & &
\node[label = {[label distance = 0.1cm]0:2431}, red node] (2431) {} ; & &
\node[label = {above right:3241}, red node] (3241) {} ; & \\
\node[label = {left:1432}, red node] (1432) {} ; & &
\node[label = {left:4123}, red node] (4123) {} ; & &
\node[label = {[label distance = 0.01cm]180:2413}, red node] (2413) {} ; & &
\node[label = {[label distance = 0.01cm]0:3142}, red node] (3142) {} ; & &
\node[label = {right:2341}, red node] (2341) {} ; & &
\node[label = {right:3214}, red node] (3214) {} ; \\
& \node[label = {below left:1423}, red node] (1423) {} ; & &
\node[label = {[label distance = 0.1cm]182:1342}, red node] (1342) {} ; & &
\node[label = {below:2143}] (2143) {} ; & &
\node[label = {right:3124}, red node] (3124) {} ; & &
\node[label = {below right:2314}, red node] (2314) {} ; & \\
& & & \node[label = {below left:1243}] (1243) {} ; & &
\node[label = {[label distance = 0.01cm]190:1324}, red node] (1324) {} ; & &
\node[label = {below right:2134}] (2134) {} ; & & & \\
& & & & & \node[label = {below:1234}] (1234) {} ; & & & & & \\
};
\draw (4321)--(4312)--(4132)--(1432)--(1423)--(1243)--(1234)--(2134)--(2314)--(2341)--(3241)--(3421)--(4321);
\draw (4321)--(4231)--(4132);
\draw (4231)--(3241);
\draw (4231)--(2431);
\draw (4231)--(4213);
\draw (4312)--(4213)--(2413)--(2143)--(3142)--(3241);
\draw (4312)--(3412)--(2413)--(1423)--(1324)--(1234);
\draw (3421)--(3412)--(3214)--(3124)--(1324);
\draw (3421)--(2431)--(2341)--(2143)--(2134);
\draw (4132)--(4123)--(1423);
\draw (4132)--(3142)--(3124)--(2134);
\draw (4213)--(4123)--(2143)--(1243);
\draw (4213)--(3214);
\draw (3412)--(1432)--(1342)--(1243);
\draw (2431)--(1432);
\draw (2431)--(2413)--(2314);
\draw (3142)--(1342)--(1324);
\draw (4123)--(3124);
\draw (2341)--(1342);
\draw (2314)--(1324);
\draw (3412)--(3142);
\draw (3241)--(3214)--(2314);
\draw[red line] (1324)--(1423)--(1432)--(4132)--(4231);
\draw[red line] (1324)--(2314)--(3214)--(3241)--(4231);
\draw[red line] (1324)--(1342)--(1432);
\draw[red line] (1324)--(3124)--(3214);
\draw[red line] (2314)--(2341)--(2431)--(4231);
\draw[red line] (1423)--(4123)--(4213)--(4231);
\draw[red line](4123)--(4132);
\draw[red line] (3214)--(4213);
\draw[red line] (2341)--(3241);
\draw[red line](1432)--(2431);
\draw[red line] (1342)--(2341);
\draw[red line] (3124)--(4123);
\draw[red line] (1423)--(2413)--(4213);
\draw[red line] (3124)--(3142)--(4132);
\draw[red line] (3142)--(3241);
\draw[red line] (1342)--(3142);
\draw[red line] (2314)--(2413);
\draw[red line] (2413)--(2431);
\end{tikzpicture}
\caption{Bruhat interval $[1324, 4231]$}
\end{subfigure}~
\begin{subfigure}[b]{0.49\textwidth}
\begin{tikzpicture}[scale=6]
\tikzset{every node/.style={draw=blue!50,fill=blue!20, circle, thick, inner sep=1pt,font=\footnotesize}}
\tikzset{red node/.style = {fill=red!20!white, draw=red!75!white}}
\tikzset{red line/.style = {line width=0.3ex, red, nearly opaque}}
\coordinate (3142) at (1/3, 1/2, 1/6);
\coordinate (4231) at (2/3, 1/2, 1/6);
\coordinate (4312) at (5/6, 2/3, 1/2);
\coordinate (4321) at (5/6, 1/2, 1/3);
\coordinate (3421) at (5/6, 1/3, 1/2);
\coordinate (4213) at (2/3, 5/6, 1/2);
\coordinate (1324) at (1/3, 1/2, 5/6);
\coordinate (2413) at (2/3, 1/2, 5/6);
\coordinate (3412) at (5/6, 1/2, 2/3);
\coordinate (2314) at (1/2, 2/3, 5/6);
\coordinate (4123) at (1/2, 5/6, 1/3);
\coordinate (4132) at (1/2, 2/3, 1/6);
\coordinate (3214) at (1/2, 5/6, 2/3);
\coordinate (3124) at (1/3, 5/6, 1/2);
\coordinate (2431) at (2/3, 1/6, 1/2);
\coordinate (1432) at (1/2, 1/6, 2/3);
\coordinate (1423) at (1/2, 1/3, 5/6);
\coordinate (1342) at (1/3, 1/6, 1/2);
\coordinate (2341) at (1/2, 1/6, 1/3);
\coordinate (3241) at (1/2, 1/3, 1/6);
\coordinate (1243) at (1/6, 1/3, 1/2);
\coordinate (2143) at (1/6, 1/2, 1/3);
\coordinate (1234) at (1/6, 1/2, 2/3);
\coordinate (2134) at (1/6, 2/3, 1/2);
\draw[red line] (2431)--(4231)--(4213);
\draw[thick, draw=blue!70] (4213)--(4312)--(3412)--(2413)--(2314)--(3214)--cycle;
\draw[thick, draw=blue!70] (4312)--(4321)--(3421)--(3412);
\draw[thick, draw=blue!70] (3421)--(2431)--(1432)--(1423)--(2413);
\draw[thick, draw=blue!70] (1423)--(1324)--(2314);
\draw[thick, draw=blue!70] (1432)--(1342)--(1243)--(1234)--(1324);
\draw[thick, draw=blue!70] (1234)--(2134)--(3124)--(3214);
\draw[thick, draw=blue!70] (3124)--(4123)--(4213);
\draw[thick, draw=blue!70, dashed] (2134)--(2143)--(3142)--(4132)--(4123);
\draw[thick, draw=blue!70, dashed] (2143)--(1243);
\draw[thick, draw=blue!70, dashed] (3142)--(3241)--(2341)--(1342);
\draw[thick, draw=blue!70, dashed] (2341)--(2431);
\draw[thick, draw=blue!70, dashed] (3241)--(4231)--(4132);
\draw[thick, draw=blue!70, dashed] (4231)--(4321);
\draw[red line] (2314)--(2413)--(4213)--(3214)--cycle;
\draw[red line] (3214)--(3124)--(4123)--(4213);
\draw[red line] (3124)--(1324)--(2314);
\draw[red line] (1324)--(1423)--(2413);
\draw[red line] (1324)--(1342)--(1432)--(1423);
\draw[red line] (2413)--(2431)--(1432);
\draw[red line, dashed] (3124)--(3142)--(4132)--(4123);
\draw[red line, dashed] (4132)--(4231)--(3241)--(3142);
\draw[red line, dashed] (3241)--(2341)--(1342)--(3142);
\draw[red line, dashed] (2341)--(2431);

\node [label = {[label distance = 0cm]left:$\vertex{1234}$}] at (1234) {};
\node[label = {[label distance = 0cm]left:$\vertex{1243}$}] at (1243) {};
\node[label = {[label distance = 0cm]right:$\vertex{1324}$}, red node] at (1324) {};
\node[label = {[label distance = 0cm]left:$\vertex{1342}$}, red node] at (1342) {};
\node [label = {[label distance = 0cm]above:$\vertex{1423}$}, red node] at (1423) {};
\node[label = {[label distance = -0.2cm]below:$\vertex{1432}$}, red node] at (1432) {};
\node [label = {[label distance = 0cm]left:$\vertex{2134}$}] at (2134) {};
\node[label = {[label distance = -0.1cm]below right:$\vertex{2143}$}] at (2143) {};
\node[label = {[label distance = 0cm]below:$\vertex{2314}$}, red node] at (2314) {};
\node[label = {[label distance = 0cm]right:$\vertex{2341}$}, red node] at (2341) {};
\node[label = {[label distance = 0cm]left:$\vertex{2413}$}, red node] at (2413) {};
\node[label = {[label distance = -0.2cm]below:$\vertex{2431}$}, red node] at (2431) {};
\node[label = {[label distance = -0.2cm]above:$\vertex{3124}$}, red node] at (3124) {};
\node[label = {[label distance = -0.2cm]above:$\vertex{3142}$}, red node] at (3142) {};
\node[label = {[label distance = -0.2cm]above:$\vertex{3214}$}, red node] at (3214) {};
\node [label = {[label distance = -0.1cm]above:$\vertex{3241}$}, red node] at (3241) {};
\node[label = {[label distance = 0cm]below left:$\vertex{3412}$}, red node] at (3412) {};
\node[label = {[label distance = 0cm]right:$\vertex{3421}$}] at (3421) {};
\node[label = {[label distance = -0.2cm]above:$\vertex{4123}$},red node] at (4123) {};
\node [label = {[label distance = 0cm]below:$\vertex{4132}$}, red node] at (4132) {};
\node[label = {[label distance = 0cm]right:$\vertex{4213}$}, red node] at (4213) {};
\node[label = {[label distance = 0cm]left:$\vertex{4231}$}, red node] at (4231) {};
\node[label = {[label distance = 0cm]right:$\vertex{4312}$}] at (4312) {};
\node [label = {[label distance = 0cm]right:$\vertex{4321}$}] at (4321) {};
\end{tikzpicture}
\caption{Bruhat interval polytope $\Q^{1324}_{4231}$}\label{fig:1324-4231b}
\end{subfigure}
\caption{The interval $[1324,4231]$ is a Boolean algebra, but $\Q^{1324}_{4231}$ is not a cube.}\label{fig:1324-4231}
\end{figure}
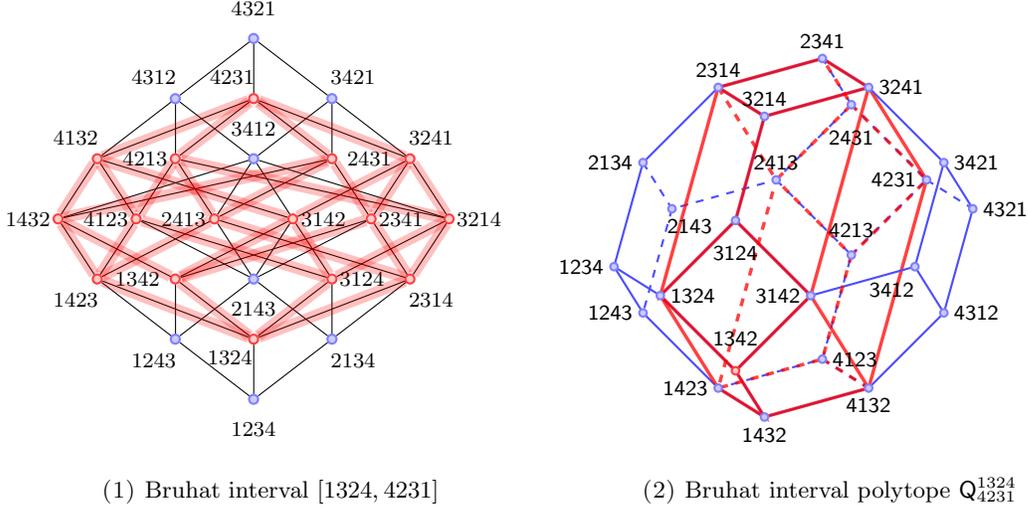

In the following, for simplicity, when a polytope~$\Q$ is combinatorially equivalent to a cube (or a $d$-cube), we say that~$\Q$ is a \emph{cube} (or a \emph{$d$-cube}). We also say that an interval $[v,w]$ is \emph{Boolean} if it is a Boolean algebra.

Using the following two facts
\begin{enumerate}
\item if $\Q$ is a simple polytope of dimension $\ge 2$ and every $2$-face of $\Q$ is a $2$-cube, then $\Q$ is a cube (\cite[Problems and Exercises 0.1 in p.23]{zieg98} and \cite[Appendix]{yu-ma16}), and
\item every $2$-interval is a diamond (\cite[Lemma 2.7.3]{BB05Combinatorics}),
\end{enumerate}
we can prove:

\begin{proposition}[{\cite[Proposition 5.6]{LMP2021}}] \label{prop:3-5}
Suppose that $\Q^v_w$ is toric. Then $\Q^v_w$ is a cube if and only if it is simple. {\rm(}This is equivalent to saying that a toric Richardson variety is a Bott manifold if and only if it is smooth.{\rm)}
\end{proposition}

The following gives a characterization of when $\Q^v_w$ is a cube.

\begin{theorem}[{\cite[Theorem 5.7]{LMP2021} and \lee{\cite[Theorem~1.1]{can2023toric}}}] \label{theo:3-6}
	The following statements are equivalent. 
	\begin{enumerate}
		\item A Bruhat interval polytope $\Q^v_w$ is a cube. \label{smooth_toric_Richardson_1}
		\item $\Q^v_w$ is toric and $[v,w]$ is Boolean. \label{smooth_toric_Richardson_2}
		\item A Richardson variety $X^v_w$ is a Bott manifold. \label{smooth_toric_Richardson_3}
		\item $\ell(w)- \ell(v) \leq n-1$ and $[v,w]$ is Boolean. \label{smooth_toric_Richardson_4}
	\end{enumerate}
\end{theorem}

\lee{The equivalences between~\eqref{smooth_toric_Richardson_1}, \eqref{smooth_toric_Richardson_2}, and \eqref{smooth_toric_Richardson_3} are given in~\cite[Theorem~5.7]{LMP2021}, and the} \park{equivalence} \lee{between \eqref{smooth_toric_Richardson_2} and \eqref{smooth_toric_Richardson_4} is given by~\cite[Theorem~1.1]{can2023toric}. }
In the above theorem, we cannot drop either \emph{toric} or \emph{Boolean}.

\begin{example}
\label{exam:3-1}
\begin{enumerate}
\item For $n=4$, $v=1324$, and $w=4231$ in $S_4$, the interval $[v,w]$ is Boolean of length $4$ but $\dim \Q^v_w=3$, so $\Q^v_w$ is not toric. \lee{We notice that $\ell(w)-\ell(v) = 5-1=4 \not\leq 3 = n-1$.} Since the vertices $\vertex{v}$ and $\vertex{w}$ have degree $4$, $\Q^v_w$ is not a cube.
\item When $v=1324$ and $w=3412$, we have $\ell(w)-\ell(v)=4-1=3$ and $\dim \Q^v_w=3$. Hence $\Q^v_w$ is toric. However, $[v,w]$ is a $4$-crown (see \cite[p.52]{BB05Combinatorics}) and not Boolean. The vertices $\vertex{v}$ and $\vertex{w}$ have degree $4$ and the others are simple vertices, so $\Q^v_w$ is not a cube. See Figure~\ref{fig_BIP_1324-3412}.
\end{enumerate}
\end{example}

It is shown in~\cite[Theorem 3.5.2]{re02} that the largest rank of Boolean Bruhat intervals in ${S}_{n+1}$ is at least $n+\lfloor\frac{n-1}{2}\rfloor$ by finding a sufficient condition on $v$ and $w$ for $[v,w]$ to be Boolean. This implies that there are infinitely many Boolean Bruhat intervals which are not toric like the Boolean interval $[1324,4231]$. In fact,
for any non-negative integer $k$, there is a Boolean interval $[v,w]$ such that $\dim\Q^v_w=\ell(w)-\ell(v)-k$. See {\cite[Proposition 6.4]{LMP2021}} for details.

\begin{remark}
\lee{Can and Saha considered toric} \park{Richardson} \lee{varieties for any connected, simple, simply-connected algebraic group over an algebraically closed field in~\cite{can2023toric}. They provide characterizations of toric Richardson varieties in terms of the combinatorics of posets.}
\end{remark}

\subsection{Conditions on \texorpdfstring{$v$}{v} and \texorpdfstring{$w$}{w} for \texorpdfstring{$\Q^v_w$}{Qvw} to be a cube}\label{sec:Conditions on v and w}
Now we find a sufficient condition for a Bruhat interval polytope $\Q^v_w$ to be a cube.

It is shown in~\cite[{\S 5 and \S 6}]{HHMP19} that $\Q^v_w$ is toric (in fact, a cube) if $v^{-1}=[a_1,\dots,a_{n-1},n]$ and $w^{-1}=[n,a_1,\dots,a_{n-1}]$ or $v^{-1}=[1,b_2,\dots,b_n]$ and $w^{-1}=[b_2,\dots,b_n,1]$. In these cases,
\begin{equation*}\label{eq:minimal}
\begin{split}
&w=s_{1}s_{2}\cdots s_{n-1}v \quad\text{and}\quad \ell(w)-\ell(v)=n-1,\ \text{or}\\
&w=s_{n-1}s_{n-2}\cdots s_1v \quad\text{and}\quad \ell(w)-\ell(v)=n-1.
\end{split}
\end{equation*}
These examples and Theorem~\ref{thm:toric}(3) motivate us to study the following case:
\begin{equation} \label{eq:6-1}
\text{$w=s_{j_1}s_{j_2}\cdots s_{j_m}v$ or $vs_{j_1}s_{j_2}\cdots s_{j_m}$ where $\ell(w)-\ell(v)=m$ and $j_1,\dots,j_m$ are distinct.}
\end{equation}

\begin{proposition} [{\cite[Proposition 7.1]{LMP2021}}]\label{prop:6-1}
Suppose that $w=s_{j_1}s_{j_2}\cdots s_{j_m}v$ or $w=vs_{j_1}s_{j_2}\cdots s_{j_m}$ with $\ell(w)-\ell(v)=m$. Then $j_1,\dots,j_m$ are distinct if and only if $\Q^v_w$ is toric.
\end{proposition}

Note that not every Bruhat interval polytope in the above proposition is a cube. For instance, if $v = 1324$ and $w=3412$, then $v= s_2$ and $w = s_2s_3s_1s_2$, so $w=s_2s_3s_1 v$. However, the polytope~$\Q^v_w$ is not a cube as in Figure~\ref{fig_BIP_1324-3412}. On the other hand, for $v=1243$ and $w=3412$, since $v=s_3$ and $w=s_2s_3s_1s_2=s_2s_1s_3s_2$, neither $wv^{-1}$ nor $v^{-1}w$ is a product of distinct adjacent transpositions. However, the Bruhat interval polytope $\Q^v_w$ is a $3$-cube. These two examples show that it seems difficult to characterize $v$ and $w$ for which $\Q^v_w$ is a cube.

Let us find a sufficient condition on $v$ and $w$ for $\Q^v_w$ to be a cube. For that, we prepare some notations.
For $p,q\in [n-1]$, we set
\begin{equation}\label{eq_s_pq}
s(p,q)=\begin{cases} s_ps_{p+1}\cdots s_q \quad&\text{if $p\le q$},\\
s_ps_{p-1}\cdots s_q \quad &\text{if $p\ge q$}.
\end{cases}
\end{equation}
For each $s(p,q)$, we also set
\[
\bar{p}=\min\{p,q\},\quad \bar{q}=\max\{p,q\}.
\]
We note that if $j_1,\dots,j_m\in [n-1]$ are distinct, then we have a \emph{minimal} expression
\begin{equation} \label{eq:6-2}
s_{j_1}s_{j_2}\cdots s_{j_m}=s(p_1,q_1)s(p_2,q_2)\cdots s(p_r,q_r)
\end{equation}
where the intervals $[\bar{p}_1,\bar{q}_1],\dots, [\bar{p}_r,\bar{q}_r]$ are disjoint and $r$ is the minimum among such expressions.

\begin{example} \label{exam:6-1}
Here are examples of minimal expressions.
\begin{enumerate}
\item $s_1s_2\cdots s_{n-1}=s(1,n-1)$,\quad $s_{n-1}s_{n-2}\cdots s_1=s(n-1,1)$.

\item $s_1s_3s_8s_2s_4s_7s_6=s_3s_4s_1s_2s_8s_7s_6=s(3,4)s(1,2)s(8,6)$.

\item $s_2s_8s_4s_7s_1s_6=s_2s_1s_4s_8s_7s_6=s(2,1)s(4,4)s(8,6)$.
\end{enumerate}
\end{example}

We say that the product $s_{j_1}s_{j_2}\cdots s_{j_m}$ in \eqref{eq:6-2} is \emph{proper} if no two intervals among $[\bar{p}_1,\bar{q}_1],\dots, [\bar{p}_r,\bar{q}_r]$ are adjacent, in other words, the cycles defined by $s(p_1,q_1)$,$\ldots$, $s(p_r,q_r)$ are disjoint. In Example~\ref{exam:6-1}, (1) and (3) are proper, but (2) is not because the intervals $[3,4]$ and $[1,2]$ are adjacent.

The following provides a sufficient condition for $v$ and $w$ such that the Richardson variety $X^v_w$ is smooth and toric.

\begin{proposition} [{\cite[Proposition 7.3]{LMP2021}}] \label{prop:6-2}
Suppose that $s_{j_1}s_{j_2}\cdots s_{j_m}$ is a proper minimal expression.
If $w=s_{j_1}s_{j_2}\cdots s_{j_m}v$ or $w=vs_{j_1}s_{j_2}\cdots s_{j_m}$ with $\ell(w)-\ell(v)=m$, then the Bruhat interval polytope $\Q^v_w$ is a cube.
\end{proposition}

In fact, the converse of Proposition~\ref{prop:6-2} is also true. That is, 
a Bruhat interval polytope $\Q^v_w$ is a cube for any $v$ and $w$ in~\eqref{eq:6-1} if and only if the product $s_{j_1} s_{j_2} \cdots s_{j_m}$ is proper. See~{\cite[Corollary~7.7]{LMP2021}}.


\subsection{Toric varieties of Catalan type}\label{sec_Catalan_numbers}
Many smooth toric Richardson varieties arise from polygon triangulations. This subsection is a preparation for that. Namely we explain how to associate a compact smooth toric variety with a polygon triangulation. 

Let $\polygon{n+2}$ denote a convex polygon in the plane with $n+2$ vertices (or \textit{convex $(n+2)$-gon} for simplicity). We label the vertices from $0$ to $n+1$ in counterclockwise order. A \textit{triangulation} of~$\polygon{n+2}$ is a decomposition of $\polygon{n+2}$ into a set of $n$ triangles by adding $n-1$ diagonals of $\polygon{n+2}$ which do not intersect in their interiors. Then the \textit{Catalan number} 
\[
C_n=\frac{1}{n+1}\binom{2n}{n}
\] 
is the number of triangulations of $\polygon{n+2}$. Figure~\ref{fig_triangulation_5gon} shows the triangulations of $\polygon{5}$. 
Here, we consider all five triangulations to be different because we are triangulating an $(n+2)$-gon with \emph{labelled} vertices.

\begin{figure}[H]
\begin{tikzpicture}[x=10mm, y=5mm, label distance = 1mm]
\foreach \x in {2,...,5,6}{
\coordinate (a\x) at (-18-\x*72:1cm) ; 
} 

\draw[thick] (a2)--(a3)--(a4)--(a5)--(a6)--(a2);
\draw[thick] (a2)--(a4)
(a4)--(a6);

\node[right] at (a4) {$4$};
\node[right] at (a5) {$3$};
\node[below] at (a6) {$2$};
\node[left] at (a2) {$1$};
\node[left] at (a3) {$0$};
\end{tikzpicture}\hspace{0.1cm} %
\begin{tikzpicture}[x=10mm, y=5mm, label distance = 1mm]
\foreach \x in {2,...,5,6}{
\coordinate (a\x) at (-18-\x*72:1cm) ; 
} 
\draw[thick] (a2)--(a3)--(a4)--(a5)--(a6)--(a2);
\draw[thick] (a4)--(a2)
(a5)--(a2);
\node[right] at (a4) {$4$};
\node[right] at (a5) {$3$};
\node[below] at (a6) {$2$};
\node[left] at (a2) {$1$};
\node[left] at (a3) {$0$};
\end{tikzpicture}%
\hspace{0.1cm} %
\begin{tikzpicture}[x=10mm, y=5mm, label distance = 1mm]
\foreach \x in {2,...,5,6}{
\coordinate (a\x) at (-18-\x*72:1cm) ; 
} 
\draw[thick] (a2)--(a3)--(a4)--(a5)--(a6)--(a2);
\draw[thick] (a5)--(a3)
(a5)--(a2);
\node[right] at (a4) {$4$};
\node[right] at (a5) {$3$};
\node[below] at (a6) {$2$};
\node[left] at (a2) {$1$};
\node[left] at (a3) {$0$};
\end{tikzpicture} \hspace{0.1cm} %
\begin{tikzpicture}[x=10mm, y=5mm, label distance = 1mm]
\foreach \x in {2,...,5,6}{
\coordinate (a\x) at (-18-\x*72:1cm) ; 
} 
\draw[thick] (a2)--(a3)--(a4)--(a5)--(a6)--(a2);
\draw[thick] (a5)--(a3)
(a3)--(a6);
\node[right] at (a4) {$4$};
\node[right] at (a5) {$3$};
\node[below] at (a6) {$2$};
\node[left] at (a2) {$1$};
\node[left] at (a3) {$0$};
\end{tikzpicture}%
\hspace{0.1cm} %
\begin{tikzpicture}[x=10mm, y=5mm, label distance = 1mm]
\foreach \x in {2,...,5,6}{
\coordinate (a\x) at (-18-\x*72:1cm) ; 
} 
\draw[thick] (a2)--(a3)--(a4)--(a5)--(a6)--(a2);
\draw[thick] (a4)--(a6)
(a3)--(a6);
\node[right] at (a4) {$4$};
\node[right] at (a5) {$3$};
\node[below] at (a6) {$2$};
\node[left] at (a2) {$1$};
\node[left] at (a3) {$0$};
\end{tikzpicture}
\caption{Triangulations of $\polygon{5}$}\label{fig_triangulation_5gon}
\end{figure}
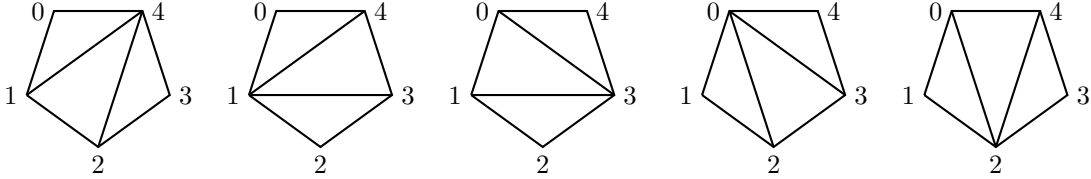

There are many interpretations of Catalan numbers such as binary trees, Dyck paths, and binary operations, see~\cite{Stanley_Catalan}. In the following, we review the correspondence between polygon triangulations and binary trees and then define a toric variety of Catalan type.

A \textit{binary tree} is a rooted plane tree with at most two children at each vertex defined recursively as follows. 
An empty set $\emptyset$ is a binary tree. Otherwise, a binary tree has a root vertex $v$, a left subtree, and a right subtree, both of which are binary trees. 
The set of triangulations~$\mathscr{T}$ of a convex polygon $\polygon{n+2}$ with $n+2$ vertices and the set of binary trees~$\B_\mathscr{T}$ with $n$ vertices have the following bijective connection defined recursively. Assume that $\mathscr{T}$ is a triangulation of $\polygon{n+2}$. For $n=0$, the associated binary tree is an empty set. For $n \geq 1$, we
put a vertex of $\B_\mathscr{T}$ in the interior of each triangle of $\mathscr{T}$. The root vertex of~$\B_\mathscr{T}$ corresponds to the vertex in the triangle having the side $\{0,n+1\}$. Now by deleting the side $\{0,n+1\}$ of the $(n+2)$-gon, we obtain a union of two triangulated polygons which define the left and right subtrees as follows: the polygon having the vertex $0$ defines the left subtree; the polygon having the vertex $n+2$ defines the right subtree.
The binary trees associated with the triangulation of $\polygon{5}$ in Figure~\ref{fig_triangulation_5gon} are shown in Figure~\ref{fig_binary_tree_v3}. The root vertex is the vertex with the additional circle.

\begin{figure}[h]
\begin{tikzpicture}[every node/.style = {circle, fill=black, inner sep = 1.5pt, },
level distance=5mm, 
level 1/.style={sibling distance=7mm}, 
level 2/.style={sibling distance=5mm},
level 3/.style={sibling distance=5mm} 
]
\node[black, draw, circle, inner sep = 1.5 pt, fill=black, double] {}
child[missing] {}
child{
node {}
child[missing] {}
child{
node {}
}
};
\end{tikzpicture}\hspace{1cm}%
\begin{tikzpicture}[every node/.style = {circle, fill=black, inner sep = 1.5pt, },
level distance=5mm, 
level 1/.style={sibling distance=7mm}, 
level 2/.style={sibling distance=5mm},
level 3/.style={sibling distance=5mm} 
]
\node[black, draw, circle, inner sep = 1.5 pt, fill=black, double] {}
child[missing] {}
child{
node {}
child{
node {}
}
child[missing] {} 
};
\end{tikzpicture}%
\hspace{1cm}%
\begin{tikzpicture}[every node/.style = {circle, fill=black, inner sep = 1.5pt, },
level distance=5mm, 
level 1/.style={sibling distance=7mm}, 
level 2/.style={sibling distance=5mm},
level 3/.style={sibling distance=5mm} 
]
\node[black, draw, circle, inner sep = 1.5 pt, fill=black, double] {}
child{
node {}
child[missing] {}
child{
node {}
}
} 
child[missing] {} ;
\end{tikzpicture}\hspace{1cm}%
\begin{tikzpicture}[every node/.style = {circle, fill=black, inner sep = 1.5pt, },
level distance=5mm, 
level 1/.style={sibling distance=7mm}, 
level 2/.style={sibling distance=5mm},
level 3/.style={sibling distance=5mm} 
]
\node[black, draw, circle, inner sep = 1.5 pt, fill=black, double] {}
child{
node {}
child{
node {}
}
child[missing] {} 
} 
child[missing] {} ;
\end{tikzpicture}\hspace{1cm}%
\raisebox{1.5em}{ \begin{tikzpicture}[every node/.style = {circle, fill=black, inner sep = 1.5pt, },
level distance=5mm, 
level 1/.style={sibling distance=7mm}, 
level 2/.style={sibling distance=5mm},
level 3/.style={sibling distance=5mm} 
]
\node[black, draw, circle, inner sep = 1.5 pt, fill=black, double] {}
child{
node {}
} 
child{
node{}
} ;
\end{tikzpicture}}
\caption{Binary trees with three vertices}\label{fig_binary_tree_v3}
\end{figure}
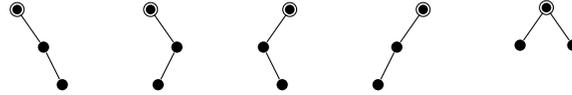

If a binary tree has zero or two children at each vertex, then it is called \emph{full}. 
Similarly to $\B_\mathscr{T}$, a full binary tree $\mathscr{C}_\mathscr{T}$, which contains $\B_\mathscr{T}$ as a subgraph, is associated with a triangulation $\mathscr{T}$ of~$\polygon{n+2}$. Indeed, in addition to the vertices of $\B_\mathscr{T}$, we place a vertex outside the side~$\{i,i+1\}$ of $\polygon{n+2}$ for each $i=0,\dots,n$, and connect it to the vertex in the interior of the triangle having the side~$\{i,i+1\}$ by an edge. This produces the full binary tree $\mathscr{C}_\mathscr{T}$. The root vertex of~ $\mathcal{C}_\mathscr{T}$ is the same as that of $\B_\mathscr{T}$. Moreover, if we delete all leaves of $\mathscr{C}_{\mathscr{T}}$, then we obtain the binary tree $\B_{\mathscr{T}}$. Here, a \emph{leaf} is a vertex with no children. Indeed, the number of vertices of $\mathscr{C}_{\mathscr{T}}$ is $2n+1$.
The full binary trees associated with the triangulation of $\polygon{5}$ in Figure~\ref{fig_triangulation_5gon} are shown in Figure~\ref{fig_full_binary_tree_v3}, where the colored vertices represent the vertices placed outside the polygon.

\begin{figure}[h]
\begin{tikzpicture}[nb/.style = {circle, fill=black, inner sep = 1.5pt, },
nc/.style = {circle, fill=purple!70, inner sep = 1.5pt, },
level distance=5mm, 
level 1/.style={sibling distance=7mm}, 
level 2/.style={sibling distance=5mm},
level 3/.style={sibling distance=5mm} 
]
\node[black, draw, circle, inner sep = 1.5 pt, fill=black, double] {}
child {
node[nc]{}
}
child{
node[nb] {}
child {
node[nc]{}
}
child{
node[nb] {}
child{
node[nc]{}
}
child{
node[nc]{}
}
}
};
\end{tikzpicture}\hspace{1cm}%
\begin{tikzpicture}[nb/.style = {circle, fill=black, inner sep = 1.5pt, },
nc/.style = {circle, fill=purple!70, inner sep = 1.5pt, },
level distance=5mm, 
level 1/.style={sibling distance=7mm}, 
level 2/.style={sibling distance=5mm},
level 3/.style={sibling distance=5mm} 
]
\node[black, draw, circle, inner sep = 1.5 pt, fill=black, double] {}
child{
node[nc] {}
}
child{
node [nb]{}
child{
node [nb]{}
child{
node[nc]{}
}
child{
node[nc]{}
}
}
child{
node[nc] {}
} 
};
\end{tikzpicture}%
\hspace{1cm}%
\begin{tikzpicture}[nb/.style = {circle, fill=black, inner sep = 1.5pt, },
nc/.style = {circle, fill=purple!70, inner sep = 1.5pt, },
level distance=5mm, 
level 1/.style={sibling distance=7mm}, 
level 2/.style={sibling distance=5mm},
level 3/.style={sibling distance=5mm} 
]
\node[black, draw, circle, inner sep = 1.5 pt, fill=black, double] {}
child{
node [nb]{}
child {
node [nc]{}
}
child{
node [nb]{}
child{
node[nc]{}
}
child{
node[nc]{}
}
}
} 
child{
node[nc] {}
} ;
\end{tikzpicture}\hspace{1cm}%
\begin{tikzpicture}[nb/.style = {circle, fill=black, inner sep = 1.5pt, },
nc/.style = {circle, fill=purple!70, inner sep = 1.5pt, },
level distance=5mm, 
level 1/.style={sibling distance=7mm}, 
level 2/.style={sibling distance=5mm},
level 3/.style={sibling distance=5mm} 
]
\node[black, draw, circle, inner sep = 1.5 pt, fill=black, double] {}
child{
node[nb] {}
child{
node [nb]{}
child{
node[nc]{}
}
child{
node[nc]{}
}
}
child{
node[nc]{}
} 
} 
child{
node[nc]{}
} ;
\end{tikzpicture}\hspace{1cm}%
\raisebox{1.5em}{ \begin{tikzpicture}[nb/.style = {circle, fill=black, inner sep = 1.5pt, },
nc/.style = {circle, fill=purple!70, inner sep = 1.5pt, },
level distance=5mm, 
level 1/.style={sibling distance=7mm}, 
level 2/.style={sibling distance=5mm},
level 3/.style={sibling distance=5mm} 
]
\node[black, draw, circle, inner sep = 1.5 pt, fill=black, double] {}
child{
node [nb]{}
child{
node[nc]{}
}
child{
node[nc]{}
}
} 
child{
node[nb]{}
child{
node[nc]{}
}
child{
node[nc]{}
}
} ;
\end{tikzpicture}}
\caption{Full binary trees with three vertices}\label{fig_full_binary_tree_v3}
\end{figure}
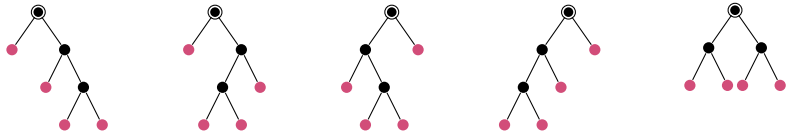

Now we introduce the left tree and the right tree of a triangulation $\mathscr{T}$ of $\polygon{n+2}$ using the corresponding full binary tree $\mathcal{C}_\mathscr{T}$. The \textit{left tree} $\lT$ (resp., \textit{right tree} $\rT$) consists of the sides in $\mathscr{T}$ intersecting with an edge of $\mathcal{C}_\mathscr{T}$ connecting a vertex and its left child (resp., right child). See Figure~\ref{fig_atom_and_coatom_together}.
\lee{The graphs $\lT$ and $\rT$ are indeed \emph{trees} as shown in~\cite[Lemma~4.4]{LMP_Catalan}.} 
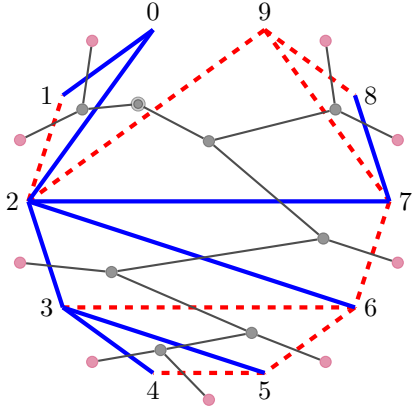
\begin{figure}[H]
\begin{center}
\begin{tikzpicture}[x=10mm, y=5mm, label distance = 1mm, scale = 0.8]
\foreach \x in {1,2,...,11}{
\coordinate (\x) at (3*36-\x*36:3cm) ; 
}

\node[above] at (11) {$9$};
\foreach \x/\y in {2/8,3/7,4/6}{
\node[right] at (\x) {$\y$};
}
\foreach \x/\y in {5/5,6/4}{
\node[below] at (\x) {$\y$};
}
\foreach \x/\y in {7/3,8/2,9/1}{
\node[left] at (\x) {$\y$};
}
\node[above] at (10) {$0$};

\draw[red, dashed,ultra thick ] 
(1) edge (2) 
(1) edge (3)
(1) edge (8)
(3) edge (4)
(5) edge (6)
(8) edge (9)
(4) edge (5)
(4) edge (7);
\draw[blue, ultra thick] (10) edge (8)
(10) edge (9)
(8) edge (7)
(8) edge (4)
(8) edge (3)
(7) edge (6)
(7) edge (5)
(3) edge (2) ;
\begin{scope}[every node/.style={opacity=0.6}]
\node[black!70, draw, circle, inner sep = 1.5 pt, fill=black!70, double] (v1) at (-6.5*36:2cm) {};
\node[black!70, draw, circle, inner sep = 1.5pt, fill=black!70] (v2) at (-6*36:2.6cm) {};
\node[black!70, draw, circle, inner sep = 1.5pt, fill=black!70] (v3) at (2.5*36:1cm) {}; 
\node[black!70, draw, circle, inner sep = 1.5pt, fill=black!70] (v4) at (1*36:2.6cm) {};
\node[black!70, draw, circle, inner sep = 1.5pt, fill=black!70] (v5) at (-0.5*36:2cm) {};
\node[black!70, draw, circle, inner sep = 1.5pt, fill=black!70] (v6) at (-4*36:2cm) {};
\node[black!70, draw, circle, inner sep = 1.5pt, fill=black!70] (v7) at (-2*36:2.3cm) {};
\node[black!70, draw, circle, inner sep = 1.5pt, fill=black!70] (v8) at (-3*36:2.6cm) {}; 
\foreach \x/\y in { 2/8, 3/7, 4/6, 5/5, 6/4, 7/3, 8/2, 9/1, 10/0}{
\node[purple!70, draw, circle, inner sep = 1.5pt, fill=purple!70] (vv\x) at (3*36-\x*36+18:3.3cm) {}; 
} 
\draw[black!70, thick] (v1)--(v2)
(v1)--(v3)
(v3)--(v4)
(v3)--(v5)
(v5)--(v6)
(v6)--(v7)
(v7)--(v8); %
\draw[black!70, thick] (v2)--(vv10)
(v2)--(vv9)
(v4)--(vv2)
(v4)--(vv3)
(v5)--(vv4)
(v6)--(vv8)
(v7)--(vv5)
(v8)--(vv6)
(v8)--(vv7);
\end{scope}
\end{tikzpicture}
\end{center}
\caption{The left tree (colored in blue) and the right tree (colored in red and dashed) of a triangulation of $\polygon{10}$. Here, we color the vertices $V(\mathcal{C}_\mathscr{T}) \setminus V(\mathcal{B}_\mathscr{T})$ purple.}\label{fig_atom_and_coatom_together}
\end{figure}

From the definition of the left and the right trees, 
we obtain the following lemma.
\begin{lemma}[{\cite[Lemma~4.5]{LMP_Catalan}}]\label{lemma_uniqueness_kl_k_kr}
Let $\mathscr{T}$ be a triangulation of $\polygon{n+2}$.
For each $1 \leq k \leq n$, there is only one edge $\{k_{\ll}, k\}$ with $k_{\ll} < k$ in the left tree $\lT$ and similarly there is only one edge $\{k, k_{\rr}\}$ with $k < k_{\rr}$ in the right tree $\rT$. Moreover, $k_{\ll}, k, k_{\rr}$ are the vertices of a triangle in $\mathscr{T}$. 
\end{lemma}

For the triangulation $\mathscr{T}$ in Figure~\ref{fig_atom_and_coatom_together}, we display the vertices $k_{\ll}$ and $k_{\rr}$ for each $1 \leq k \leq n$ in the following table:
\begin{center}
\begin{tabular}{c|cccccccc}
\toprule 
$k$ & $1$ & $2$ & $3$ & $4$ & $5$ & $6$ & $7$ & $8$ \\
\midrule 
$k_{\ll}$ & $0$ & $0$ & $2$ & $3$ & $3$ & $2$ & $2$ & $7$\\
$k_{\rr}$ & $2$ & $9$ & $6$ & $5$ & $6$ & $7$ & $9$ & $9$\\
\bottomrule
\end{tabular}
\end{center}
For each $k=1,\dots,n$, we define
\begin{enumerate}
\item $\mathbf{p}_k=\mathbf{e}_{k_{\ll}+1}-\mathbf{e}_{k+1}$,
\item $\mathbf{q}_k=-\mathbf{e}_k+\mathbf{e}_{k_{\rr}}$,
\end{enumerate}
where $\{\mathbf{e}_1,\dots,\mathbf{e}_{n+1}\}$ is the standard basis of $\Z^{n+1}$. The vectors $\mathbf{p}_1,\dots, \mathbf{p}_n$ (similarly, $\mathbf{q}_1,\dots,\mathbf{q}_n$) form a basis of the $n$-dimensional lattice 
\[
M=\{(x_1,\dots,x_{n+1})\in \Z^{n+1}\mid x_1+\dots+x_{n+1}=0\}. 
\]

Through the dot product on $\Z^{n+1}$, the dual lattice $N$ of~$M$ can be identified with the quotient lattice of $\Z^{n+1}$ by the sublattice generated by $(1,\dots,1)$, i.e.,  
\[
N=\Z^{n+1}/\Z(1,\dots,1).
\]
Let $\varpi_i$ $(i=1,\dots,n)$ be the quotient image of $\mathbf{e}_1+\dots+\mathbf{e}_i$ in $N$. Then $\{\varpi_1,\dots,\varpi_n\}$ is a basis of~$N$. For convenience, we set $\varpi_0=\varpi_{n+1}=\mathbf{0}$. 

To each side $\{a,b\}$ in the triangulation $\mathscr{T}$ with $a <b$, we assign the vector $\varpi_a-\varpi_b$ and denote it by $\mathbf{v}_a$ when $\{a,b\} \in \rT$ and $\mathbf{w}_b$ when $\{a,b\} \in \lT$, in other words, 
\begin{equation*}\label{eq:defvkwk}
\begin{split}
\mathbf{v}_{k} &= \varpi_{k}- \varpi_{k_{\rr}},\\ 
\mathbf{w}_k &= \varpi_{k_{\ll}} - \varpi_{k}, 
\end{split}
\end{equation*}
for $k=1,\dots,n$ by Lemma~\ref{lemma_uniqueness_kl_k_kr}.
Note that the zero vector $\mathbf{0}$ is assigned to the distinguished side ${\{0,n+1\}}$ because $\varpi_0=\varpi_{n+1}=\mathbf{0}$. 
Then for the paring $\langle\ ,\ \rangle$ between $N$ and $M$ induced from the dot product on~$\Z^{n+1}$, we have 
\[
\langle\mathbf{v}_i,\mathbf{p}_j\rangle=\langle \mathbf{w}_i,\mathbf{q}_j\rangle=\delta_{ij},
\]
where $\delta_{ij}$ denotes the Kronecker delta. 
See Figure~\ref{fig_vectors_left_right_trees} for the assignment of the vectors $\mathbf{w}_k$'s and $\mathbf{v}_k$'s to the left and right trees of the triangulation in Figure~\ref{fig_atom_and_coatom_together}. One can check that $\langle\mathbf{v}_i,\mathbf{p}_j\rangle=\langle \mathbf{w}_i,\mathbf{q}_j\rangle=\delta_{ij}$ using Table~\ref{table1}.

\begin{table}[H]
\begin{varwidth}[b]{0.55\linewidth}
\centering
\begin{tabular}{c|c|c|c|cl}
\toprule
$k$ & $\mathbf{p}_k$ & $\mathbf{q}_k$ & $\mathbf{v}_k$ & $\mathbf{w}_k$ \\
\midrule 
$1$ & $\mathbf{e}_1-\mathbf{e}_2$ & $-\mathbf{e}_1+\mathbf{e}_2$ & $\varpi_1 - \varpi_2$ & $-\varpi_1$ \\
$2$ & $\mathbf{e}_1-\mathbf{e}_3$ & $-\mathbf{e}_2+\mathbf{e}_9$ & $\varpi_2$ & $ -\varpi_2$ \\
$3$ & $\mathbf{e}_3-\mathbf{e}_4$ & $-\mathbf{e}_3+\mathbf{e}_6$ & $\varpi_3 - \varpi_6$ & $\varpi_2 - \varpi_3$ \\
$4$ & $\mathbf{e}_4-\mathbf{e}_5$ & $-\mathbf{e}_4+\mathbf{e}_5$ & $\varpi_4 - \varpi_5$ & $\varpi_3 - \varpi_4$ \\
$5$ & $\mathbf{e}_4-\mathbf{e}_6$ & $-\mathbf{e}_5+\mathbf{e}_6$ & $\varpi_5 - \varpi_6 $ & $\varpi_3 - \varpi_5$ \\
$6$ & $\mathbf{e}_3-\mathbf{e}_7$ & $-\mathbf{e}_6+\mathbf{e}_7$ & $\varpi_6 - \varpi_7$ & $\varpi_2 - \varpi_6$ \\
$7$ & $\mathbf{e}_3-\mathbf{e}_8$ & $-\mathbf{e}_7+\mathbf{e}_9$ & $\varpi_7$ & $\varpi_2 - \varpi_7$\\
$8$ & $\mathbf{e}_8-\mathbf{e}_9$ & $-\mathbf{e}_8+\mathbf{e}_9$ & $\varpi_8$ & $\varpi_7 - \varpi_8$ \\
\bottomrule
\end{tabular}
\caption{Vectors $\mathbf{p}_k, \mathbf{q}_k, \mathbf{v}_k, \mathbf{w}_k$ for the triangulation in Figure~\ref{fig_atom_and_coatom_together}}
\label{table1}
\end{varwidth}%
\hfill
\begin{minipage}[b]{0.45\linewidth}
\centering
\begin{tikzpicture}[x=10mm, y=5mm, label distance = 1mm, scale = 0.68]
\foreach \x in {1,2,...,11}{
\coordinate (\x) at (3*36-\x*36:3cm) ; 
} 

\node[above] at (11) {$9$};
\foreach \x/\y in {2/8,3/7,4/6}{
\node[right] at (\x) {$\y$};
}
\foreach \x/\y in {5/5,6/4}{
\node[below] at (\x) {$\y$};
}
\foreach \x/\y in {7/3,8/2,9/1}{
\node[left] at (\x) {$\y$};
}
\node[above] at (10) {$0$};
\draw[blue, thick] (10)--(9) node[left, midway, black]{$\mathbf{w}_1$};
\draw[blue, thick] (10)--(8) node[right, midway, black]{$\mathbf{w}_2$};
\draw[blue, thick] (8)--(7) node[left, midway, black]{$\mathbf{w}_3$};
\draw[blue, thick] (8)--(4) node[left, midway, black]{$\mathbf{w}_6$};
\draw[blue, thick] (8)--(3) node[below, midway, black]{$\mathbf{w}_7$};
\draw[blue, thick] (7)--(6) node[left, midway, black]{$\mathbf{w}_4$};
\draw[blue, thick] (7)--(5) node[right, midway, black]{$\mathbf{w}_5$};
\draw[blue, thick] (3)--(2) node[right, midway, black]{$\mathbf{w}_8$};
\draw[red, dashed, thick] (1)--(2) node[right, midway, black]{$\mathbf{v}_8$};
\draw[red, dashed, thick] (1)--(3) node[left, midway, black]{$\mathbf{v}_7$};
\draw[red, dashed, thick] (1)--(8) node[right, midway, black]{$\mathbf{v}_2$};
\draw[red, dashed, thick] (3)--(4) node[right, midway, black]{$\mathbf{v}_6$};
\draw[red, dashed, thick] (5)--(6) node[below, midway, black]{$\mathbf{v}_4$};
\draw[red, dashed, thick] (8)--(9) node[left, midway, black]{$\mathbf{v}_1$};
\draw[red, dashed, thick] (4)--(5) node[right, midway, black]{$\mathbf{v}_5$};
\draw[red, dashed, thick] (4)--(7) node[above, midway, black]{$\mathbf{v}_3$};
\end{tikzpicture}

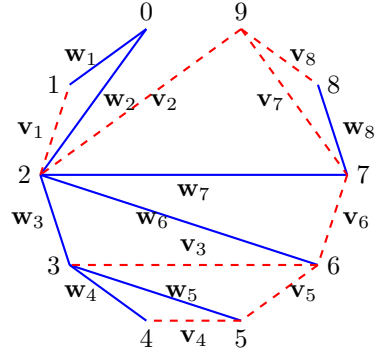
\captionof{figure}{Vectors assigned to the sides of $\mathscr{T}$ in Figure~\ref{fig_atom_and_coatom_together}}\label{fig_vectors_left_right_trees}
\end{minipage}
\end{table}

\begin{lemma}[{\cite[Lemma 6.1]{LMP_Catalan}}] \label{lemm:fan}
The collection of cones spanned by $\mathbf{u}_i$ $(i\in I)$, where $\mathbf{u}_i$ is either $\mathbf{v}_i$ or $\mathbf{w}_i$ and $I$ runs over all subsets of $[n]$, forms a complete non-singular fan $\Sigma_\mathscr{T}$ in $N\otimes \R$. 
\end{lemma}

The above lemma says that the underlying simplicial complex of the fan $\Sigma_\mathscr{T}$ is the boundary complex of an $n$-dimensional cross-polytope. It is known from~\cite[Corollary~3.5]{MasudaPanov08} that such a fan is indeed the normal fan of an $n$-cube and the smooth compact toric variety $X(\Sigma_\mathscr{T})$ associated with the fan $\Sigma_\mathscr{T}$ is a Bott manifold (cf. Proposition~\ref{proposition_MasudaPanov}).

\begin{definition}
Since $\Sigma_\mathscr{T}$ is associated with a polygon triangulation $\mathscr{T}$, we say that the fan $\Sigma_\mathscr{T}$ and the corresponding (smooth compact) toric variety $X(\Sigma_\mathscr{T})$ are of \emph{Catalan type}. 
\end{definition}

Since the vertices $k_{\ll},k,k_{\rr}$ form a triangle for each $k\in [n]$, there is a unique $k_0\in [n]$ such that $\mathbf{v}_{k_0} + \mathbf{w}_{k_0} = \mathbf{0}$, and for $k\in [n]\backslash\{k_0\}$ we have 
\begin{equation*} \label{eq:vkwksum}
\mathbf{v}_k + \mathbf{w}_k = \begin{cases} \mathbf{v}_{k_{\ll}} & \text{ if } \{ k_{\ll}, k_{\rr}\} \in E(\rT), \\
\mathbf{w}_{k_{\rr}} & \text{ if } \{k_{\ll}, k_{\rr}\} \in E(\lT).
\end{cases}
\end{equation*} 
Note that $k_0$ is the remaining vertex of the triangle containing the distinguished side $\{0, n+1\}$. See Figure~\ref{fig_vectors_left_right_trees} and~\cite[Lemma 5.1]{LMP_Catalan} for more details. Hence the set of primitive collections of $\Sigma_\mathscr{T}$ is
\begin{equation*}
\{ \{ \mathbf{v}_k, \mathbf{w}_k \} \mid k \in [n] \}.
\end{equation*}
\lee{(See Definition~\ref{def_primitive} for definition of primitive collections.)}
Then Batyrev's criterion (\cite[Proposition 2.3.6]{Batyrev}, see also Appendix~\ref{sec_Fano_Bott}) implies the following:

\begin{lemma} [{\cite[Lemma 6.4]{LMP_Catalan}}]\label{lemm:Fano}
The toric variety $X(\Sigma_\mathscr{T})$ is Fano. 
\end{lemma}

The toric varieties of Catalan type are classified as follows. 

\begin{theorem}[{\cite[Theorem 6.5]{LMP_Catalan}}]\label{theo:fan_and_toric_of_Catalan}
Let $\mathscr{T}$ and $\mathscr{T}'$ be triangulations of $\polygon{n+2}$. Then the fans $\Sigma_\mathscr{T}$ and $\Sigma_{\mathscr{T}'}$ are isomorphic \textup{(}equivalently, the toric varieties $X(\Sigma_\mathscr{T})$ and $X(\Sigma_{\mathscr{T}'})$ are isomorphic\textup{)} if and only if the binary trees~$\B_{\mathscr{T}}$ and $\B_{\mathscr{T}'}$ are isomorphic as \emph{unordered} rooted trees.
\end{theorem}

The above theorem implies that the number of isomorphism classes of $n$-dimensional toric varieties of Catalan type is the same as that of unordered binary trees with $n$ vertices. It is known that the latter is the Wedderburn--Etherington number $b_{n+1}$.

\begin{corollary}[{\cite[Corollary 6.6]{LMP_Catalan}}]\label{cor_enumerate}
The number of isomorphism classes of $n$-dimensional toric varieties of Catalan type is the Wedderburn--Etherington number $b_{n+1}$. 
\end{corollary} 

Here, the Wedderburn--Etherington number $b_n$ $(n\ge 1)$ is the number of ways to parenthesize a string of $n$ letters subject to a commutative (but nonassociative) binary operation and it appears in counting several different objects (see \href{https://oeis.org/A001190}{Sequence A001190} in OEIS~\cite{OEIS}, \cite[A56 in p.133]{Stanley_Catalan}). The generating function $B(x)=\sum_{n\ge 1} b_nx^n$ of the Wedderburn--Etherington numbers satisfies the functional equation
\[
B(x)=x+\frac{1}{2}B(x)^2+\frac{1}{2}B(x^2), 
\] 
which was the motivation of Wedderburn in his work~\cite{Wedderburn22} {and was considered by Etherington~\cite{Etherington37}}.
This functional equation is equivalent to the recurrence relation
\[
b_{2m-1}=\sum_{i=1}^{m-1}b_ib_{2m-i-1}\quad (m\ge 2),\qquad b_{2m}=b_m(b_m+1)/2+\sum_{i=1}^{m-1}b_ib_{2m-i}
\]
with $b_1=1$. Using this recurrence relation, one can calculate the Wedderburn--Etherington numbers, see Table~\ref{table_bn}. 

\begin{table}[H]
\begin{tabular}{c|rrrrrrrrrrrrrrr}
\toprule
$n$ & $1$ & $2$ & $3$ & $4$ & $5$ & $6$ & $7$ & $8$ & $9$ & $10$ & $11$ & $12$ & $13$ & $14$ & $15$ \\
\midrule
$b_n$ & $1$ & $1$ & $1$ & $2$ & $3$ & $6$ & $11$ & $23$ & $46$ & $98$ & $207$ & $451$ & $983$ & $2179$ & $4850$ \\ 
\bottomrule
\end{tabular}
\caption{Wedderburn--Etherington numbers $b_n$ for small values of $n$} \label{table_bn}
\end{table}


\subsection{Smooth toric Richardson varieties of Catalan type}\label{sec:smooth toric Richardson varieties of Catalan type}
We say that a smooth toric Richardson variety $\Xvw$ is of Catalan type if it is of Catalan type as a toric variety, in other words, if the normal fan of $\Q^{v}_{w}$ is of Catalan type. 
In this subsection, we show that every toric variety of Catalan type appears as a smooth toric Richardson variety. However, the converse is not true, i.e.,  there are smooth toric Richardson varieties which are not of Catalan type. For instance, $X_{3142}$ is a (smooth) toric Schubert variety but not Fano, so it is not of Catalan type by Lemma~\ref{lemm:Fano}. 

For a permutation $u \in {S}_n$, we define permutations $\uh{u}$ and $\ut{u}$ in $S_{n+1}$ by
\[
\uh{u}(i) = \begin{cases}
1 & \text{ if } i = 1, \\
u(i-1)+1 & \text{ if }2 \leq i \leq n+1, 
\end{cases}
\qquad
\ut{u}(i) = \begin{cases}
u(i) & \text{ if } 1 \leq i \leq n, \\
n+1 & \text{ if } i =n+1.
\end{cases}
\]
For example, if $u = 2314$, then $\uh{u} = 13425$ and $\ut{u} = 23145$. As one may see, the permutation $\uh{u}$ is obtained from $u$ by putting the additional number $1$ at the \textit{head} (with the original numbers increased by $1$) while $\ut{u}$ is obtained from $u$ by putting the additional number $n+1$ at the \textit{tail}. In the notation, $\head$ stands for \textit{head} and $\tail$ stands for \textit{tail}. 

One notes that if a pair $(v,w)$ of elements in $\Sn$ satisfies 
\[
w = s(n,1)v\quad \text{(resp. $w=s(1,n)v$)} \quad \text{ and } \quad\ell(w) = \ell(v) + n,
\]
then $v(1)=1$ (resp. $v(n+1)=n+1$), so that 
\begin{equation} \label{eq:pairvw}
\text{$(v^{-1},w^{-1}) = (\uh{u}, \uh{u}s(1,n))$ (resp. $(v^{-1},w^{-1})=(\ut{u}, \ut{u}s(n,1))$\quad for some $u \in {S}_{n}$.}
\end{equation}
\lee{Here, $s(n,1) = s_{n} s_{n-1} \cdots s_1$ and $s(1,n) = s_1 s_2 \cdots s_{n}$ as in~\eqref{eq_s_pq}.}

By Proposition~\ref{prop:6-2}, the Bruhat interval polytope $\Q^v_w$ for the pair~$(v,w)$ in \eqref{eq:pairvw} is an $n$-cube and our concern is the pairs in \eqref{eq:pairvw}. We first consider the former case $(v^{-1},w^{-1}) = (\uh{u}, \uh{u}s(1,n))$. 

We recall a surjection from ${S}_n$ to the set of binary trees with $n$ vertices (cf.~\cite[Appendix~A]{LFCG_dialgebras}). 
{To} a permutation $u \in {S}_n$, we first associate a binary tree $\widetilde\psi(u)$ with vertex labels by finding the smallest number in the one-line notation of $u$ inductively. We start with the one-line notation $u(1) u(2) \ \cdots \ u(n)$ of $u$. 
The smallest integer, say $u(p)$, in the sequence (which is $1$ here) becomes the root of the binary tree $\widetilde\psi(u)$ with $n$ vertices. Then the subsequence $u(1) \ \cdots \ u(p-1)$ will provide the left subtree of the root vertex, and the subsequence $u(p+1) \ \cdots \ u(n)$ will provide the right subtree of the root vertex. 
More precisely, the smallest integer in the subsequence $u(1) \ \cdots \ u(p-1)$ presents the root of a binary tree with $p-1$ vertices, and it is the left child of the root vertex of~$\widetilde\psi(u)$. On the other hand, the smallest integer in the subsequence $u(p+1) \ \cdots \ u(n)$ presents the root of a binary tree with $n-p$ vertices, and it is the right child of the root vertex of $\widetilde\psi(u)$. Continuing this process, we obtain the binary tree $\widetilde\psi(u)$. Finally, erasing the vertex labels, we obtain the desired binary tree $\psi(u)$ with $n$ vertices.

For example, if $u=31687524$, the root of the binary tree $\widetilde\psi(u)$ is $u(2)$, and its left and right subtrees have the roots $u(1)$ and $u(7)$, respectively. Continuing this process, we first get the binary tree $\widetilde\psi(u)$, and then by erasing the labels we obtain the binary tree $\psi(u)$, see Figure~\ref{fig_binary_tree_and_permutation}. Note that $v=21687534$ gives the same binary tree as $u=31687524$, i.e.,  $\psi(v)=\psi(u)$.
\begin{figure}[hbt]
\centering
\begin{subfigure}[b]{0.45\textwidth}
\centering
\begin{tikzpicture}[node/.style={}, baseline = -0.5ex,scale=.8]
\node[node] (1) at (0,5) {$u(2)$};
\node[node] (2) at (-2,4) {$u(1)$};
\node[node] (3) at (3,4) {$u(7)$};
\node[node] (4) at (5,3) {$u(8)$};
\node[node] (5) at (1,3) {$u(6)$};
\node[node] (6) at (-1,2) {$u(3)$};
\node[node] (7) at (1,1) {$u(5)$};
\node[node] (8) at (-1,0) {$u(4)$};
\draw (1) to (2);
\draw (1) to (3);
\draw (3) to (4);
\draw (3) to (5);
\draw (5) to (6);
\draw (6) to (7);
\draw (7) to (8);
\end{tikzpicture}
\end{subfigure}
\hspace{.5cm}
\begin{subfigure}[b]{0.45\textwidth}
\centering
\begin{tikzpicture}[scale=.8]
\begin{scope}
\node[black!70, draw, circle, inner sep = 1.5 pt, fill=black!70, double] (1) at (0,5) {};
\node[black!70, draw, circle, inner sep = 1.5pt, fill=black!70] (2) at (-2,4) {};
\node[black!70, draw, circle, inner sep = 1.5pt, fill=black!70] (3) at (3,4) {}; 
\node[black!70, draw, circle, inner sep = 1.5pt, fill=black!70] (4) at (5,3) {};
\node[black!70, draw, circle, inner sep = 1.5pt, fill=black!70] (5) at (1,3) {};
\node[black!70, draw, circle, inner sep = 1.5pt, fill=black!70] (6) at (-1,2) {};
\node[black!70, draw, circle, inner sep = 1.5pt, fill=black!70] (7) at (1,1) {};
\node[black!70, draw, circle, inner sep = 1.5pt, fill=black!70] (8) at (-1,0) {}; 
\end{scope}
\draw (1) to (2);
\draw (1) to (3);
\draw (3) to (4);
\draw (3) to (5);
\draw (5) to (6);
\draw (6) to (7);
\draw (7) to (8);
\end{tikzpicture}
\end{subfigure}
\caption{The binary trees $\widetilde\psi(u)$ and $\psi(u)$ for $u = 31687524$}\label{fig_binary_tree_and_permutation}
\end{figure}
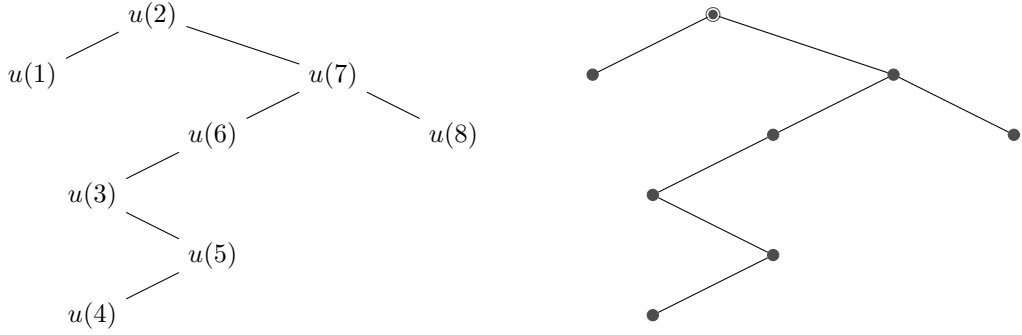

Through the canonical bijection between the set of binary trees with $n$ vertices and that of triangulations of $\polygon{n+2}$, 
the assignment 
\[
\psi \colon {S}_n \twoheadrightarrow \{ \text{binary trees with $n$ vertices} \}=\{\text{triangulations of $\polygon{n+2}$}\} 
\]
is surjective ({cf.~\cite[Appendix~A]{LFCG_dialgebras}}).

\begin{proposition}[{\cite[Proposition 7.2]{LMP_Catalan}}]\label{prop_u_head_atoms_and_lTrT}
Let $u \in {S}_n$ and let $\mathscr{T} = \psi(u)$ be the corresponding triangulation of $\polygon{n+2}$. We denote by $\lT$ and $\rT$ the left and right trees of $\mathscr{T}$ as before. Then the edges of $\lT$ correspond to the atoms of the Bruhat interval $[\uh{u}, \uh{u}s(1,n)]$ while the edges of $\rT$ correspond to the coatoms of the Bruhat interval $[\uh{u}, \uh{u}s(1,n)]$. More precisely, 
\[
\begin{split}
&\{{(i,j)} \mid \uh{u} \lessdot \uh{u} t_{i,j} \leq \uh{u} s(1,n) \} =
\{ {(i,j)} \mid \{i-1,j-1\} \in E(\lT)\}, \\
&\{{(i,j)} \mid \uh{u} \leq \uh{u} s(1,n) t_{i,j} \lessdot \uh{u} s(1,n) \} 
= \{{(i,j)} \mid \{i,j\} \in E(\rT)\}.
\end{split}
\]
Here, $x \lessdot y$ means that $x<y$ and there is no $z$ such that 
$x < z < y$. 
\end{proposition}

\begin{example}\label{example_atom}
For $u = 31687524$, the triangulation $\mathscr{T}=\psi(u)$ is as shown in Figure~\ref{fig_atom_and_coatom_together}. Since $v^{-1}=\uh{u}= 142798635$ and $w^{-1}=\uh{u} s(1,8) = 427986351$, 
there are eight atoms of the interval $[v^{-1},w^{-1}]$ given by $v^{-1} t_{i,j}$, where $(i,j)$ is one of the following pairs:
\[
(1,2), (1,3), (3,4), (4,5), (4,6), (3,7), (3,8), (8,9).
\] 
These pairs provide the edges of $\lT$ by subtracting $1$ from every component.
On the other hand, there are eight coatoms given by $w^{-1} t_{i,j}$, where $(i,j)$ is one of the following pairs: 
\[
(1,2), (2,9), (3,6), (4,5), (5,6), (6,7), (7,9), (8,9).
\]
These pairs are the edges of $\rT$.
\end{example}

\begin{theorem}[{\cite[Theorem 7.4]{LMP_Catalan}}] \label{theo:normal_fan_s(1,n)}
For $u\in {S}_n$, the normal fan of the Bruhat interval polytope~$\Q^v_w$ for $(v^{-1},w^{-1})=({\uh{u}},{\uh{u}s(1,n)})$ is the fan~$\Sigma_{\psi(u)}$ associated with the triangulation~$\psi(u)$ of $\polygon{n+2}$. 
\end{theorem}

Hence, any $n$-dimensional fan of Catalan type is realized as the normal fan of $\Q^{v}_{w}$ with $(v^{-1},w^{-1})=(\uh{u},\uh{u}s(1,n))$ for some $u\in {S}_n$.
As for the latter case $(v^{-1},w^{-1})=(\ut{u},\ut{u}s(n,1))$, we have the following. 

\begin{theorem}[{\cite[Theorem 7.5]{LMP_Catalan}}] \label{theo:normal_fan_s(n,1)}
For $u\in {S}_n$, the normal fan of the Bruhat interval polytope $\Q^v_w$ for $(v^{-1},w^{-1})=(\ut{u},\ut{u}s(n,1))$ is isomorphic to the fan $\Sigma_{\psi(w_0uw_0)}$ associated with the triangulation~$\psi(w_0uw_0)$ of $\polygon{n+2}$, where $w_0$ denotes the longest element of ${S}_n$. 
\end{theorem}

The following is a direct consequence of Corollary~\ref{cor_enumerate}. 

\begin{corollary}[{\cite[Corollary 8.2]{LMP_Catalan}}]\label{cor_enumeration_Xvw}
The number of isomorphism classes of $n$-dimensional smooth toric Richardson varieties of Catalan type is the Wedderburn--Etherington number $b_{n+1}$. 
\end{corollary}


\section{Problems}\label{sec_problems}
The study of torus orbit closures in the flag variety is related to the geometry of Schubert varieties and the combinatorics of Bruhat interval polytopes. In this section, we pose some possible avenues for further exploration related to the discussions in this chapter.

\subsection{Poincar\'e polynomial of \texorpdfstring{$Y^v_w$}{Yvw}.}
It is known that a Schubert variety $X_w$ is smooth (in type~$A$) if its Poincar\'e polynomial is palindromic (see \cite[Theorem~6.0.4 and p.208]{BL20Singular}). 
\masuda{Similarly,} \lee{the generic torus orbit closure $Y_w$ in $X_w$ is smooth if its Poincar\'e polynomial is palindromic (see \cite[Proposition~5.6]{Gaetz22_one-skeleton}\footnote{ 
We notice that in the previous version of this article, we posed a problem: \emph{Is the generic torus orbit closure $Y_w$ in $X_w$ smooth if its Poincar\'e polynomial is palindromic?} and it has been solved in~\cite[Proposition~5.6]{Gaetz22_one-skeleton}.)})}\park{.} 
\masuda{When $Y_w$ is singular, its Poincar\'e polynomial still has some restrictions. For instance, Brion's  inequalities (see \cite{Brion00})  tell us that if we express 
\[
\Poin(Y_w,t)=\sum_{i=0}^{d} b_it^{2i}
\]
where $d=\dim_\C Y_w$, then $\sum_{i=0}^k b_i\le \sum_{i=0}^k b_{d-i}$ for $0\le \forall k\le [d/2]$.  
\begin{Problem}
Is it true that $b_k\le b_{d-k}$ for $0\le \forall k\le [d/2]$?
\end{Problem}
}

%
%
%
Recall that $Y^v_w$ is the generic torus orbit closure in the Richardson variety $X^v_w$.  
The Poincar\'e polynomial of $Y_w$ is computable by Theorem~\ref{Thm_Poincare} but such a formula is not known for $Y^v_w$ in general unless $Y^v_w$ is smooth.  

\begin{Problem}
\begin{enumerate}
\item Does the virtual Poincar\'e polynomial of $Y^v_w$ agree with the Poincar\'e polynomial $\Poin(Y^v_w,t)$ of $Y^v_w$? (See Remark~\ref{rmk_existence_of_F} for virtual Poincar\'e polynomials.) 
\item Is there a polynomial $A^v_w(t)$ defined similarly to $A_w(t)$ such that $\Poin(Y^v_w,t)=A^v_w(t^2)$?
\item Similarly to \park{\cite[Proposition~5.6]{Gaetz22_one-skeleton},} is $Y^v_w$ smooth if $\Poin(Y^v_w,t)$ is palindromic?
\end{enumerate}
\end{Problem}

%

\subsection{Combinatorics of \texorpdfstring{$\Q^v_w$}{Qvw}}
We pose five problems on the combinatorics of Bruhat interval polytopes.  

There are many pairs $(v,w)$ such that the Bruhat interval polytope $\Q^v_w$ is a cube. However, those pairs are not completely understood.

\begin{Problem}
Find all pairs $(v,w)$ such that $\Q^v_w$ is combinatorially equivalent to a cube, equivalently $X^v_w$ is a smooth toric variety.  
\end{Problem}

The following is a restatement of Conjecture~\ref{conjecture_smoothness_Ywv}. 

\begin{Problem}
Is $\Q^v_w$ simple when the two vertices $\vertex{v}$ and $\vertex{w}$ are  simple in $\Q^v_w$?
\end{Problem}

Since $\Q^v_w$ and $\Q^{v^{-1}}_{w^{-1}}$ have the same dimension, 
$\Q^v_w$ is toric if and only if $\Q^{v^{-1}}_{w^{-1}}$ is toric. Moreover, when $\Q^v_w$ is toric, its combinatorial type is determined by the poset structure of $[v,w]$. Therefore, when $\Q^v_w$ is toric, $\Q^{v^{-1}}_{w^{-1}}$ is simple if $\Q^v_w$ is simple.  
We ask whether we can drop the condition \emph{toric} in this statement.

\begin{Problem}
Is $\Q^{v^{-1}}_{w^{-1}}$ simple if $\Q^v_w$ is simple?
\end{Problem}

In the proof of Theorems~\ref{thm:smooth-one} and \ref{thm:singular-one}, we use the fact that if $s_r$ does not appear in a reduced decomposition of $w$ and $v\le w$, then $\Q^v_{s_rw}$ and $\Q^v_{ws_r}$ are combinatorially equivalent to $\Q^v_w \times \Q_{s_r}$ (see \cite[Proposition~5.7]{LMP_complexity_one}). We wonder whether this fact can be generalized as follows. 

\begin{Problem}
Consider  $u,v,w\in\Sn$ satisfying $v\leq w$ and $[e,u]\cap [e,w]=\{e\}$. Are $\Qvw{v}{wu}$ and $\Qvw{v}{uw}$ combinatorially equivalent to $\Qvw{v}{w}\times \Q_{u}$?
\end{Problem}

We set
\[
c(v,w) \colonequals \ell(w)-\ell(v)-\dim \Q^v_w.
\]
It is the complexity of the $T$-action on the Richardson variety $X^v_w$. Note that $c(e,w)=c(w)$, see Definition~\ref{def:c(w)} for $c(w)$.
We note that studying the complexity $c(v,w)$ of a Richardson variety is closely related to studying the complexity of the corresponding \textit{Kazhdan--Lusztig variety} as is shown in~\cite[Corollary~4.16]{dontenbury2022complexity}.

Recall from Theorem~\ref{thm:smooth-one} that a polytope $\Q_w$ is simple and $c(w) = 1$ if and only if $\Q_w$ is combinatorially equivalent to the product of the hexagon and the cube of dimension $\ell(w)-3$. 
Moreover, a polytope $\Q^v_w$ is simple and $c(v,w) = 0$ if and only if $\Q^v_w$ is combinatorially equivalent to a cube (see Theorem~\ref{thm:toric} and Theorem~\ref{theo:3-6}). Therefore it is natural to ask the following:

\begin{Problem}
Is a polytope $\Q^v_w$ simple and $c(v,w)=1$ if and only if $\Q^v_w$ is combinatorially equivalent to the product of the hexagon and the cube of dimension~$\ell(w)-\ell(v)-3$? 
\end{Problem}


\appendix 
\section{Toric varieties}
\label{appendix}

\subsection{Toric varieties}\label{sec_toric_varieties}
We recall the background of the theory of toric varieties from~\cite{CLS11Toric}. We first recall the definition of toric varieties.

\begin{definition}[{\cite[Definition~3.1.1]{CLS11Toric}}]
A \emph{toric variety} of complex dimension $n$ is a normal algebraic variety containing an algebraic torus $T \colonequals (\C^\ast)^n$ as a Zariski open dense subset such that the action of the torus on itself extends to the whole variety. 
\end{definition}

\begin{definition}[{\cite[Definition~3.1.2]{CLS11Toric}}]
Let $N$ be a lattice. A \emph{fan} $\Sigma$ in $N_{\R} \colonequals N \otimes_{\Z}\R$ is a finite collection of cones $\sigma \subseteq N_{\R}$ such that:
\begin{enumerate}
\item Every $\sigma \in \Sigma$ is a strongly convex rational polyhedral cone, i.e.,  for each $\sigma$ there exists a finite set $S \subset N$ such that
\[
\sigma = \text{Conv}(S) \colonequals \left\{ \sum_{u \in S} c_u u \,\middle|\, c_u \geq 0 \right\} \subset N_{\R}
\]
and $\sigma \cap (-\sigma) = \{\boldsymbol{0}\}$.
\item For all $\sigma \in \Sigma$, each face of $\sigma$ is also in $\Sigma$.
\item For all $\sigma_1, \sigma_2 \in \Sigma$, the intersection $\sigma_1 \cap \sigma_2$ is a face of each. 
\end{enumerate}
\end{definition}

Let $\Sigma$ be a fan in $N_{\R}$. Let $M$ be a dual lattice of $N$ and we set $M_{\R} \colonequals M \otimes_{\Z} \R$. We denote by~$\langle~,~\rangle$ the pairing between $M_{\R} \colonequals M \otimes_{\Z} \R$ and $N_{\R}$. 
Each cone $\sigma \in \Sigma$ gives an affine toric variety
\[
U_{\sigma} = \text{Spec}(\C [\sigma^{\vee} \cap M]),
\]
where
\[
\sigma^{\vee} \colonequals \{ m \in M_{\R} \mid \langle m, u \rangle \geq 0 \text{ for all } u \in \sigma \}.
\]
By gluing these affine toric varieties, we get a variety $X_{\Sigma}$. 
It turns out that this variety is a toric variety and there is a correspondence between normal separated toric varieties and fans.
\begin{theorem}[{\cite[Theorem~3.1.5 and Corollary~3.1.8]{CLS11Toric}}]
For a fan $\Sigma$, the variety $X_{\Sigma}$ is a normal separated toric variety. Conversely, for a normal separated toric variety $X$, there exists a fan $\Sigma$ such that $X$ is isomorphic to $X_{\Sigma}$.
\end{theorem}

A \emph{convex polytope} is the convex hull  of a finite set of points in the Euclidean space $\R^n$. It is well known that every convex polytope is a bounded intersection of finitely many half-spaces. Two polytopes are \emph{combinatorially equivalent} if their face posets are isomorphic.
A \emph{lattice polytope} is a convex polytope whose vertices are in the lattice $\Z^n\subset \R^n$.

For a full dimensional lattice polytope $\mathsf{P}\subset {\R}^n$, one can associate a fan, called the \emph{normal fan of $\mathsf{P}$}. 
Consider the presentation of $\mathsf{P}$ given by the intersection of half-spaces:
\[
\mathsf{P} = \{ m \in \R^n \mid \langle m, u_{\mathsf F} \rangle \geq - a_{\mathsf F} \text{ for every facet }\mathsf F \}.
\]
For a face $\mathsf{Q}$ of $\mathsf{P}$, we set 
\[
\sigma_{\mathsf Q} \colonequals \text{Cone}(u_{\mathsf F} \mid \mathsf F \text{ contains } \mathsf Q).
\]
Thus the cone $\sigma_{\mathsf F}$ is the ray generated by $u_{\mathsf F}$ for a facet $\mathsf F$ and $\sigma_{\mathsf{P}} = \{0\}$. It is known that 
\[
\Sigma_{\mathsf P} \colonequals \{ \sigma_{\mathsf Q} \mid \mathsf Q \text{ is a face of }\mathsf P\}
\] 
becomes a fan (see~\cite[Theorem~2.3.2]{CLS11Toric}) and we call it the \emph{normal fan of $\mathsf P$}. 
Moreover, we have the following correspondence.
\begin{theorem}[{\cite[Theorem~6.2.1]{CLS11Toric}}]\label{thm_bijective_between_polytopes_and_toric_var}
\[
\begin{tikzcd}
\{ \mathsf{P} \subset \R^n \mid \text{ $\mathsf{P}$ is  a full-dimensional lattice polytope}\} \arrow[<->,d,"1-1"]\\
\{ (X_{\mathsf{P}}, D) \mid  \text{$X_{\mathsf{P}} = X_{\Sigma_{\mathsf{P}}}$ is a projective toric variety, $D$ is a  $(\C^\ast)^n$-invariant ample divisor on $X_{\mathsf{P}}$}\}.
\end{tikzcd}	
\]
\end{theorem}
Furthermore, the torus invariant subvarieties in $X_{\mathsf{P}}$ correspond to faces of the polytope $\mathsf{P}$. For example,
the vertices of~$\mathsf{P}$ correspond to the $T$-fixed points of $X_{\mathsf{P}}$.

For a vertex~$\mathsf v$ of a polytope $\mathsf{P}$, the \emph{degree} $d(\mathsf v)$ of $\mathsf v$ is the number of edges meeting at~$\mathsf v$.
For an $n$-dimensional polytope $\mathsf{P}$, a vertex~$\mathsf{v}$ of~$\mathsf{P}$ is said to be \emph{simple} if $d(\mathsf{v}) = n$. When all the vertices of~$\mathsf{P}$ are simple, we call $\mathsf{P}$ a \emph{simple polytope}.

A vertex $\mathsf{v}$ of a lattice polytope $\mathsf{P}$ is said to be \emph{smooth} if it is simple and the primitive direction vectors of the edges emanating from $\mathsf{v}$ form a basis for $\Z^n$. We call a vertex of~$\mathsf{P}$ \emph{singular} if it is not smooth.  A lattice polytope $\mathsf{P}$ is said to be \emph{smooth} if all the vertices of~$\mathsf{P}$ are smooth. We call a lattice polytope $\mathsf{P}$ is \emph{singular} if some vertex of $\mathsf{P}$ is singular.
See Figure~\ref{fig:def-poly}.
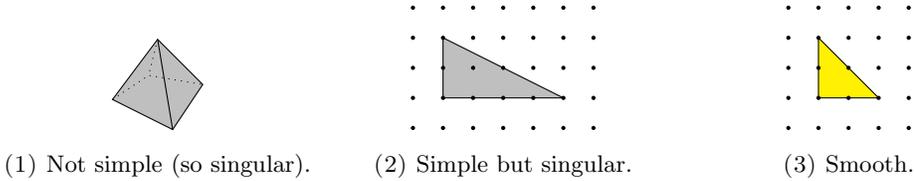
\begin{figure}[h]
    \begin{subfigure}[b]{.3\textwidth}
        \begin{center}
            \begin{tikzpicture}[scale=0.4]
                \filldraw[draw=black,fill=lightgray] (0,0)--(2,-1)--(3,0.5)--(1.5,2)--cycle;
                \draw[dotted](0,0)--(1.2,0.8)--(3,0.5);
                \draw[dotted] (1.5,2)--(1.2,0.8);
                \draw (2,-1)--(1.5,2);
            \end{tikzpicture}
        \end{center}
    \caption{Not simple (so singular).}
        \end{subfigure}
        \begin{subfigure}[b]{.3\textwidth}
        \begin{center}
            \begin{tikzpicture}[scale=0.4]
                \filldraw[draw=black,fill=lightgray] (0,0)--(4,0)--(0,2)--cycle;
                \foreach \x in {-1,0,...,5}
                    \foreach \y in {-1,0,...,3}
                    {
                    \fill (\x,\y) circle (2pt);
                    }
            \end{tikzpicture}
        \end{center}
    \caption{Simple but singular.}
        \end{subfigure}
        \begin{subfigure}[b]{.3\textwidth}
        \begin{center}
            \begin{tikzpicture}[scale=0.4]
                \filldraw[draw=black,fill=yellow] (0,0)--(2,0)--(0,2)--cycle;
                \foreach \x in {-1,0,...,3}
                    \foreach \y in {-1,0,...,3}
                    {
                    \fill (\x,\y) circle (2pt);
                    }
            \end{tikzpicture}
        \end{center}
    \caption{Smooth.}
    \end{subfigure}
    \caption{Examples of singular or smooth lattice polytopes.}
    \label{fig:def-poly}
    \end{figure}

As one may expect, there are geometric interpretations of these terminologies.
\begin{proposition}[{\cite[Theorem~2.4.3]{CLS11Toric}}]
Let $\mathsf{P}$ be a lattice polytope and let $X_{\mathsf{P}}$ be the corresponding toric variety.
A vertex~$\mathsf{v}$ of~$\mathsf{P}$ is smooth if and only if $X_{\mathsf{P}}$ is smooth at the corresponding fixed point. Moreover, $X_{\mathsf{P}}$ is smooth if and only if the polytope $\mathsf{P}$ is smooth. 
\end{proposition}

There is a combinatorial way to determine whether a smooth compact toric variety is Fano.
\begin{definition}\label{def_primitive}
For a fan $\Sigma$, a subset $R$ of the primitive ray vectors is called a \defi{primitive collection} of $\Sigma$ if 
\[
\Cone(R) \notin \Sigma
\quad \text{ but }\quad \Cone(R \setminus \{\mathbf{u}\}) \in \Sigma \quad \text{ for every }\mathbf{u} \in R.
\] 
\end{definition}
Note that primitive collections of $\Sigma_\mathsf{P}$ correspond to the minimal non-faces of $\mathsf{P}$.
For a primitive collection $R = \{\mathbf{u}'_1, \dots,\mathbf{u}'_{\ell}\}$, we get $\mathbf{u}'_1 + \cdots+\mathbf{u}'_{\ell}=\boldsymbol{0}$ or there exists a unique cone $\sigma$ such that $\mathbf{u}'_1 + \cdots+\mathbf{u}'_{\ell}$ is in the interior of $\sigma$. That is, 
\begin{equation}\label{eq:primitive}
	\mathbf{u}'_1 + \cdots+\mathbf{u}'_{\ell}=\begin{cases}
	\boldsymbol{0}, &\text{ or }\\
	a_1 \mathbf{u}_1 + \cdots+ a_{r} \mathbf{u}_{r},&{}
	\end{cases}
\end{equation}
where $\mathbf{u}_1,\dots,\mathbf{u}_{r}$ are the primitive generators of $\sigma$ and $a_1,\dots,a_{r}$ are positive integers.
We call \eqref{eq:primitive} a primitive relation, and the \emph{degree} $\deg R$ of a primitive collection $R$ is defined to be $\ell - (a_1+\cdots+a_r)$.
Batyrev~\cite{Batyrev} gave a criterion for a projective toric variety to be Fano or weak Fano.
\begin{proposition}[{\cite[Proposition~2.3.6]{Batyrev}}]\label{prop:batyrev}
A smooth compact toric  variety $X_\Sigma$ is Fano \textup{(}respectively, weak Fano\textup{)} if and only if $\mathrm{deg}(R)>0$ \textup{(}respectively, $\mathrm{deg}(R)\geq 0$\textup{)} for every primitive collection $R$ of $\Sigma$. 
\end{proposition}

We can also distinguish two smooth Fano toric varieties using the primitive relations.
\begin{proposition}[{\cite[Proposition 2.1.8 and Theorem 2.2.4]{Batyrev}}]\label{prop:batyrev-iso}
Two smooth Fano toric varieties $X_\Sigma$ and $X_{\Sigma'}$ are isomorphic as toric varieties if and only if there is a bijection between the sets of rays of $\Sigma$ and $\Sigma'$ inducing a bijection between maximal cones and preserving the primitive relations.
\end{proposition}

\subsection{Fano Bott manifolds}\label{sec_Fano_Bott}
One of the interesting families of smooth toric varieties is the family of Bott manifolds. In this subsection, we first recall the definition of a Bott manifold, and then we characterize Fano Bott manifolds and classify them.

\begin{definition}[{\cite{GK94Bott}}]
A \emph{Bott tower} is an iterated $\C P^1$-bundle starting with a point:
\[
\begin{tikzcd}[row sep = 0.6em]
B_n \rar & B_{n-1} \rar & \cdots \rar & B_1 \rar & B_0, \\
P(\underline{\C} \oplus \xi_n) \arrow[u, equal]& P(\underline{\C} \oplus \xi_{n-1}) \arrow[u, equal]& & \C P^1 \arrow[u, equal] & 
\{\text{a point}\} \arrow[u, equal]
\end{tikzcd}
\]
where each $B_i$ is the complex projectivization $P(\underline{\C} \oplus \xi_i)$ of the Whitney sum of a 
holomorphic line bundle~$\xi_i$ and the trivial line bundle $\underline{\C}$ 
over $B_{i-1}$. The total space $B_n$ is called a \textit{Bott manifold}. 
\end{definition}

Each Bott manifold is a smooth projective toric variety associated with a smooth lattice polytope combinatorially equivalent to a cube, and the converse also holds.

\begin{proposition}[{\cite[Corollary~3.5]{MasudaPanov08}}]\label{proposition_MasudaPanov}
If a smooth lattice polytope $\mathsf{P}$ is combinatorially equivalent to a cube, then the toric variety $X_{\mathsf{P}}$ is a Bott manifold. Indeed, the family of Bott manifolds is $$\{ X_{\mathsf P} \mid \mathsf P \text{ is a smooth lattice polytope that is combinatorially equivalent to a cube} \}.$$
\end{proposition} 

Let $B_{n}$ be a Bott manifold. Then there is a smooth lattice polytope $\mathsf{P}$ combinatorially equivalent to a cube of dimension $n$ such that $B_n=X_{\mathsf{P}}$.  The polytope $\mathsf{P}$ has $2n$ facets and there are $n$ pairs of facets not intersecting with each other. Let $\{\mathbf{v}_1,\dots,\mathbf{v}_n,\mathbf{w}_1,\dots,\mathbf{w}_n\}$ be the ray generators of the normal fan $\Sigma_P$ of $\mathsf{P}$ such that $\mathbf{v}_i$ and $\mathbf{w}_i$ do not form a cone in $\Sigma_P$ for each $i\in [n]$. That is, each primitive collection of $\Sigma_\mathsf{P}$ corresponds to the set $\{\mathbf{v}_i,\mathbf{w}_i\}$. By Proposition~\ref{prop:batyrev}, we obtain
\begin{equation}\label{eq:Bott-Fano}
\text{$B_{n}$ is Fano if and only if $\mathbf{v}_i+\mathbf{w}_i$ is either $\mathbf{v}_j$ or $\mathbf{w}_j$ for some $j$ unless $\mathbf{v}_i+\mathbf{w}_i=\boldsymbol{0}$. }
\end{equation}
Using \eqref{eq:Bott-Fano}, to a Fano Bott manifold $B_{n}$ we associate a signed rooted forest with the vertex set~$[n]$ as follows:
\begin{itemize}
\item vertex $i$ is a root if $\mathbf{v}_i+\mathbf{w}_i=\boldsymbol{0}$,
\item we draw an edge with $+$ sign between $i$ and $j$ if $\mathbf{v}_i+\mathbf{w}_i=\mathbf{v}_j$, and
\item we draw an edge with $-$ sign  between $i$ and $j$  if $\mathbf{v}_i+\mathbf{w}_i=\mathbf{w}_j$.
\end{itemize}
We can also construct a Fano Bott manifold from a signed rooted forest up to isomorphism. 

Let $F$ be a signed rooted forest with vertex set $[n]$. For each $i$, by changing the signs of all edges connecting $i$ and its children simultaneously, we get a new signed rooted forest $r_i(F)$. Then the Bott manifold $B_{r_i(F)}$ corresponding to $r_i(F)$ is isomorphic to the Bott manifold $B_F$ corresponding to $F$ by Proposition~\ref{prop:batyrev-iso}.
Let $\mathcal{SF}_n$ be the set of all signed rooted forests on the vertex set $[n]$. Denote 
by $\sim$ the equivalence relation on $\mathcal{SF}_n$ generated by $r_i$'s 
for all $i$'s.

\begin{theorem}[{\cite[Theorem 3.2]{CLMP}}]\label{thm_isom_FB}
The isomorphism classes in Fano Bott manifolds of complex dimension $n$ 
bijectively correspond to $\mathcal{SF}_n/\!\!\sim$.
\end{theorem}

We provide all signed rooted forests having three vertices in 
Figure~\ref{fig_SF_3}. Among these ten signed rooted forests, there are four equivalence classes in $\mathcal{SF}_3/\!\!\sim$ as follows:
\[
(1)\sim (2)\sim (3)\sim (4),\quad (5)\sim (7),\quad (6),\quad (8)\sim (9),\quad (10) 
\]
In the above, the first and the third equivalence classes arise as toric Richardson varieties of Catalan type, and the last two arise as products of Catalan type. However, the second class does not arise as (products of) Catalan type.   
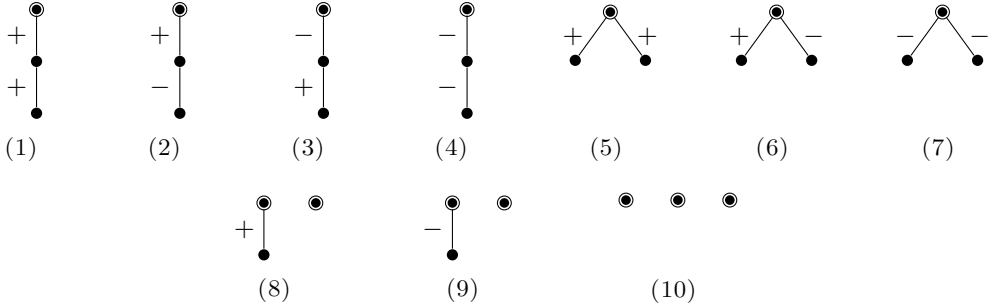
\begin{figure}[h]
\begin{subfigure}[b]{0.12\textwidth}
\centering
\begin{tikzpicture}[ node distance = 1.5em]

\node[root, vertex] (1) {};
\node[vertex, below = of 1] (2) {};
\node[vertex, below = of 2] (3) {};
\draw (1)--(2) node[midway, left] {$+$};
\draw (2)--(3) node[midway, left] {$+$};
\end{tikzpicture}
\caption{}
\label{fig_signed_rooted_tree_n3_1}
\end{subfigure}
\begin{subfigure}[b]{0.12\textwidth}
\centering
\begin{tikzpicture}[ node distance = 1.5em]
\node[root, vertex] (1) {};
\node[vertex, below = of 1] (2) {};
\node[vertex, below = of 2] (3) {};
\draw (1)--(2) node[midway, left] {$+$};
\draw (2)--(3) node[midway, left] {$-$};
\end{tikzpicture}
\caption{}
\label{fig_signed_rooted_tree_n3_2}
\end{subfigure}
\begin{subfigure}[b]{0.12\textwidth}
\centering
\begin{tikzpicture}[ node distance = 1.5em]
\node[root, vertex] (1) {};
\node[vertex, below = of 1] (2) {};
\node[vertex, below = of 2] (3) {};
\draw (1)--(2) node[midway, left] {$-$};
\draw (2)--(3) node[midway, left] {$+$};
\end{tikzpicture}
\caption{}
\label{fig_signed_rooted_tree_n3_3}
\end{subfigure}
\begin{subfigure}[b]{0.12\textwidth}
\centering
\begin{tikzpicture}[ node distance = 1.5em]
\node[root, vertex] (1) {};
\node[vertex, below = of 1] (2) {};
\node[vertex, below = of 2] (3) {};
\draw (1)--(2) node[midway, left] {$-$};
\draw (2)--(3) node[midway, left] {$-$};
\end{tikzpicture}
\caption{}
\label{fig_signed_rooted_tree_n3_4}
\end{subfigure}
\begin{subfigure}[b]{0.14\textwidth}
\centering
\raisebox{2em}{ \begin{tikzpicture}[node distance = 1.5em and 1 em]
\node[root, vertex] (1) {};
\node[vertex, below left = of 1] (2) {};
\node[vertex, below right = of 1] (3) {};
\draw (1)--(2) node[midway, left] {$+$};
\draw (1)--(3) node[midway, right] {$+$};
\end{tikzpicture}}
\caption{}
\label{fig_signed_rooted_tree_n3_5}
\end{subfigure}
\begin{subfigure}[b]{0.14\textwidth}
\centering
\raisebox{2em}{ \begin{tikzpicture}[node distance = 1.5em and 1 em]
\node[root, vertex] (1) {};
\node[vertex, below left = of 1] (2) {};
\node[vertex, below right = of 1] (3) {};
\draw (1)--(2) node[midway, left] {$+$};
\draw (1)--(3) node[midway, right] {$-$};
\end{tikzpicture}}
\caption{}
\label{fig_signed_rooted_tree_n3_6}
\end{subfigure} 
\begin{subfigure}[b]{0.14\textwidth}
\centering
\raisebox{2em}{ \begin{tikzpicture}[node distance = 1.5em and 1 em]
\node[root, vertex] (1) {};
\node[vertex, below left = of 1] (2) {};
\node[vertex, below right = of 1] (3) {};
\draw (1)--(2) node[midway, left] {$-$};
\draw (1)--(3) node[midway, right] {$-$};
\end{tikzpicture}}
\caption{}
\label{fig_signed_rooted_tree_n3_7}
\end{subfigure} 

\vspace{1em}

\begin{subfigure}[b]{0.16\textwidth}
\centering
\begin{tikzpicture}[ node distance = 1.5em]
\node[root, vertex] (1) {};
\node[vertex, below = of 1] (2) {};
\node[root,vertex, right = of 1] (3) {};
\draw (1)--(2) node[midway, left] {$+$};
\end{tikzpicture}
\caption{}
\label{fig_signed_rooted_tree_n3_8}
\end{subfigure} 
\begin{subfigure}[b]{0.16\textwidth}
\centering
\begin{tikzpicture}[ node distance = 1.5em]
\node[root, vertex] (1) {};
\node[vertex, below = of 1] (2) {};
\node[root, vertex, right = of 1] (3) {};

\draw (1)--(2) node[midway, left] {$-$};
\end{tikzpicture}
\caption{}
\label{fig_signed_rooted_tree_n3_9}
\end{subfigure}
\begin{subfigure}[b]{0.2\textwidth}
\centering
\raisebox{2em}{ \begin{tikzpicture}[ node distance = 1.5em]
\node[root, vertex] (1) {};
\node[root, vertex, right = of 1] (2) {};
\node[root, vertex, right = of 2] (3) {};
\end{tikzpicture}}
\caption{}
\label{fig_signed_rooted_tree_n3_10}
\end{subfigure}
\caption{Signed rooted forests with $3$ 
vertices}\label{fig_SF_3}
\end{figure}

Similarly, there are thirteen equivalence classes in $\mathcal{SF}_4/\!\!\sim$ as shown in Figure~\ref{fig_SF4}, where $+$ signs on edges are omitted. In the figure, (7), (9), (11) arise as toric Richardson varieties of Catalan type, and (1), (2), (3), (4), (6)  arise as products of Catalan type.  However, the remaining ones do not arise as (products of) Catalan type.   

\begin{figure}[H]
\begin{subfigure}[b]{0.2\textwidth}
\centering
\raisebox{1.5cm}{
\begin{tikzpicture}[node distance = 1.5em, 
baseline={(1.base)}]
\node[root, vertex] (1) {};
\node[root, vertex, right = of 1] (2) {};
\node[root, vertex, right = of 2] (3) {};
\node[root, vertex, right = of 3] (4) {};
\end{tikzpicture}}
\caption{}
\end{subfigure}
\begin{subfigure}[b]{0.2\textwidth}
\centering
\raisebox{1.5cm}{
\begin{tikzpicture}[node distance = 1.5em, baseline={(1.base)}]
\node[root, vertex] (1) {};
\node[vertex, below = of 1] (2) {};
\node[root, vertex, right = of 1] (3) {};
\node[root, vertex, right = of 3] (4) {};
\draw (1)--(2);
\end{tikzpicture}}
\caption{}
\end{subfigure}
\begin{subfigure}[b]{0.2\textwidth}
\centering
\raisebox{1.5cm}{
\begin{tikzpicture}[node distance = 1.5em, baseline={(1.base)}]
\node[root, vertex] (1) {};
\node[vertex, below = of 1] (2) {};
\node[vertex, below = of 2] (3) {};
\node[root, vertex, right = of 1] (4) {};

\draw (1)--(2)--(3);
\end{tikzpicture}}
\caption{}
\end{subfigure}
\begin{subfigure}[b]{0.2\textwidth}
\centering
\raisebox{1.5cm}{
\begin{tikzpicture}[node distance = 1.5em, baseline={(1.base)}]
\node[root, vertex] (1) {};
\node[vertex, below = of 1] (2) {};
\node[root, vertex, right = of 1] (3) {};
\node[vertex, below = of 3] (4) {};

\draw (1)--(2)
(3)--(4);
\end{tikzpicture}}
\caption{}
\end{subfigure} \\ \vspace{1em}
\begin{subfigure}[b]{0.2\textwidth}
\centering
\raisebox{2cm}{
\begin{tikzpicture}[node distance = 1.5em and 1 em, baseline={(1.base)}]
\node[root, vertex] (1) {};
\node[vertex, below left = of 1] (2) {};
\node[vertex, below right = of 1] (3) {};
\node[root, vertex, right = 2.5em of 1] (4) {};

\draw (1)--(2)
(1)--(3);
\end{tikzpicture}}
\caption{}
\end{subfigure}
\begin{subfigure}[b]{0.2\textwidth}
\centering
\raisebox{2cm}{
\begin{tikzpicture}[node distance = 1.5em and 1 em, baseline={(1.base)}]
\node[root, vertex] (1) {};
\node[vertex, below left = of 1] (2) {};
\node[vertex, below right = of 1] (3) {};
\node[root, vertex, right = 2.5em of 1] (4) {};

\draw (1)--(2);
\draw (1)--(3) node[midway, right] {$-$};
\end{tikzpicture}}
\caption{}
\end{subfigure}
\begin{subfigure}[b]{0.18\textwidth}
\centering
\raisebox{2cm}{
\begin{tikzpicture}[node distance = 1.5em, baseline={(1.base)}]
\node[root, vertex] (1) {};
\node[vertex, below = of 1] (2) {};
\node[vertex, below = of 2] (3) {};
\node[vertex, below = of 3] (4) {};

\draw (1)--(2)--(3)--(4);
\end{tikzpicture}}
\caption{}
\end{subfigure}
\begin{subfigure}[b]{0.18\textwidth}
\centering
\raisebox{2cm}{
\begin{tikzpicture}[node distance = 1.5em and 1 em, baseline={(1.base)}]
\node[root, vertex] (1) {};
\node[vertex, below left = of 1] (2) {};
\node[vertex, below right = of 1] (3) {};
\node[vertex, below = of 2] (4) {};

\draw (1)--(2)--(4)
(1)--(3);
\end{tikzpicture}}
\caption{}
\end{subfigure}
\begin{subfigure}[b]{0.2\textwidth}
\centering
\raisebox{2cm}{
\begin{tikzpicture}[node distance = 1.5em and 1 em, baseline={(1.base)}]
\node[root, vertex] (1) {};
\node[vertex, below left = of 1] (2) {};
\node[vertex, below right = of 1] (3) {};
\node[vertex, below = of 2] (4) {};

\draw (1)--(2)--(4);
\draw (1)--(3) node[midway, right] {$-$};
\end{tikzpicture}}
\caption{}
\end{subfigure} \\ \vspace{1em}
\begin{subfigure}[b]{0.2\textwidth}
\centering
\raisebox{1.5cm}{
\begin{tikzpicture}[node distance = 1.5em and 1 em, baseline={(1.base)}]
\node[root, vertex] (1) {};
\node[vertex, below = of 1] (2) {};
\node[vertex, below left = of 2] (3) {};
\node[vertex, below right = of 2] (4) {};

\draw (1)--(2)--(3);
\draw (2)--(4);
\end{tikzpicture}}
\caption{}
\end{subfigure}
\begin{subfigure}[b]{0.2\textwidth}
\centering
\raisebox{1.5cm}{
\begin{tikzpicture}[node distance = 1.5em and 1 em, baseline={(1.base)}]
\node[root, vertex] (1) {};
\node[vertex, below = of 1] (2) {};
\node[vertex, below left = of 2] (3) {};
\node[vertex, below right = of 2] (4) {};

\draw (1)--(2)--(3);
\draw (2)--(4) node[midway, right] {$-$};
\end{tikzpicture}}
\caption{}
\end{subfigure}
\begin{subfigure}[b]{0.2\textwidth}
\centering
\raisebox{1.5cm}{
\begin{tikzpicture}[node distance = 1.5em and 1 em, baseline={(1.base)}]
\node[root, vertex] (1) {};
\node[vertex] (2) [below = of 1] {};
\node[vertex, below left = of 1] (3) {};
\node[vertex] (4) [below right = of 1] {};
\draw (1)--(2)
(1)--(3)
(1)--(4);
\end{tikzpicture}}
\caption{}
\end{subfigure}
\begin{subfigure}[b]{0.2\textwidth}
\centering
\raisebox{1.5cm}{
\begin{tikzpicture}[node distance = 1.5em and 1 em, baseline={(1.base)}]
\node[root, vertex] (1) {};
\node[vertex] (2) [below = of 1] {};
\node[vertex, below left = of 1] (3) {};
\node[vertex] (4) [below right = of 1] {};
\draw (1)--(2)
(1)--(3)
(1)--(4) node[midway, right] {$-$};
\end{tikzpicture}}
\caption{}
\end{subfigure}
\caption{Representatives of $\mathcal{SF}_4/\!\!\sim$.}
\label{fig_SF4}
\end{figure}
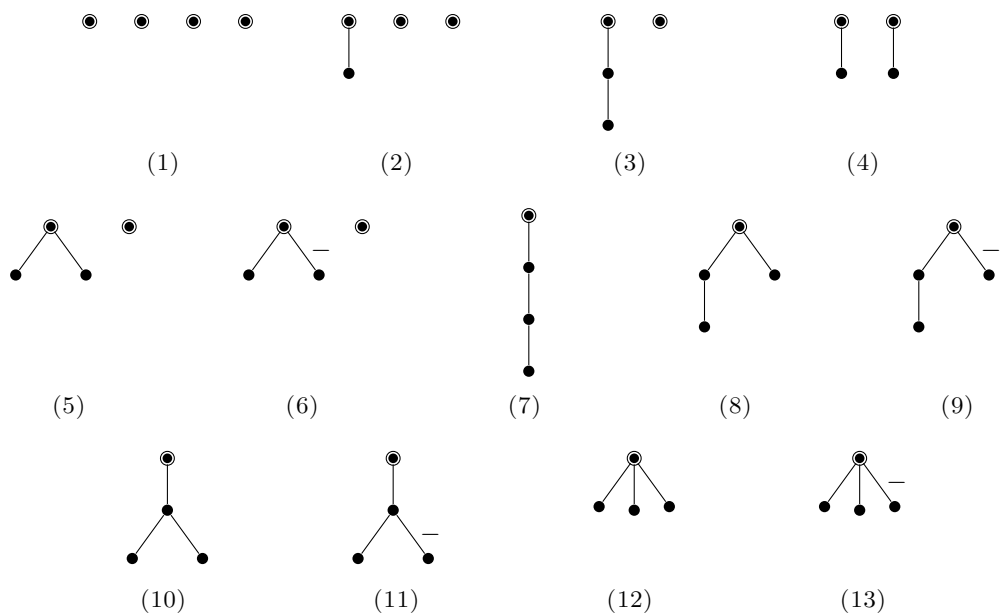

\end{document}